\documentclass[aap]{imsart}

\RequirePackage{amsthm,amsmath,amsfonts,amssymb}
\RequirePackage[numbers]{natbib}
\RequirePackage[colorlinks,citecolor=blue,urlcolor=blue]{hyperref}
\RequirePackage{graphicx}

\usepackage{amsmath}
\usepackage{amssymb,amsthm,amsfonts,amsbsy,latexsym,amsxtra}
\usepackage{bm}
\usepackage{mathtools}
\usepackage[cal=pxtx,scr=boondoxo]{mathalfa}
\usepackage{graphicx}
\usepackage{grffile}
\usepackage{appendix}
\usepackage{url}
\usepackage{subcaption}
\usepackage{float}
\usepackage{caption}
\usepackage[utf8]{inputenc}

\usepackage{dsfont}
\usepackage[final]{showkeys}
\usepackage{extpfeil}
\usepackage[subnum]{cases}
\usepackage{hyperref}
\graphicspath{{images/}}

\usepackage{enumerate}

\theoremstyle{plain}
\newtheorem{theorem}{Theorem}[section]
\newtheorem{lemma}[theorem]{Lemma}
\newtheorem{proposition}[theorem]{Proposition}
\newtheorem{corollary}[theorem]{Corollary}
\newtheorem{claim}[theorem]{Claim}

\newtheorem{assumption}[theorem]{Assumption}

\newtheorem{observation}[theorem]{Observation}
\theoremstyle{remark}

\newtheorem{remarky}[theorem]{Remark}

\newtheorem{definition}[theorem]{Definition}

\DeclareMathOperator*{\argmax}{arg\,max}
\DeclareMathOperator*{\argmin}{arg\,min}

\renewcommand{\P}{\mathbb{P}}
\newcommand{\E}{\mathbb{E}}

\def \toinp    {\buildrel {\P}\over{\longrightarrow}}

\DeclarePairedDelimiter\floor{\lfloor}{\rfloor}
\DeclarePairedDelimiter\ceil{\lceil}{\rceil}

\DeclareRobustCommand*{\ora}{\overrightarrow}
\DeclareRobustCommand*{\ola}{\overleftarrow}

\newcommand{\be}{\begin{equation}}
\newcommand{\ee}{\end{equation}}

\newcommand{\CA}{\mathcal{A}}
\newcommand{\CB}{\mathcal{B}}

\newcommand{\CE}{\mathcal{E}}

\newcommand{\CL}{\mathcal{L}}

\newcommand{\CP}{\mathcal{P}}
\newcommand{\CQ}{\mathcal{Q}}

\newcommand{\CV}{\mathcal{V}}
\newcommand{\CW}{\mathcal{W}}

\newcommand{\rd}{\mathrm{d}}
\newcommand{\re}{\mathrm{e}}

\newcommand{\ind}[1]{\mathds{1}_{\left\{{#1}\right\}}}

\usepackage{accents}
\newcommand{\ubar}[1]{\underaccent{\bar}{#1}}

\newcommand{\vareps}{\varepsilon}

\newcommand{\sss}[1]{{\scriptscriptstyle #1}}

\newcommand{\goodpathto}[2]{{\overset{[#1,#2)}{\leadsto}}}

\numberwithin{equation}{section}
\begin{document}
\begin{frontmatter}
 \title{Distance evolutions in growing preferential attachment graphs}
 \runtitle{Distance evolutions in PAMs}

 \begin{aug}
  \author[A]{\fnms{Joost} \snm{Jorritsma}\ead[label=e1, mark]{j.jorritsma@tue.nl}}
  \author[B]{\fnms{J\'ulia} \snm{Komj\'athy}\ead[label=e2] {J.Komjathy@tudelft.nl}}
  \address[A]{Department of Mathematics and Computer Science, Eindhoven University of Technology, \printead{e1}}
  \address[B]{Delft Institute of Applied Mathematics, Delft  University of Technology, \printead{e2}}
 \end{aug}

 \date{\today}

 \begin{abstract}
  We study the evolution of the graph distance and weighted distance between two fixed vertices in dynamically growing random graph models. More precisely, we consider preferential attachment models with power-law exponent $\tau\in(2,3)$, sample two vertices $u_t,v_t$ uniformly at random when the graph has $t$ vertices, and study the evolution of the graph distance between these two fixed vertices as the surrounding graph grows. This yields a discrete-time stochastic process in $t'\geq t$, called the distance evolution.
  We show that there is a tight strip around the function $4\frac{\log\log(t)-\log(\log(t'/t)\vee1)}{|\log(\tau-2)|}\vee 2$ that the distance evolution never leaves with high probability as $t$ tends to infinity. We extend our results to weighted distances, where every edge is equipped with an i.i.d.\ copy of a non-negative random variable $L$.
 \end{abstract}
\end{frontmatter}
\section{Introduction}
In 1999, Faloutsos, Faloutsos, and Faloutsos studied the topology of the early Internet network, discovering power-laws in the degree distribution and short average hopcounts between routers \cite{faloutsos1999internet}. Undoubtedly, the Internet has grown explosively in the last two decades. It would be interesting to investigate what has happened to the graph structure surrounding the early routers (or their direct replacements) that were already there in 1999, ever since.
Natural questions about the evolving graph surrounding these early routers are:
\begin{itemize}
 \item How did the number of connections of the routers gradually change? Did the early routers become important hubs in the network?
 \item Can we quantify the number of hops needed to connect two early routers? Particularly, did the hopcount decrease or increase while their importance in the network changed, and more and more connections arrived? If so, how did the distance gradually evolve?
\end{itemize}
These kinds of questions drive the mathematics in the present paper.
We initiate a research line that studies how certain graph properties defined on a fixed set of vertices evolve as the surrounding graph  grows.
We consider the weighted-distance evolution  in two classical \emph{preferential attachment models} (PAMs).
Studying the evolution of a property on fixed vertices may sound as a natural mathematical question. Yet, only the evolution of the degree of fixed vertices has been addressed so far in the PAM literature \cite{dereich2009random,mori2002}.

A realization of a classical preferential attachment graph can be constructed according to an iterative procedure. One starts with an initial graph $\mathrm{PA}_1=(V_1,E_1)$ on the vertex set $V_1=\{1\}$ and edge set $E_1=\varnothing$, after which vertices arrive sequentially at deterministic times $t\in\{2,3,\dots\}$. We denote the graph at time $t$ by $\mathrm{PA}_t$ and label all the vertices by their arrival time, also called \emph{birth time}.
The arriving vertex $t$ connects to present vertices such that it is more likely to connect to vertices with a high degree at time $t$. Let $\P(\{t\to v\}\mid \mathrm{PA}_{t-1})$ denote the probability that $t$ connects to $v<t$.
We consider two classical (non-spatial) variants of the model, the so-called $(m,\delta)$-model based on \cite{berger2005spread,bollobas2004diameterpa} and the independent connection model \cite{dereich2009random}. They are formally defined in Section \ref{sec:models-examples}. They both assume that there exists $\tau\in(2,3)$ such that
\be
\P(\{t\to v\}\mid \mathrm{PA}_{t-1})  = \frac{D_v(t-1)}{t(\tau-1)} + O(1/t), \label{eq:intro-p-conn}
\ee
where $D_v(t-1)$ denotes the degree of vertex $v$ directly after the arrival of vertex $t-1$. As a result, the asymptotic degree distribution has a power-law decay with exponent $\tau$ \cite{dereich2009random,hofstad2016book1}, that we therefore call the power-law exponent.

The \emph{graph-distance evolution} is a discrete-time stochastic process that we denote by  $\big(d_G^{(t')}(u_t,v_t)\big)_{t'\geq t}$ and  define formally in Definition \ref{def:distances} below.
Here, $u_t$ and $v_t$ are two \emph{typical vertices}, i.e., they are sampled uniformly at random from the vertices in $\mathrm{PA}_t$. The graph distance $d_G^{(t')}(u_t,v_t)$ is the number of edges on the shortest path between $u_t$ and $v_t$ that uses only vertices that arrived at latest at time $t'$.
The distance evolution $\big(d_G^{(t')}(u_t,v_t)\big)_{t'\geq t}$ is nonincreasing in $t'$, since new edges arrive in the graph that may form a shorter path between $u_t$ and $v_t$.
We will now state our main result for the graph-distance evolution.
To describe the graph distance we define for $t'\geq t$, writing $a\vee b:=\max\{a,b\}$,
\begin{equation}
 K_{t,t'} = 2 \left\lfloor \frac{\log\log(t) - \log\big(\log(t'/t)\vee 1\big)}{|\log(\tau-2)|}\right\rfloor\vee 1. \label{eq:kt}
\end{equation}

\begin{theorem}[Graph-distance evolution]\label{cor:graph-evolution}
 Consider the preferential attachment model with power-law exponent $\tau\in(2,3)$. Let $u_t, v_t$ be two typical vertices in $\mathrm{PA}_t$. Then
 \begin{equation}
  \bigg(\sup_{t'\geq t} \left|d_G^{(t')}(u_t,v_t) - 2K_{t,t'}\right|\bigg)_{t\geq 1} \label{eq:graph-evolution}
 \end{equation}
 is a tight sequence of random variables.
\end{theorem}
Here, a sequence of random variables $(X_n)_{n\geq 1}$ is called tight if $\lim_{M\to\infty}\sup_n\P(|X_n|>M) = 0$.
Theorem \ref{cor:graph-evolution} tracks the evolution of $d_G^{(t')}(u_t,v_t)$ as time passes and the graph around $u_t$ and $v_t$ grows, since in \eqref{eq:graph-evolution} the supremum  is taken  over $t'$.  Below, in Theorem \ref{thm:weighted-evolution}, we extend Theorem \ref{cor:graph-evolution} to a general setting and consider the so-called \emph{weighted-distance evolution} $\big(d_L^{(t')}(u_t,v_t)\big)_{t'\geq t}$. There, we equip every edge in the graph with a weight, an i.i.d.\ copy of a random variable $L$.  We consider the evolution of the weighted distance, the sum of the weights along the least-weighted path from $u_t$ to $v_t$ that is present at time $t'$. We obtain results for \emph{any} non-negative random variable $L$ that serves as edge-weight distribution.

As a consequence of Theorem \ref{cor:graph-evolution}, we  obtain a hydrodynamic limit, i.e., a scaled version of the distance evolution converges under proper time scaling uniformly in probability to a non-trivial deterministic function.
\begin{corollary}[Hydrodynamic limit for the graph-distance evolution]\label{cor:hydro-graph}
 Consider the preferential attachment model with power-law exponent $\tau\in(2,3)$. Define $T_t(a):=t\exp(\varepsilon\log^a(t))$ for $a\geq0$ and arbitrary $\varepsilon>0$. Let $u_t, v_t$ be two typical vertices in $\mathrm{PA}_t$. Then
 \be
 \sup_{a\geq0}\left| \frac{d_G^{(T_t(a))}(u_t,v_t)}{\log\log(t)}  - (1-\min\{a,1\})\frac{4}{|\log(\tau-2)|}\right| \toinp 0, \qquad \text{as }t\to\infty.\label{eq:intro-hydro-graph}
 \ee
\end{corollary}
This can be verified by computing the value of $K_{t,T_t(a)}$ using \eqref{eq:kt}, substituting this value into \eqref{eq:graph-evolution}, and then dividing all terms by $\log\log(t)$.

Observe that in Corollary \ref{cor:hydro-graph} all $\log\log(t)$-terms have vanished when $a=1$.
The following consequence of Corollary \ref{cor:graph-evolution} illustrates the rate at which smaller order terms appear and vanish. In particular, the graph distance is of constant order as soon as $t'/t$ is of polynomial order in $t$.
\begin{corollary}[Lower-order terms]
 Consider the preferential attachment model with power-law exponent $\tau\in(2,3)$. Let $u_t, v_t$ be two typical vertices in $\mathrm{PA}_t$.
 Let $g(t)$ be any function that is bounded from above by $2K_{t,t}$, and set
 $T_g(t):=t^{1+(\tau-2)^{-g(t)/4}}$. Then, for two typical vertices $u_t$ and $v_t$ in $\mathrm{PA}_t$,
 \[
  \Big(d_G^{(T_g(t))}(u_t,v_t) - g(t)\Big)_{t\geq 1}
 \]
 is a tight sequence of random variables.
\end{corollary}
Indeed, setting any $g(t)$ that tends to infinity with $t$ results in a time scale $T_g(t)\sim t^{1+\varepsilon_{g(t)}}$ for some $\varepsilon_{g(t)}\to 0$ as $t$ tends to infinity.
\begin{remarky}\label{remark:graph-distance-2}
 Using a similar martingale argument as in Lemma \ref{lemma:degree-lowerbound} below for the degree of the vertices $u_t$ and $v_t$, one can show that when $t'=\Theta(t^{2/(3-\tau)})$, there will be a vertex that connects to both $u_t$ and $v_t$. Hence, the distance evolution settles on two.
\end{remarky}
\subsection{Literature perspectives on PAMs}
\subsubsection{Snapshot analysis}\label{sec:lit-snapshot}
The two models studied in this paper are the most commonly used pure PAMs in the literature, i.e., in these models it is solely the preferential attachment mechanism that drives the changes in the graph topology. These PAMs are mathematically defined by Bollob\'as and Riordan \cite{bollobas2004diameterpa}, and Dereich and M\"orters \cite{dereich2009random}.
For an overview of rigorous results and references we refer to \cite{hofstad2016book1}, but also to recent works on these models \cite{berger2014asymptotic, caravenna2016diameter,dereich2012typical,dereich2009random, dereich2013random,dommers2010diameters,jorkom2019weighted}.
Since the original PAM, many variants with more involved dynamics and connection functions have been introduced. In \cite{pittel2010random, janson2019preferential}, the vertex set is fixed and only edges are formed dynamically. The variations introduced in \cite{alves2019preferentialedgestep,deijfen2009growing} allow for edges being formed (or deleted in \cite{cooper2004random, deijfen2009growing}) between existing vertices. Refs.\ \cite{dereich2009random,dereich2013random} consider a version where the attachment function can be sublinear in the degree. In \cite{cipriani2019dynamical,dereichMaillerMortersCondens2017,dereichOrtgieseCondens2014,freeman2018extensive} vertices are equipped with a fitness and in \cite{malyshkin2014powerofchoic} the arriving vertices have a power of choice. Spatial variants where vertices have a location in an underlying Euclidean space are studied in \cite{aiello2008spatial,jacob2015spatial,jacob2017robustness}. Here, closeness in Euclidean distance is combined with preferential attachment. The age-dependent random connection model \cite{gracar2018age,gracarluchtrath2020robustness} is a recent spatial version. There the connection probabilities are not governed by the degree of vertices, but by their relative age compared to the arriving vertex.
In these papers, several graph properties have been studied in the large network limit, i.e., as the number of vertices $t$ tends to infinity. Stochastic processes on PAMs have been analysed in \cite{berger2005spread,can2015metastability} for the contact process and in \cite{amin2018prefBootstrap} for bootstrap percolation.

The above mentioned results and papers provide statements about static snapshots of the graph $\mathrm{PA}_t$ in the large network limit: the network is considered at a \emph{single time} $t$ as $t$ tends to infinity.
This snapshot analysis allows for comparison to (simpler) static random graph models, such as the configuration model \cite{bollobas1980countingregulargraphs,molloyreed1995randomgraphswithgivendegree}, Chung-Lu model \cite{chung2002average}, and the Norros-Reittu model \cite{norros2006conditionally}, and strikes to classify properties of random graphs as either \emph{universal} or \emph{model-dependent}.
See \cite{hofstad2016book1, hofstad2016book2} and its references for universal properties.
Due to the snapshot analysis, temporal changes of the graph that are reflected in the statements of Theorem \ref{thm:weighted-evolution}, are absent in earlier works for graph properties other than the degree of fixed vertices \cite{dereich2009random,mori2002}.

\subsubsection{Future directions: evolving properties}
This paper commences a research line by studying an evolving graph property (other than the degree of fixed vertices  \cite{dereich2009random,mori2002}).
Statements involving the evolution of a property describe the structure of the graph during a \emph{time interval}, rather than at a \emph{single time}. We consider the distance evolution in two classical preferential attachment models.
This requires a more fine-grained control of the \emph{entire graph} than the degree evolution of a \emph{fixed vertex}, and also yields more insight in the evolution of the structure of the graph.
One of the main reasons to consider distances for these classical PAMs is that they display a notable change over time. The growth terms decrease from $\log\log$-order to constant order as the graph grows. This is in contrary to, for example, the local clustering coefficient, a graph property related to the number of triangles which a typical vertex is a member of. The local clustering coefficient of a typical vertex is  of constant order and tends to zero for typical vertices due to the locally tree-like structure in classical PAMs.

A natural extension of the present paper would be to study the distance evolution in PAMs where the asymptotic degree distribution has finite variance. For this regime, it is known that the \emph{static} typical graph distance is of order $\Theta(\log(t))$, but the precise constant has not been determined.
We expect that in this regime the time-scaling of the growth is different from the scaling of the hydrodynamic limit in  Corollary \ref{cor:hydro-graph} and to see the distance drop by a constant factor when $t'/t$ is of polynomial order, rather than stretched exponential in the logarithm.

Distances in spatial preferential attachment (SPA) are studied in \cite{hirsch2018distances_spatial_pa} for the regime where $\tau\in(2,3)$: \cite{hirsch2018distances_spatial_pa} proves an upper bound using a similar \emph{two-connector} procedure that we also use here. The lower bound for distances in SPA for $\tau\in(2,3)$ and asymptotic results for other parameter regimes remain interesting open problems.

In most PAMs the graph and its edge set are increasing over time. In \cite{cooper2004random, deijfen2009growing} variations of PAMs are introduced where edges can be deleted. As a result the distance evolution is no longer monotone and other behaviour may be expected.

The variations of PAMs mentioned in Section \ref{sec:lit-snapshot} all have properties that can be considered from a non-static perspective.
For instance, one could analyse the local clustering coefficient in versions of PAMs that are not locally tree-like.
Static analysis of the local clustering coefficient on spatial variants of PAMs have been done in \cite{gracar2018age, jacob2015spatial}.
Some frequently studied global properties  are the size of the giant component and its robustness against site or edge percolation \cite{dereich2013random,eckhoff2014vulnerability,gracarluchtrath2020robustness,jacob2017robustness}, and condensation phenomena \cite{bianconi2001bose,cipriani2019dynamical, dereichMaillerMortersCondens2017,freeman2018extensive,malyshkin2014powerofchoic}.

\subsection{Methodology}
The proof of Theorem \ref{cor:graph-evolution} and Theorem \ref{thm:weighted-evolution} below consist of a lower bound and an upper bound.
For the upper bound we prove that at all times $t'\geq t$ there is a path from $u_t$ to $v_t$ that has length at most $2K_{t,t'} + M_G$ (from \eqref{eq:kt}) for some constant $M_G$ and contains only vertices born, i.e., arrived, before time $t'$. We first heuristically argue that the scaling for the graph distance in \eqref{eq:intro-hydro-graph} is a natural scaling. After that, we turn to the difficulties that arise in handling the dynamics.
The degree $D_{q_t}(t')$ of a vertex $q_t\in\{u_t,v_t\}$ at time $t'$ is of order $(t'/q_t)^{1/(\tau-1)}$.
Writing $t'=T_t(a):=t\exp(\log^a(t))$ and approximating the birth time of the uniform vertex $q_t$ by $t$, we have that
\[
 D_{q_t}\big(T_t(a)\big)\approx \exp\big(\log^a(t)/(\tau-1)\big)=:s^\sss{(a)}.
\]
Generally, a vertex of degree $s$ is at graph distance two from many vertices that have degree approximately $s^{1/(\tau-2)}$. This allows for an iterative \emph{two-connector procedure} that starts from an initial vertex with degree at least $s^\sss{(a)}$ and reaches in the $k$-th iteration a vertex with degree approximately ${(s^\sss{(a)})}^{(\tau-2)^{-k}}$. We call $k\mapsto (s^\sss{(a)})^{(\tau-2)^{-k}}$ the degree-threshold sequence. At each iteration, we greedily extend the path by two edges, arriving to such a higher-degree vertex. In the edge-weighted version, these two edges are chosen to minimize the total edge-weight among all such two edges. This two-connector procedure to vertices with increasing degree is iterated until the  well-connected inner core is reached. The inner core is the set of vertices with degree roughly $T_t(a)^{1/(2(\tau-1))}$ at time $T_t(a)$. Hence, for $a<1$, the total number of iterations to reach the inner core is approximately
\be
\min\Big\{k: (s^\sss{(a)})^{(\tau-2)^{-k}}\geq T_t(a)^{1/(2(\tau-1))}\Big\} \approx (1-a)\frac{\log\log(t)}{|\log(\tau-2)|}.\label{eq:meth-upper-kt}
\ee
By construction, the graph distance from $u_t$ and $v_t$ to the inner core is two times the right-hand side (rhs) in \eqref{eq:meth-upper-kt}.
The graph and weighted distance between vertices in the inner core are negligible, yielding the scaling in \eqref{eq:intro-hydro-graph}, as well as the upper bounds in Theorems \ref{cor:graph-evolution} and \ref{thm:weighted-evolution}.
\begin{figure}[t]
 \centering
 \includegraphics[width=0.93\textwidth]{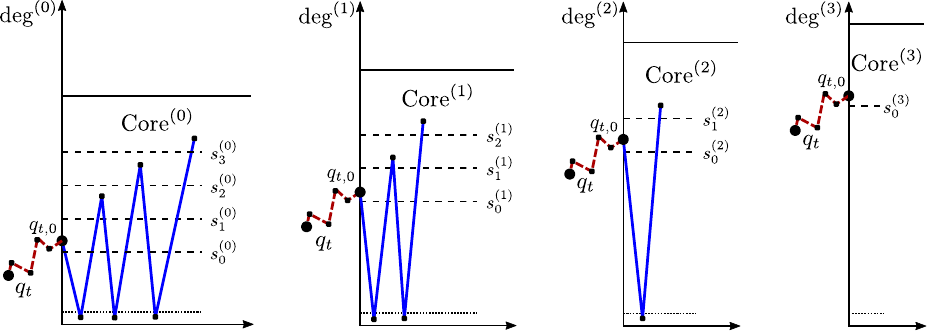}
 \caption{Construction of a path from $q_t$ to the inner core at the times $(t_i)_{i\leq 3}$, where $t_0=t$.
  The $y$-axis represents the degree of vertices at time $t_i$ and the connected dots the vertices on the path segments from $q_{t}$ via $q_{t,0}$ to the inner core. The $x$-axis represents graph distance from $q_{t,0}$. The black dashed horizontal lines represent the degree-threshold sequence, while the continuous black lines represent the maximal degree in the graph at time $t_i$. The degree of $q_{t,0}$, the maximal degree in the graph and the degree threshold for the inner core all increase over time. The degree of vertex $q_{t,0}$ satisfies the inner-core threshold at time $t_3$. The red dashed segment from $q_t$ to $q_{t,0}$
  is the same for all $i$, while the blue segment from $q_{t,0}$ is constructed at the times $(t_i)_{i\leq 3}$.}
 \label{fig:upper-bound}
\end{figure}

There are three main difficulties in the outlined procedure.
Firstly, it is not good enough to start the two-connector procedure from $u_t$ (or $v_t$) because the error terms coming from controlling the growth of the degree of $u_t$ (or $v_t$) at $t'$ close to $t$ are too large. To resolve this, we start the procedure from a vertex -- say $q_{t,0}$ -- that has degree at least $s_0^\sss{(0)}$ at time $t$ for some large but universally bounded constant $s_0^\sss{(0)}$.
The segment between $q_t$ and $q_{t,0}$ is fixed for all $t'\geq t$, so that we only have to account for a possible error once.
Secondly, we need to bound the degree of the vertex $q_{t,0}$ from below over the entire time interval $[t, \infty)$, not just at a specific time $t'$.
For this we employ martingale arguments.
Lastly, to make the error probabilities summable in $t'$, we argue that the two-connector procedure does not have be executed for every time $t'\geq t$, but only along a specific subsequence of times $(t_i)_{i\geq 0}$, where
\[
 t_i\in \big[t\exp\big((\tau-2)^{-i+1}\big),t\exp\big((\tau-2)^{-i-1}\big)\big].
\]
This sequence is chosen such that at time $t_{i+1}$ one iteration less than at time $t_i$ is needed to reach the inner core from the initial vertices, and these are exactly the times when $K_{t,t'}$ crosses an integer and hence a previously present path is no longer short enough. See Figure \ref{fig:upper-bound} for a sketch. On the time scale $T_t(a)$, the number of iterations scales linearly in $a\in[0,1]$.

\begin{figure}[t]
 \centering
 \includegraphics[width=0.93\textwidth]{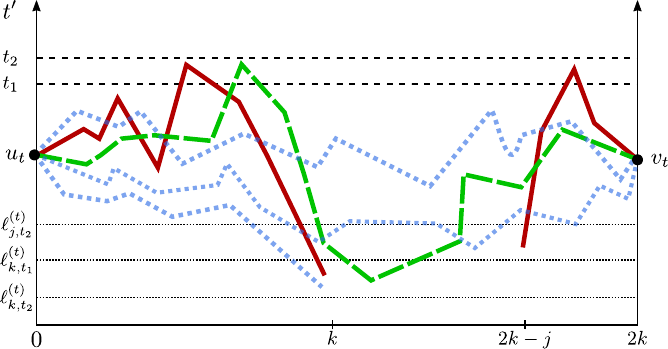}
 \caption{Good and bad path decomposition for the lower bound.
 The $y$-axis represents arrival time of vertices and the $x$-axis the graph distance from $u_t$. Bad paths are displayed in red, good paths are green and dashed. The blue dotted paths represent possible paths that are absent at time $t_1$. Let $t_2>t_1$ and $k>j$. The black tiny-dotted horizontal lines represent the birth-threshold array $(\ell_{k,t'}^\sss{(t)})_{k\geq 0, t'\geq t}$ which is decreasing in $t'$ at the times $t_1,t_2$ and also decreasing in $k$. At time $t_1$ there is neither a bad path of length at most $k$ present, nor a good path of length $2k$ that connects $u$ and $v$. Then, if $u$ and $v$ are at time $t_2$ at graph distance $2k$, there must be either a bad path of length at most $k$ emanating from $u$ or $v$ that traverses a vertex in $(t_1,t_2]$, or there must be a short good path traversing such a vertex. Observe that the good path is allowed to traverse a vertex in $[\ell_{k,t_2}^\sss{(t)},\ell_{k,t_1}^\sss{(t)})$. In particular, this holds for $t_2=t_1+1$.}
 \label{fig:lower-bound-intro}
\end{figure}
For the lower bound  we first bound the probability that the graph distance $d_G^{(t')}(u_t,v_t)$ is ever too short and then extend it to weighted distances.
To estimate the probability of a too short path being ever present, we develop a refined truncated path-counting method inspired by \cite{dereich2012typical}.
Let, for a fixed $t$, $(\ell_{k,t'}^{\sss{(t)}})_{k\geq0,t'\geq t}$ be an array of birth times, i.e., arrival times of vertices.
The path-counting method first excludes possible paths from $u_t$ to $v_t$ that are unlikely to be present in $\mathrm{PA}_{t'}$, called \emph{bad paths}.
A bad path of length $k$ reaches a vertex born before time $\ell_{k,t'}^\sss{(t)}$ using only vertices born before $t'$. The longer a path is, the more likely it is that an old vertex can be reached. Moreover, as the graph grows, it becomes more likely that there is a short path to an old vertex. The array of birth times $(\ell_{k,t'}^\sss{(t)})_{k\geq 0,t'\geq t}$ is therefore nonincreasing in both parameters.
Among the other possible paths that are too short, the \emph{good paths}, the method counts the expected number of paths from $u_t$ to $v_t$  that are present in $\mathrm{PA}_{t'}$.
More precisely, the expected number of these paths of length at most $2K_{t,t'}-M_G$ is shown to be much smaller than one for some $M_G>0$.
The decomposition of good and bad paths is done for every $t'>t$, in an interlinked way.
The crucial observation is that if there is no too short path present at time $t'$, but there is a too short path present at time $t'+1$, then the vertex labelled $t'+1$ must be on this connecting path and thus it must be either on a bad path or on a too short good path. This trick allows us to develop a first moment method much sharper than a union bound simply over $t'$, since we only need to bound the expected number of bad or too short good paths that are restricted to pass through the newly arrived vertex $t'$.
 These bounds are a factor $1/t'$ smaller than similar bounds without the restriction. As a result, the error bound is summable in $t'$ and tends to zero as $t$ tends to infinity.  See Figure \ref{fig:lower-bound-intro} for a sketch of the argument.

To extend the result from graph distances to weighted distances for Theorem \ref{thm:weighted-evolution} below, we observe that if the graph distance between  $u_t$ and $v_t$ is at least $2K_{t,t'}-M_G/2$, then the graph neighbourhoods of radius $K_{t,t'}-M_G/2$ must be disjoint. A path that connects $u_t$ to $v_t$ must cross the boundaries of these graph neighbourhoods. We bound the number of vertices at distance precisely $k$ from $q_t\in\{u_t,v_t\}$ from above, for $k\leq K_{t,t'}-M_G/2$. This allows to bound the weight of the least-weight edge between vertices at distance $k$ and $k+1$  from $q_t$ from below. The sum of these minimal weight bounds is then a lower bound to reach the boundary.
However, the error probabilities are not summable in $t'$. To resolve this, we show that it is sufficient to consider only a subsequence of times, similarly to the upper bound.

\subsection*{Organisation}
In the next section we rigorously define the models.
The lower bound is proven in Section \ref{sec:lower-bound}.
In Section \ref{sec:upper-bound} we present the proof of the upper bound.

\subsubsection*{Notation}
For two functions $f(x)$ and $g(x)$, we say $f(x)=o\big(g(x)\big)$ if $\lim_{x\to\infty} f(x)/g(x) = 0$, and write $f(x)=O(g(x))$ if $\lim_{x\to\infty} f(x)/g(x)<\infty$. For $\min\{a,b\}$ and $\max\{a,b\}$, we write $a\wedge b$ and $a\vee b$, respectively.
We define $[n]:=\{1,\dots,n\}$, while $\floor{n}:=\max\{x\in\mathbb{N}: x\leq n\}$ and $\ceil{n}:=\max\{x\in\mathbb{N}: x\geq n\}$.
Let $(X_n)_{n\geq0}$ and $(Y_n)_{n\geq 0}$ be two sequences of random variables. We say that $X_0$ dominates $Y_0$ if there exists a coupling of the random variables such that $\P(X_0\geq Y_0)=1$. Similarly, the sequence $(X_n)_{n\geq0}$ dominates $(Y_n)_{n\geq 0}$ if there exists a coupling of the sequences such that $\P(\forall_{n\geq0}:X_n\geq Y_n)=1$. A random graph dominates a random graph $H$ if there exists a coupling such that every edge in $H$ is also contained in $G$. If a random object $X$ dominates $Y$, we write $X\overset{d}\geq Y$.
We say that $(X_n)_{n\geq 0}$ converges in probability to a random variable $X_\infty$, i.e., $X_n\toinp X_\infty$, if for all $\vareps>0$ it holds that $\P(|X_n-X_{\infty}|>\vareps)=o(1)$.
A sequence of events $(\CE_n)_{n\geq 0}$ holds with high probability (whp)
if $\P(\CE_n)=1-o(1)$, and abbreviate `with probability' by w/p. The complement of an event $\CE$ is denoted by $\neg\CE$.
For a sequence of vertices in $(\pi_i)_{i\leq n}$ with birth times at most $t'$, we write $\{\pi_0\leftrightarrow\pi_1\}$ for the event that $\pi_0$ and $\pi_1$ are connected by an edge in $\mathrm{PA}_{t'}$ for $t'\geq \max\{\pi_0,\pi_1\}$. Moreover, we define $\{\pi_0\leftrightarrow\cdots\leftrightarrow\pi_n\}:=\{\pi_0\leftrightarrow\pi_1\}\cap\cdots\cap\{\pi_{n-1}\leftrightarrow\pi_{n}\}$.
The sequence (or path) $(\pi_i)_{i\leq n}$ is called self-avoiding if $\pi_i\neq \pi_j$ for all $i\neq j$.

\section{Model definition and general results}\label{sec:models-examples}
The first model that we introduce is a classical model where every arriving vertex connects to a fixed $m\in\mathbb{N}$ vertices, and the edges are created sequentially.
It is often called the $(m,\delta)$-model, and appeared first in \cite{berger2005spread,bollobas2004diameterpa}, for variations see \cite[Chapter 8]{hofstad2016book1}.
Denote by $D^{\leftarrow}_v(t, j)$ the number of incoming connections of a vertex $v$ after $j$ edges have been formed at time $t$, for $j=\{1,\dots,m\}$.
We abbreviate $D_v^{\leftarrow}(t):=D_v^{\leftarrow}(t,m)$, and denote by $\{t\overset{j}{\rightarrow}v\}$ the event that the $j$-th edge, for $j\in\{1,\dots,m\}$, of vertex $t$ connects to $v<t$.
\begin{definition}[Fixed-outdegree preferential attachment]\label{def:fpa}
 Fix $m\in\mathbb{N}, \delta\in(-m,\infty)$. Let $\mathrm{FPA}_1(m,\delta)$ be a single vertex without any edges.
 We define $\mathrm{FPA}(m,\delta)$ by the following sequence of conditional connection probabilities corresponding to the attachment of the $j$-th edge
 \begin{equation}
  \P\big(\{t\overset{j}{\rightarrow}v \} \mid \mathrm{FPA}_{(t,j^-)}\big) = \frac{D^{\leftarrow}_{v}(t,j-1) + m(1+\delta/m)}{(t-2)(\delta+2m) + j-1+m+\delta},
  \quad v\in[t-1],\label{eq:fpa}
 \end{equation}
 where $\mathrm{FPA}_{(t,j^-)}$ denotes the graph right before the insertion of the $j$-th edge of $t$. An important parameter of the model is
 \begin{equation}
  \tau_{m,\delta} := 3 + \delta/m. \label{eq:tau-fpa}
 \end{equation}
\end{definition}
\noindent The denominator in \eqref{eq:fpa} is a normalizing constant.
Definition \ref{def:fpa} does not allow self-loops, since $v\in[t-1]$, but allows multiple edges between vertices.

More recently, a similar model has been introduced where the outdegree of arriving vertices is variable, since the arriving vertex connects independently to existing vertices \cite{dereich2009random}. Again, $D_v^{\leftarrow}(t)$ denotes the indegree of vertex $v$ right after time $t$.
\begin{definition}[Variable-outdegree preferential attachment]\label{def:vpa}
 Let $f : \mathbb{N}\to (0,\infty)$ be a concave function satisfying $f(0)\leq 1$ and $f(1)-f(0)<1$. \
 We call $f$ the \emph{attachment rule}. Let $\mathrm{VPA}_1(f)$ be a single vertex without any edges.
 The model $\mathrm{VPA}(f)$ is defined by the following sequence of conditional connection probabilities corresponding to the attachments of the vertex arriving  $t$, i.e.,
 \begin{equation*}
  \P\big(\{t\rightarrow v\} \mid \mathrm{VPA}_{t-1}\big) = \frac{f\big(D_v^{\leftarrow}(t-1)\big)}{t},\qquad v\in[t-1],
 \end{equation*}
 where the connections to existing vertices are formed independently of each other. Important parameters of the model are
 \begin{equation}
  \gamma_f := \lim_{k\rightarrow\infty} f(k)/k, \qquad \tau_f:= 1+1/\gamma, \label{eq:tau-vpa}
 \end{equation}
 which are well-defined by the concavity of $f$, assuming $\gamma_f>0$. We call $\tau_f$ the power-law exponent.
 In this paper we restrict ourselves to affine attachment rules, i.e., $f(k)=\gamma k + \beta$.
\end{definition}
\noindent
Observe that in $\mathrm{VPA}(f)$, as in $\mathrm{FPA}(m,\delta)$, no self loops are possible. However, unlike $\mathrm{FPA}(m,\delta)$, $\mathrm{VPA}(f)$ does not allow for multiple edges between vertices.  Generally $\mathrm{FPA}(m,\delta)$ and $\mathrm{VPA}(f)$ show qualitatively the same behaviour when $\tau_{m,\delta}=\tau_f$. Therefore, we often refer to preferential attachment (PA) with a power-law exponent $\tau>2$, by which we mean either $\tau_{m,\delta}$ in \eqref{eq:tau-fpa} or $\tau_f$ in \eqref{eq:tau-vpa}.
Observe that \eqref{eq:intro-p-conn} holds for both models.

We now formalize the notion of paths for a sequence of growing graphs, which is used to define distances and distance evolutions.
\begin{definition}[Paths]\label{def:paths}
 We call a vertex tuple $(\pi_0,\dots,\pi_n)=:\bm{\pi}$ a $q$-path if $\pi_0=q$, and we call it a $(u,v)$-path if $\pi_0=u$, $\pi_n=v$, and $u\neq v$.
 The path $\bm{\pi}$ is called $t'$-possible if $\max_{i\leq n} \pi_i \leq t'$ and $t'$-present if it is $t'$-possible and all edges $\{(\pi_0,\pi_1),\dots,(\pi_{n-1},\pi_n)\}$ are present in the graph at time $t'$.
\end{definition}

For $u_t,v_t\in V_t$, let $\Omega_{t'}(u_t,v_t):=\{\bm{\pi}: \bm{\pi}\text{ is a $t'$-present $(u_t,v_t)$-path}\}$ denote the set of $t'$-present paths. Since the edge set and vertex set are increasing in $t'$, new paths between $u_t$ and $v_t$ emerge. Hence, we have that $\Omega_{t'}(u_t,v_t)\subseteq \Omega_{\tilde t}(u_t,v_t)$ for $\tilde t\ge t'$.

We equip every edge with a weight, an i.i.d.\ copy of a non-negative random variable $L$. The weight of an edge represents the time (for a fluid/information) to traverse an edge.
The model where weighted distances are studied in (random) graphs, is also called \emph{first-passage percolation},
see \cite{auffinger201750,hofstad2017stochasticprocesses} and their references for an overview of first-passage percolation on (random) graphs.

\begin{definition}[Distances in graphs]\label{def:distances}
 Consider the graph $\mathrm{PA}_t=(V_t,E_t)$ and let every edge $e$ be equipped with a weight $L_e$.
 We define the graph distance and weighted distance between $u_t,v_t\in V_t$ at time $t'$ as
 \begin{align*}
  d_G^{(t')}(u_t,v_t):= \min_{\bm{\pi}\in\Omega_{t'}(u_t,v_t)}\sum_{e\in\bm{\pi}}1,
  \qquad
  d_L^{(t')}(u_t,v_t):= \min_{\bm{\pi}\in\Omega_{t'}(u_t,v_t)}\sum_{e\in\bm{\pi}} L_e.
 \end{align*}
 For a vertex $v$ and a vertex set $\CW\subseteq[t']$, we define
 \begin{align*}
  d_G^{(t')}(v,\CW):= \min_{w\in\CW}d_G^{(t')}(v,w),
  \qquad
  d_L^{(t')}(v,\CW):= \min_{w\in\CW}d_L^{(t')}(v,w).
 \end{align*}
 If $u_t$ and $v_t$ are two typical vertices, i.e., they are sampled uniformly at random from $\mathrm{PA}_t$, then we call $\big(d_G^{(t')}(u_t,v_t)\big)_{t'\geq t}$ and $\big(d_L^{(t')}(u_t,v_t)\big)_{t'\geq t}$ the graph-distance evolution and weighted-distance evolution, respectively.
\end{definition}

To state our main result, we introduce two quantities to classify edge-weight distributions. Let
\begin{equation}
 \bm{I}_1(L) := \sum_{k=0}^\infty F_L^{(-1)}\big(\exp(-\exp(k))\big), \qquad \bm{I}_2(L):=  \sum_{k=0}^\infty \frac{1}{k}F_{L-b}^{(-1)}\big(\exp(-\exp(k))\big),\label{eq:exp-crit}
\end{equation}
where $F_L^{(-1)}(y):=\inf_x\{x\in\mathbb{R}: F_L(x)\geq y\}$ is the generalized inverse of $F_L(x):=\P(L\leq x)$, and $b:=\inf_x\{x: F_L(x)>0\}$.
See Remark \ref{remark:weight-dist} below for comments on $\bm{I}_1(L)$ and $\bm{I}_2(L)$. The following function will describe the weighted distance. Define for $a,b\in\mathbb{N}$
\begin{equation}
  \mathcal{Q}(a,b]:=\sum_{k=a+1}^{b}F_L^{(-1)}\big(\exp\big(-(\tau-2)^{-k/2}\big)\big),\label{eq:qt-sum}
\end{equation}
so that $\CQ(a,b]$ is a sum consisting of $b-a$ terms.
Recall $K_{t,t'}$ from \eqref{eq:kt} that describes the graph distance. We define for $t'\geq t$ its weighted-distance counterpart
\begin{equation}
 Q_{t,t'}:= \CQ(K_{t,t}-K_{t,t'}, K_{t,t}]. \label{eq:qt-kt}
\end{equation}
\begin{theorem}[Main result] \label{thm:weighted-evolution}
 Consider the preferential attachment model with power-law exponent $\tau\in(2,3)$. Equip every edge upon creation with an i.i.d.\ copy of the non-negative random variable $L$. Let $u_t,v_t$ be two typical vertices at time $t$. If $\bm{I}_2(L)<\infty$, then
 \begin{equation}
  \bigg(\sup_{t'\geq t} \left|d_L^{(t')}(u_t,v_t) - 2Q_{t,t'}\right|\bigg)_{t\geq 1} \label{eq:weighted-evolution-restricted-l}
 \end{equation}
 is a tight sequence of random variables.
 Regardless of the value of $\bm{I}_2(L)$, for any $\delta,\vareps>0$, there exists $M_L>0$ such that
 \begin{equation}
  \P\Big(\forall t'\geq t: 2Q_{t,t'}- M_L\leq d_L^{(t')}(u_t,v_t)\leq 2(1+\vareps)Q_{t,t'} + M_L\Big) \leq \delta.\label{eq:weighted-evolution-general-l}
 \end{equation}
\end{theorem}
Theorem \ref{thm:weighted-evolution} tracks the evolution of $d_L^{(t')}(u,v)$ as time passes and the graph around $u$ and $v$ grows, since in \eqref{eq:weighted-evolution-restricted-l} the supremum  is taken  over $t'$ and $t'$ is inside the $\P$-sign in \eqref{eq:weighted-evolution-general-l}.
It is the $(1+\varepsilon)$-factor in the upper bound in \eqref{eq:weighted-evolution-general-l} that makes \eqref{eq:weighted-evolution-general-l} different from \eqref{eq:weighted-evolution-restricted-l}. Thus, the lower bound is tight for any non-negative weight distribution.
A special case of Theorem \ref{thm:weighted-evolution} is when the edge-weight distribution $L\equiv 1$. Then the weighted distance and graph distance coincide, yielding Theorem \ref{cor:graph-evolution}, since $\bm{I}_2(1)=0$.

Observe that $2Q_{t,t'}=2\CQ(K_{t,t}-K_{t,t'}, K_{t,t}]$ in \eqref{eq:qt-kt} could be seen as two sums, each consisting of $K_{t,t'}$ terms: the number of terms in $Q_{t,t'}$ is equal to the number of edges on the shortest graph-distance path. The additive constant $M_L$ ensures that there will be many almost-shortest paths, from which we are able to choose one with low edge-weight. As time passes, the degrees of  $u_t$ and $v_t$ increase, so that it becomes more likely that there are edges close to $u_t$ and $v_t$ that have small edge-weights. Since the terms in $Q_{t,t'}$ are decreasing in $k$, this intuitively explains that $Q_{t,t'}$ consists of the \emph{smallest} $K_{t,t'}$ terms of $Q_{t,t}$, rather than the \emph{largest} $K_{t,t'}$ terms of the sum defining $Q_{t,t}$.
However, if $L\equiv 1+X$ for some random variable $X$ that satisfies $\bm{I}_1(X)<\infty$ (e.g., $X$ exponential, gamma, or a power of uniform on $[0,1]$), then $|K_{t,t'}-Q_{t,t'}|\le M$ for some constant $M$. Consequently, the graph distance and weighted distance are of the same order (up to additive constants). This phenomenon has also been observed for the Configuration Model \cite{Baroni2019}. As a result, for weight distributions with $\bm{I}_1(X)<\infty$ the location of the summation interval in \eqref{eq:qt-kt} does not influence the main result: there exists a constant $M_1$ such that $\CQ(0, K_{t,t'}]-\CQ(K_{t,t}-K_{t,t'}, K_{t,t}]\le M_1$ for all $t'\ge t\ge 0$.
For the other case, if $L=1+X$ such that $\bm{I}_1(X)=\infty$, such a constant does not exist.
For such distributions, the fact that the lower summation boundary in \eqref{eq:qt-kt} is shifted to $K_{t,t}-K_{t,t'}$ from $0$ matters and influences the growth rate.  As an example, we set $L$ such that the terms in the sum in \eqref{eq:qt-sum} are equal to $1+1/k$, yielding $\CQ(0,K_{t,t'}]\approx K_{t,t'}+\log(K_{t,t'})$, while for $t'$ large enough that $K_{t,t}-K_{t,t'}\gg 1$:
\begin{align*}
\CQ(K_{t,t}-K_{t,t'}, K_{t,t}]  &\approx K_{t,t'} + \log(K_{t,t}) - \log(K_{t,t}-K_{t,t'}) \ll \CQ(0,K_{t,t'}].
\end{align*}
We now recall the hydrodynamic limit for the graph-distance evolution in Corollary \ref{cor:hydro-graph}. A similar limit can be derived for the weighted-distance evolution if the weight distribution satisfies $\bm{I}_1(L)=\infty$.
The proper scaling and the constant prefactor, similar to \eqref{eq:intro-hydro-graph}, can be determined through studying the main growth term of $Q_{t,t'}$ in \eqref{eq:qt-kt} if $F_L^\sss{(-1)}$ is explicitly known.

Like Remark \ref{remark:graph-distance-2}, one can show that at the time scale $\Theta(t^{2/(3-\tau)})$ the weighted distance between $u_t$ and $v_t$ tends to $2b$ where $b:=\inf\{x\in\mathbb{R}: F_L(x)> 0\}$. At this time scale, many vertices connect to both $u_t$ and $v_t$, allowing to bound the weighted distance from above by $2(b+\varepsilon)$ for arbitrarily small $\varepsilon>0$.

Lastly, we recall the static counterpart of Theorem \ref{thm:weighted-evolution} by the authors in \cite{jorkom2019weighted} that generalizes earlier results on graph distances in PAMs \cite{caravenna2016diameter,dereich2012typical, dommers2010diameters}.
In \cite[Theorem 2.8]{jorkom2019weighted} it is shown that, for weight distributions satisfying $\bm{I}_2(L)<\infty$,
\be
\Big(d_L^{(t)}(u_t,v_t) - 2Q_{t,t}\Big)_{t\geq 1}\label{eq:intro-distances-1}
\ee
forms a tight sequence of random variables. Observe that Theorem \ref{thm:weighted-evolution} extends this result.
For the configuration model, similar results to \eqref{eq:intro-distances-1} were derived subsequently in \cite{adriaans2017weighted,baroni2017nonuniversality,bhamidi2010first}, indicating universality of first-passage percolation: the scaling for the two models is the same up to constant factors when $\tau\in(2,3)$.

We now comment on the quantities $\bm{I}_1(L)$ and $\bm{I}_2(L)$ from \eqref{eq:exp-crit} that are used to classify edge-weight distributions.
\begin{remarky}[Explosive and conservative weight distributions]\label{remark:weight-dist}
 If $\bm{I}_1(L)<\infty$, we call the weight distribution explosive, otherwise we call it conservative.
 $\bm{I}_1(L)$ measures how \emph{flat} the edge-weight distribution $F_L$ is around the origin. Many well-known distributions with support starting at zero are explosive distributions, e.g.\ $\text{Unif}[0,r]$, $\text{Exp}(\lambda)$. On the contrary, distributions that have support that is bounded away from zero automatically belong to the conservative class. The second quantity, $\bm{I}_2(L)$, measures flatness of $F_L$ around the start of its support and is infinite only for distributions that are \emph{extremely flat} near $b$. More concretely, if $F_L$ in the neighbourhood of zero satisfies for some $\beta\geq 1$
 \[
  F_L(x)=\exp\big(-\exp\big(\re^{x^{-\beta}}\big)\big),
 \]
 then $\bm{I}_2(L)=\infty$, while for $\beta\in(0,1)$ it holds that $\bm{I}_2(L)<\infty$.
 We are mostly interested in distributions that satisfy $\bm{I}_1(L)=\infty$, as by \cite[Theorem 2.8]{jorkom2019weighted} the typical weighted distance is already of constant order if $\bm{I}_1(L)<\infty$, making Theorem \ref{thm:weighted-evolution} a trivial statement in this case. Observe also that in this case $Q_{t,t'}$ in \eqref{eq:qt-kt} is bounded from above by some constant.
\end{remarky}

\section{Proof of the lower bound}\label{sec:lower-bound}
Now we prove the lower bound of Theorem \ref{thm:weighted-evolution}, i.e., we show that with probability close to one there is no \emph{too} short path between  $u_t$ and $v_t$ for any $t'\geq t$. The main contribution of this section versus existing literature, e.g.\ \cite{caravenna2016diameter,dereich2012typical,jorkom2019weighted}, is the following proposition concerning the graph distance. In its proof we develop a path-decomposition technique that uses the dynamical construction of $\mathrm{PA}_t$ in a refined way to get strong error bounds that are summable over $t'\geq t$. After the notational and conceptual set-up of the argument, we state and prove some technical lemmas.  In the end of the section, we extend Proposition \ref{proposition:lower-bound-graph} to the edge-weighted setting, using refinements of the error bounds in \cite{jorkom2019weighted}. We abbreviate $u:=u_t$ and $v:=v_t$, respectively.
\begin{proposition}[Lower bound graph distance]\label{proposition:lower-bound-graph}
 Consider the preferential attachment model with power-law exponent $\tau\in(2,3)$. Let $u, v$ be two typical vertices in $\mathrm{PA}_t$. Then for any $\delta>0$, there exists $M_G>0$ such that
 \be
 \P\big(\exists t'\geq t: d_G^{(t')}(u,v) \leq 2K_{t,t'}-2M_G\big) \leq \delta.\label{eq:lower-bound-graph}
 \ee
\end{proposition}
\noindent
Observe that $t'$ is inside the $\P$-sign. Hence, \eqref{eq:lower-bound-graph} tracks the evolution of $d_G^{(t')}(u,v)$ as time passes, and the graph around $u$ and $v$ grows.
To estimate the probability of a too short path, we use a truncated path-counting method similar to \cite{dereich2012typical}.
This method first excludes possible paths that are unlikely to be present, called \emph{bad paths}.
Then, among the rest, the \emph{good paths}, it counts the expected number of paths that are too short and present in $\mathrm{PA}_{t'}$.
More precisely, the expected number of paths between $u$ and $v$ of length at most
\be
2\underline{K}_{t,t'}:=2K_{t,t'}-2M_G\label{eq:lower-kt}
\ee
is shown to be much smaller than one.
We do this decomposition  in an interlinked way that ensures that paths are only counted once.

\subsection{Set-up for the graph-distance evolution}
Recall that the arrival time of a vertex is also called birth time.
The decomposition of good and bad paths is based on an array of birth times $(\ell_{k,t'}^\sss{(t)})$ for which we make the following assumption throughout this section.
\begin{assumption}\label{ass:ell}
 The array of birth times $(\ell_{k,t'}^\sss{(t)})_{k\geq 0,t'\geq t}$ is a positive integer-valued array that is nonincreasing in both parameters and satisfies $\ell_{0,t'}^{(t)}\leq t$. We call it the birth-threshold array.
\end{assumption}
\noindent Recall the definition of paths in Definition \ref{def:paths}.
\begin{definition}\label{def:t-bad}
 Let $(\ell_{k,t'}^\sss{(t)})_{k\geq0,t'\geq t}$ be an array satisfying Assumption \ref{ass:ell}.
 A $t'$-possible $q$-path $(\pi_0,\dots,\pi_k)$ is called $t'$-good if $t'\geq\pi_j\geq \ell_{j,t'}^\sss{(t)}$ for all $j\leq k$, otherwise it is called $t'$-bad.
 A $t'$-possible $(u,v)$-path $(\pi_0,\dots,\pi_n)$ is called $(n,t')$-good if $\pi_j\wedge\pi_{n-j}\geq \ell_{j,t'}^\sss{(t)}$ for all $j\leq \floor{n/2}$, otherwise it is called $t'$-bad.
\end{definition}
\noindent
This definition calls any path bad if it has a too \emph{old} vertex, where the threshold $\ell_{k,t'}^\sss{(t)}$ depends on the distance from $\pi_0$. Thus, all vertices on a good path are sufficiently \emph{young}.
We decompose $t'$-bad paths according to their first vertex violating the threshold.
\begin{definition}\label{def:kt-bad}
 Let $(\ell_{k,t'}^\sss{(t)})_{k\geq0,t'\geq t}$ be an array satisfying Assumption \ref{ass:ell}.
 We say that a $\pi_0$-path $(\pi_0,\dots,\pi_n)$ of length $n$ is $(k,t')$-bad if the path is $t'$-possible and $\pi_j\geq \ell_{j,t'}^\sss{(t)}$ for all $j<k$, but $\pi_k< \ell_{k,t'}^\sss{(t)}$.
\end{definition}

\begin{observation}\label{obs:lower-bound}
 Let $(\ell_{k,t'}^\sss{(t)})_{k\geq0,t'\geq t}$ be an array satisfying Assumption \ref{ass:ell}.
 Then
 \begin{enumerate}
  \item if a path is $t'$-good, then it is $\tilde{t}$-good for all $\tilde{t}\geq t'$.
  \item if a path is $t'$-bad, it is possible that it turns $\tilde{t}$-good for some $\tilde{t}>t'$.
  \item if a path $(\pi_0,\dots,\pi_n)$ is $t'$-bad, then it is $\tilde{t}$-bad for any $\tilde{t}\in[\max\{\pi_i\}, t']$.
  \item if for all $i\le k$ no $(i,t'-1)$-bad path is present in $\mathrm{PA}_{t'-1}$, then a $(k, t')$-bad path can only be present in $\mathrm{PA}_{t'}$ if it passes through vertex $t'$.
 \end{enumerate}
\end{observation}

\begin{figure}[t]
 \centering
 \begin{subfigure}{.29\textwidth}
  \centering
  \includegraphics[height=.25\textheight]{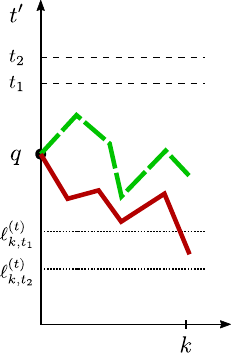}
  \caption{}
 \end{subfigure}
 \begin{subfigure}{.29\textwidth}
  \centering
  \includegraphics[height=.25\textheight]{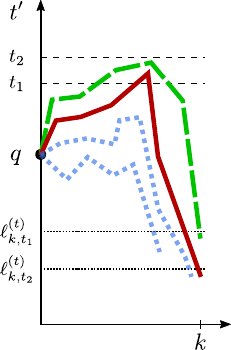}
  \caption{}
 \end{subfigure}
 \begin{subfigure}{.39\textwidth}
  \centering
  \includegraphics[height=.25\textheight]{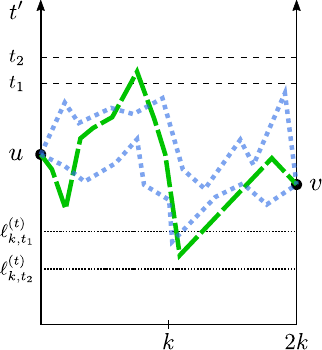}
  \caption{}
 \end{subfigure}
 \caption{Good and bad path decomposition for the lower bound. Bad paths are displayed in red, the green dashed lines are the good paths, and the blue dotted lines represent possible paths that are absent. The $y$-axis represents the birth time of the vertices and the $x$-axis the graph distance from $q$ and $u$, respectively. In Figure (A) we see that if a path is $t_1$-good, then it is also $t_2$-good
  since $t'\mapsto\ell_{k,t'}^\sss{(t)}$ is decreasing. However, the red $(k,t_1)$-bad path turns $(k,t_2)$-good. Figure (B) shows that if there is no $(k,t_1)$-bad path, a red $(k,t_2)$-bad path must pass through a vertex in $(t_1,t_2]$. Note that the green path in Figure (B) is $t_2$-good. Although it violates a birth threshold valid at time time $t_1$, the path is not $t_1$-bad because it is not $t_1$-possible.  Figure (C) shows that if there is neither a good, nor a bad $t_1$-present $(u,v)$-path of length $2k$, then a $t_2$-present (good) $(u,v)$-path must pass through a vertex in $(t_1,t_2]$. We apply the observations in these figures for $t_2=t_1+1$.}
 \label{fig:lower-bound}
\end{figure}
\noindent
All four observations follow directly from the definitions of good and bad paths, and the fact that $(\ell_{k,t'}^\sss{(t)})_{k\geq0,t'\geq t}$ is decreasing in both parameters, see Figure \ref{fig:lower-bound}(A-B). The fourth observation turns out to be crucial in our decomposition argument.
We define the events whose union implies the event between brackets in \eqref{eq:lower-bound-graph}. We start with the event of having a bad path emanating from $q\in\{u,v\}$ for $k\geq 1, t'\geq t$, i.e.,
\begin{numcases}{\CE_{\text{bad}}^{(q)}(k,t'):=}
 \big\{\text{$\exists (k,t)$-bad $q$-path}\big\}, & $t'=t$, \label{eq:lower-bad-def-t}\\
 \big\{\text{$\exists (k,t')$-bad $q$-path, $\forall_{i\leq k}:\nexists (i,t'-1)$-bad $q$-path}\big\}, & $t'>t$.\label{eq:lower-bad-def}
\end{numcases}
Here the sign $\exists$ indicates that a path is present.
For completeness, we define for $t'\geq t$, $k=0$,
\begin{equation}
 \CE_{\text{bad}}^{(q)}(0,t'):=\big\{q<\ell_{0,t'}^\sss{(t)}\big\} \subseteq \CE_{\text{bad}}^{(q)}(0,t),\label{eq:lower-bad-def-init}
\end{equation}
where the inclusion follows since $t'\mapsto\ell_{0,t'}^\sss{(t)}$ is nonincreasing.
By the additional restriction on bad paths in \eqref{eq:lower-bad-def}, the events $\CE_{\text{bad}}^{(q)}(k,t')$ are disjoint in both parameters.
For $\CE_{\text{bad}}^{(q)}(k,t')$ and $t'>t$, as a result of Observation \ref{obs:lower-bound}(4) and the restriction in the definition \eqref{eq:lower-bad-def} of not having a bad path at time $t'-1$, we only have to consider paths that pass through the vertex $t'$. This motivates to decompose the $(k,t')$-bad paths passing through vertex $t'$ according to the number of edges between the initial vertex $q\in\{u,v\}$ and $t'$.
Indeed, consider a $(k,t')$-bad $q$-path where $t'$ is the $i$-th vertex, i.e., it is of the form $(q,\pi_1,\dots,\pi_{i-1},t',\pi_{i+1},\dots,\pi_k)$.
Then, by Definition \ref{def:kt-bad}, the constraints that this path satisfies is that for $j<k$, $\pi_j\geq\ell_{j,t'}^\sss{(t)}$. This means that on the segment $(\pi_{i+1},\dots,\pi_k)=:(\sigma_1,\dots,\sigma_{k-i})$ the indices of the constraints have to be shifted by $i$, giving rise to $\sigma_j\geq\ell_{i+j,t'}^\sss{(t)}$ for $j\leq k-i$. Hence, we introduce \emph{good paths on a segment}. Recall that $\{\pi_0\leftrightarrow \cdots\leftrightarrow\pi_n\}$ means that $(\pi_0,\dots,\pi_n)$ is $t'$-present for $t'=\max_{i\leq n}\pi_i$.

\begin{definition}\label{def:good-paths-segment}
 Given an array $(\ell_{k,t'}^\sss{(t)})_{k\geq0,t'\geq t}$ satisfying Assumption \ref{ass:ell}, let
 \begin{align}
  \big\{x \goodpathto{i}{n} y \big\}_{t'} & := \big\{\mathrm{disjoint }(\pi_i, \dots, \pi_{n-1}): x=\pi_i\leftrightarrow\cdots \leftrightarrow\pi_{n-1}\leftrightarrow y\mid \forall_{j< n}: \pi_j \in [\ell_{j,t'}^\sss{(t)},t'] \big\}.\nonumber
 \end{align}
 If $\big\{x \goodpathto{i}{n} y \big\}_{t'}\neq \varnothing$, we say that there is a $t'$-good $x$-path on segment $[i,n)$. We write
 \[
  \big\{x \goodpathto{i_1}{n_1} y \big\}_{t'} \circ \big\{y \goodpathto{i_2}{n_2} z \big\}_{t'}
 \]
 for the set of self-avoiding $(x,z)$-paths that are $t'$-good on the segment $[i_1,n_1)$ from $x$ to $y$ and $t'$-good on the segment $[i_2,n_2)$ from $y$ to $z$.
\end{definition}
\noindent
Note that there is no birth restriction on the last vertex on the segment, explaining the half-open interval superscript $[i,n)$. Thus, if $\pi_k<\ell_{k,t'}^\sss{(t)}$, then
\[
 \big|\big\{q \goodpathto{0}{i} t'\big\}_{t'} \circ \big\{ t'\goodpathto{i}{k} \pi_k\}_{t'}\big| \geq 1
\]
precisely means that there is a $(k,t')$-bad $q$-path from $q$ to $\pi_k$ that has $t'$ as its $i$-th vertex.
For notational convenience we omit the subscript $t'$.

Having set up the definitions for the bad paths, we define the events that allow to count the expected number of too short good $(u,v)$-paths. Let, for $n\geq 1$,
\begin{numcases}
 {\CE_{\text{short}}^{(u,v)}(n,t'):=}
 \big\{\text{$\exists (n,t)$-good $(u,v)$-path}\big\},& $t'=t$,\label{eq:lower-short-def-t}\\
 \big\{\text{$\exists (n,t')$-good $(u,v)$-path, $\forall_{\tilde{t}<t'}:\nexists (u,v)$-path of length $n$}\big\}, &$t'>t$,
 \label{eq:lower-short-def}
\end{numcases}
and set for completeness
\begin{align}
 \CE_{\text{short}}^{(u,v)}(0,t') & :=\{u=v\}.\nonumber
\end{align}
Observe that in \eqref{eq:lower-short-def} we require that at previous times there was neither a good, nor a bad path of length $n$ between $u$ and $v$. This is a
stronger requirement than the one in \eqref{eq:lower-bad-def}, where we do not put any restrictions on good paths at a previous time, but there only one endpoint of the path ($u$ or $v$) is fixed.
By definition, for a fixed $n$, the events $\CE_{\text{short}}^{(u,v)}(n,t')$ are disjoint.
Moreover, we observe that if $\CE_{\text{short}}^{(u,v)}(n,t')$ holds, then there is a $t'$-present $(u,v)$-path of length $n$ connecting $u$ and $v$ that traverses the vertex $t'$, which is a similar observation to Observation \ref{obs:lower-bound}(4), see Figure \ref{fig:lower-bound}(C).
Using the definitions of the events $\CE_{\text{bad}}$ and $\CE_{\text{short}}$, we can bound the event between brackets in \eqref{eq:lower-bound-graph}, and hence its probability of occurring, as stated in the following lemma.

\begin{lemma}\label{cor:lower-event-splitting}
 Let $(\ell_{k,t'}^\sss{(t)})_{k\geq0,t'\geq t}$ be an array satisfying Assumption \ref{ass:ell}. Then
 \begin{align}
  \P\big(\exists t'\geq t: d_G^{(t')}(u,v) \leq 2\underline{K}_{t,t'}\big)
   & \leq \sum_{q\in\{u,v\}}\sum_{t'=t}^{\infty}\sum_{k=0}^{\underline{K}_{t,t'}}\ind{k\geq2\text{ or }t'=t}\P\big(\CE_{\mathrm{bad}}^{(q)}(k,t')\big) \label{eq:lower-total-prob-bound-a}
  \\ &\hspace{20pt}+ \sum_{t'=t}^\infty\sum_{n=0}^{2\underline{K}_{t,t'}}\ind{n\geq2\text{ or }t'=t}\P\big(\CE_{\mathrm{short}}^{(u,v)}(n,t')\big).\label{eq:lower-total-prob-bound-b}
 \end{align}
 \begin{proof}
  To prove the assertions in the statement, we will first bound the event between brackets on the left-hand side (lhs) in \eqref{eq:lower-total-prob-bound-a}. Eventually, the bound then follows by a union bound.

  \emph{Bounding the events.}
  We write $\sigma(u,v):=\{(u,v),(v,u)\}$. We aim to show that if $(\ell_{k,t'}^\sss{(t)})_{k\geq0,t'\geq t}$ is an array satisfying Assumption \ref{ass:ell}, then
  \begin{align}
   \big\{\exists t'\geq t: d_L^{(t')}(u,v)\leq 2\underline{K}_{t,t'}\big\}
   \subseteq
   \bigcup_{t'=t}^{\infty}\left(\Bigg(\bigcup_{k=0}^{\underline{K}_{t,t'}} \bigcup_{q\in\{u,v\}}\CE_{\mathrm{bad}}^{(q)}(k,t')\Bigg)
   \cup
   \bigcup_{n=0}^{2\underline{K}_{t,t'}} \CE_{\mathrm{short}}^{(u,v)}(n,t')
   \right).\label{eq:lower-dynamic-event}
  \end{align}
  Moreover, for $k,n\geq 2$ and $t'>t$
  \begin{align}
   \CE_{\mathrm{bad}}^{(u)}(k,t')     & \subseteq \bigcup_{x=1}^{\ell_{k,t'}^\sss{(t)}-1}\bigcup_{i=1}^{k-1}\Big\{\big|\big\{u\goodpathto{0}{i}t'\big\}\circ\big\{t'\goodpathto{i}{k}x\big\}\big|\geq 1\Big\},\label{eq:lower-dynamic-bad} \\
   \CE_{\mathrm{short}}^{(u,v)}(n,t') & \subseteq \bigcup_{\substack{(q_1,q_2)                                                                                                                                                             \\\in\sigma(u,v)}}\bigcup_{x=\ell_{\floor{n/2},t'}^\sss{(t)}}^{t'-1}\bigcup_{i=1}^{\floor{n/2}-1}
   \Big\{\big|\big\{q_1\goodpathto{0}{i}t'\big\}\circ\big\{t'\goodpathto{i}{\floor{n/2}}x\big\}\circ\{q_2\goodpathto{0}{\ceil{n/2}} x\big\}\big|\geq 1\Big\}\label{eq:lower-dynamic-short}                                                 \\ &\hspace{20pt}\cup \Big\{\big|\big\{q_1\goodpathto{0}{\floor{n/2}}t'\big\}\circ \big\{q_2\goodpathto{0}{\ceil{n/2}}t'\big\}\big|\ge 1\Big\}. \nonumber
  \end{align}
  We first prove \eqref{eq:lower-dynamic-short}.
  Let $\bm{\pi}$ be any path of length $n\geq2$ whose presence implies $\CE_{\text{short}}^{(u,v)}(n,t')$ for some $t'>t$, so that $\bm{\pi}$ is a $t'$-good $(u,v)$-path by the definition of $\CE_{\text{short}}^{(u,v)}(n,t')$ in \eqref{eq:lower-short-def}. From \eqref{eq:lower-short-def} it also follows that $t'$ is on $\bm{\pi}$, as there was neither a good, nor a bad $(u,v)$-path of length $n$ before time $t'$.
  Thus, the $t'$-good $(u,v)$-path $\bm{\pi}$ can be decomposed in a $t'$-good $u$-path of length $\floor{n/2}$ and a $t'$-good $v$-path of length $\ceil{n/2}$.
  Considering all possible positions of $t'$ on the path, the presence of $\bm{\pi}$ implies the event on the rhs in \eqref{eq:lower-dynamic-short}.
  There, we denoted by $x\neq t'$ the vertex at distance $\floor{n/2}$ from $q_1$ that satisfies the constraint $x\geq \ell_{\floor{n/2},t'}^\sss{(t)}$.
  Thus, $x$ is at distance $\ceil{n/2}$ from $q_1$, and since $j\mapsto\ell_{j,t'}^\sss{(t)}$ is nonincreasing also $x\geq \ell_{\ceil{n/2},t'}^\sss{(t)}$.
  So the inclusion in \eqref{eq:lower-dynamic-short} holds, since $\bm{\pi}$ was an arbitrary path.

  Similarly, let $\bm{\pi}=(q, \dots, \pi_k)$ be any path of length $k\geq 2$ whose presence implies $\CE_{\text{bad}}^{(q)}(k,t')$ for some $t'>t, q\in\{u,v\}$, so that $\pi_k<\ell_{k,t'}^\sss{(t)}$.
  By Observation \ref{obs:lower-bound}(4), vertex $t'$ must be on $\bm{\pi}$ and by a similar reasoning as before we obtain \eqref{eq:lower-dynamic-bad}.

  Lastly, we prove \eqref{eq:lower-dynamic-event} for which we rewrite the lhs as a union over time and paths, i.e.,
  \[
   \big\{\exists t'\geq t: d_L^{(t')}(u,v)\leq 2\underline{K}_{t,t'}\big\}
   =
   \bigcup_{t'=t}^{\infty}\bigcup_{n=0}^{2\underline{K}_{t,t'}}\bigcup_{\substack{(\pi_1,\dots,\pi_{n-1})\\\in[t']^{n-1},\\\text{disjoint}}}\big\{u\leftrightarrow\pi_1\leftrightarrow\cdots\leftrightarrow\pi_{n-1}\leftrightarrow v\big\}.
  \]
  Let $\bm{\pi}:=(\pi_0,\dots,\pi_n)$ be any self-avoiding path from $\pi_0:=u$ to $\pi_n:=v$ in this set.
  The smallest time $t'$ at which $\bm{\pi}$ can be present in the union on the rhs is at $t':=t\vee\max_{i\leq n} \pi_i$.
  Then, $n\leq 2\underline{K}_{t,t'}$ must hold due to the fact that $t'\mapsto \underline{K}_{t,t'}$ is nonincreasing.
  We will show now that the event that $\bm{\pi}$ is $t'$-present is captured in either $\CE_{\text{short}}^{(u,v)}(n,t')$ or $\CE_{\text{bad}}^{(q)}(k,\tilde{t})$ for some $\tilde{t}\leq t', k\leq n/2$,
  $q\in\{u,v\}$.
  For any length $n\geq 0$, if $u\wedge v<\ell_{0,t'}^\sss{(t)}$, then
  \[
   \{\bm{\pi}\text{ present}\}
   \,\subseteq \,
   \CE_{\text{bad}}^{(u)}(0,t)\cup\CE_{\text{bad}}^{(v)}(0,t),
  \]
  since $t'\mapsto \ell_{0,t'}^\sss{(t)}$ is nonincreasing.
  From now on we assume that $u\wedge v\geq\ell_{0,t'}^\sss{(t)}$. If $n\leq 1$, that is when $\{u=v\}$ or $\{u\leftrightarrow v\}$, then $\bm{\pi}$ must already be present at time $t$, i.e.,
  \[
   \{\bm{\pi}\text{ present}\}\subseteq \cup_{i\in\{0,1\}}\CE_{\text{short}}^{(u,v)}(i, t).
  \]
  From now on we assume that the length $n\geq 2$. Moreover, if $\bm{\pi}$ is a $t'$-good path, then
  \[
   \{\bm{\pi}\text{ present}\} \subseteq \CE_{\text{short}}(n,t').
  \]
  Assume $\bm{\pi}$ is not a $t'$-good $(u,v)$-path. Consequently, there is a $t'$-bad path emanating from either $u$ or $v$, which is a \emph{subpath} of $\bm{\pi}$. So, recalling Observation \ref{obs:lower-bound}(1) and (2), the first time that this bad subpath is present, i.e.,
  \[
   \tilde{t}:= \argmin_{\hat{t}\leq t'}\{\exists_{m\leq n/2}: (u,\pi_1,\dots,\pi_{\floor{n/2}})\text{ or }(v,\pi_{n-1},\dots,\pi_{n-\floor{n/2}}) \text{ is }(\hat{t},m)\text{-bad}\}
  \]
  is well-defined and at most $t'$. By Observation \ref{obs:lower-bound}(3), $\bm{\pi}$ is bad at $\tilde{t}$, so that for some $m\leq n/2$
  \[
   \{\bm{\pi}\text{ present}\} \subseteq \cup_{q\in\{u,v\}}\CE_{\text{bad}}^{(q)}(m,\tilde{t}).
  \]

  \emph{Union bound.}
  Having bounded the events between brackets on the lhs in \eqref{eq:lower-total-prob-bound-a} and \eqref{eq:lower-total-prob-bound-b},
  the assertions follow directly from a union bound on the events in \eqref{eq:lower-dynamic-event}. We argue now that the events where one of the indicators in \eqref{eq:lower-total-prob-bound-a} and \eqref{eq:lower-total-prob-bound-b}
  equals zero, happen with probability zero. We start with \eqref{eq:lower-total-prob-bound-b}: $\ind{n\geq 2\text{ or }t'=t}=0$ when both $t'>t$ and $n\in\{0,1\}$. Since no new paths connecting $u$ and $v$ of length one, i.e., a single edge, can be created after time $u\vee v\leq t$ we have that for $t'>t$ and $n\in\{0,1\}$
  \[
   \CE_{\text{short}}^{(u,v)}(n,t')=\varnothing,
  \]
  as by its definition in \eqref{eq:lower-short-def} we require that there was no path of length $n$ before time $t'$.
  Similarly, bad paths of length at most one must already be present at time $t$ since $t'\mapsto\ell_{k,t'}^\sss{(t)}$ is nonincreasing and starts at a value at most $t$. So for $q\in\{u,v\}$, $t'>t$, $k\in\{0,1\}$
  \begin{equation*}
   \CE_{\text{bad}}^{(q)}(k,t')=\varnothing.
  \end{equation*}
 \end{proof}
\end{lemma}

\subsection{Bounding the summands}
\noindent
The main goal of this section is to prove the following lemma for two suitably chosen sequences $k\mapsto\alpha_{[0,k)}$, $\beta_{[0,k)}$, defined below in \eqref{eq:alpha} and \eqref{eq:beta}. It obtains bounds on the individual summands in \eqref{eq:lower-total-prob-bound-a} and \eqref{eq:lower-total-prob-bound-b} in Lemma \ref{cor:lower-event-splitting}.
\begin{lemma}\label{cor:lower-conv}
 Let $k\mapsto\alpha_{[0,k)}$, $\beta_{[0,k)}$,  as in \eqref{eq:alpha} and \eqref{eq:beta} below,  respectively. Then there exists $C>0$ such that for $k\geq 2, n\geq 2$ and $t'>t$, $q\in\{u,v\}$
 \begin{align}
  \P\big(\CE_{\mathrm{bad}}^{(q)}(k,t')\big)     & \leq Ct'^{-1}(k-1)\alpha_{[0,k)}\sum_{x=1}^{\ell_{k,t'}^\sss{(t)}-1}x^{-\gamma}, \label{eq:lower-cor-1}                                                      \\
  \P\big(\CE_{\mathrm{short}}^{(u,v)}(n,t')\big) & \leq 2\beta_{[0,\ceil{n/2})}^2t'^{2\gamma-2} \label{eq:lower-cor-2}                                                                                          \\
                                                 & \hspace{10pt}+  Cnt'^{-1}\sum_{x=\ell_{\floor{n/2},t'}}^{t'-1}\big(\alpha_{[0,\ceil{n/2})}x^{-\gamma} + \beta_{[0,\ceil{n/2})}x^{\gamma-1}\big)^2. \nonumber
 \end{align}
 For $t'=t$, $k,n\geq 1$ and $q\in\{u,v\}$, it holds that
 \begin{align}
  \P\big(\CE_{\mathrm{bad}}^{(q)}(k,t)\big)     & \leq \alpha_{[0,k)}\sum_{x=1}^{\ell_{k,t}^\sss{(t)}-1}x^{-\gamma}, \label{eq:lower-cor-3} \\
  \P\big(\CE_{\mathrm{short}}^{(u,v)}(n,t)\big) & \leq
  \sum_{x=\ell_{\floor{n/2},t}^\sss{(t)}}^{t}\big(\alpha_{[0,\ceil{n/2})}x^{-\gamma} + \beta_{[0,\ceil{n/2})}x^{\gamma-1}\big)^2. \label{eq:lower-cor-4}
 \end{align}
\end{lemma}
We prove the lemma at the end of this section after having established the necessary preliminaries and identified the sequences $k\mapsto\alpha_{[0,k)}$, $\beta_{[0,k)}$.
The decomposition method counting paths that traverse the vertex $t'$ (for $t'>t$) yields a bound   in \eqref{eq:lower-cor-1} and \eqref{eq:lower-cor-2}  that are a factor $1/t'$ smaller than their counterparts with $t'=t$  in \eqref{eq:lower-cor-3} and \eqref{eq:lower-cor-4}. By small refinements of the methods in \cite{dereich2012typical} we obtain that the individual sums on the rhs in \eqref{eq:lower-cor-1} and \eqref{eq:lower-cor-2} are of order $1/\log^3(t')$.
This is why the error terms are summable in $t'$. The extra factor $1/t'$ illustrates the necessity of our decomposition method versus previous methods.

In order to prove Lemma \ref{cor:lower-conv}, it is crucial to understand the probabilities on having self-avoiding paths that are restricted to have specified vertices at some positions, by \eqref{eq:lower-dynamic-bad} and \eqref{eq:lower-dynamic-short}. For this we use the following proposition.
\begin{proposition}[$\mathrm{PA}(\gamma)$ {\cite[Proposition 3.1, 3.2]{dereich2012typical}}]\label{prop:pa-gamma}
 We say that a preferential attachment model satisfies the condition $\mathrm{PA}(\gamma)$, if there is a constant $\nu\in(0,\infty)$ such that for all $t'\in\mathbb{N}$, and pairwise distinct vertices
 $\pi_0, \dots, \pi_k\in[t']$
 \be
 \P(\pi_0\leftrightarrow\cdots\leftrightarrow \pi_k)\leq \prod_{i=1}^k\nu (\pi_k\wedge \pi_{k-1})^{-\gamma}(\pi_k\vee \pi_{k-1})^{\gamma-1} =: p(\pi_0, \dots,\pi_k). \label{eq:pa-gamma}
 \ee
 The above condition is satisfied for PA in Definitions \ref{def:fpa} and \ref{def:vpa} for $\gamma=1/(\tau-1)$. We set $p(\pi_0,\dots,\pi_k):=0$ if the vertices are not pairwise distinct.
\end{proposition}

\noindent For $k>i\ge0$ and a vertex $\pi_{i}\geq\ell_{i,t'}^\sss{(t)}$ and another vertex $\pi_k\in[t']$ we define
\be f_{[i,k)}^\sss{({t,t'})}(\pi_i,\pi_k) :=\sum_{\substack{(\pi_{i+1},\dots,\pi_{k-1})\\\in \CP_{(i,k)}^{\{\pi_i,\pi_k\}}}} p(\pi_{i},\dots,\pi_{k}), \label{eq:lower-f}
\ee
where for a vertex set $\CV\subset[t']$,
$\CP_{(i,k)}^{\CV}$ denotes the set of pairwise disjoint vertex tuples $(\pi_{i+1},\dots,\pi_{k-1})$ such that $\pi_j\geq\ell_{j,t'}^\sss{(t)}$, $\pi_j\notin \CV$ for
all $i<j<k$. Intuitively, $f^\sss{(t,t')}_{[i,k)}(\pi_i,\pi_k)$ is an upper bound for the expected number of $t'$-good paths on the segment $[i,k)$ from $\pi_i$ to $\pi_k$.

We derive upper bounds for the summands in \eqref{eq:lower-total-prob-bound-a} and \eqref{eq:lower-total-prob-bound-b} in terms of $f^\sss{(t,t')}_{[i,k)}$.
\begin{claim}
 Consider the preferential attachment model with power-law parameter $\tau>2$. Let $(\ell_{k,t'}^\sss{(t)})_{k\geq0,t'\geq t}$ be an array satisfying Assumption \ref{ass:ell}. Then for $k,n\geq 2$ and $t'>t$,
 \begin{align}
  \P\big(\CE_\mathrm{bad}^{(q)}(k,t')\big)     & \leq \sum_{x=1}^{\ell_{k,t'}^\sss{(t)}-1}\sum_{i=1}^{k-1}f_{[0,i)}^\sss{(t,t')}(q,t')f_{[i,k)}^\sss{(t,t')}(t',x), \label{eq:lower-prob1-t-upper}                                                                       \\
  \P\big(\CE_\mathrm{short}^{(u,v)}(n,t')\big) & \leq  \sum_{(q_1,q_2)\in\sigma(u,v)}\sum_{x=\ell_{\floor{n/2},t'}^\sss{(t)}}^{t'-1}\sum_{i=1}^{\floor{n/2}-1}f_{[0,\ceil{n/2})}^\sss{(t,t')}(q_1,x)f_{[0,i)}^\sss{(t,t')}(q_2,t')f_{[i,k)}^\sss{(t,t')}(t',x) \nonumber \\
                                               & \hspace{20pt}+ f_{[0,\ceil{n/2})}^\sss{(t,t')}(q_1,t')f_{[0,\floor{n/2})}^\sss{(t,t')}(q_2,t'), \label{eq:lower-prob2-t-upper}
 \end{align}
 while for any $k,n\geq1$ and $t'=t$
 \begin{align}
  \P\big(\CE_\mathrm{bad}^{(q)}(k,t)\big)     & \leq \sum_{x=1}^{\ell_{k,t}^\sss{(t)}-1}f_{[0,k)}^\sss{(t,t)}(q,x), \label{eq:lower-prob1-t-upper-base} \\
  \P\big(\CE_\mathrm{short}^{(u,v)}(n,t)\big) & \leq
  \sum_{(q_1,q_2)\in\sigma(u,v)} \sum_{x=\ell_{\floor{n/2},t'}^\sss{(t)}}^{t}f_{[0,\ceil{n/2})}^\sss{(t,t)}(q_1,x)f_{[0,\floor{n/2})}^\sss{(t,t)}(q_2,x).\label{eq:lower-prob2-t-upper-base}
 \end{align}
 \begin{proof}
  Recall the set of paths $\{\pi_i\goodpathto{i}{k}\pi_k\}$ from Definition \ref{def:good-paths-segment}.
  Then by Markov's inequality, \eqref{eq:lower-f}, and Proposition \ref{prop:pa-gamma}
  \begin{align*}
   \P\big(\big|\{\pi_i\goodpathto{i}{k}\pi_k\}\big|\geq 1\big) \leq \E[|\{x\goodpathto{i}{k}\pi_k\}|] & = \sum_{\substack{(\pi_{i+1},\dots,\pi_{k-1})   \\\in \CP_{(i,k)}^{\{\pi_i,\pi_k\}}}}\P(\pi_{i}\leftrightarrow\dots\leftrightarrow\pi_{k}) \\
                                                                                                      & \leq\sum_{\substack{(\pi_{i+1},\dots,\pi_{k-1}) \\\in \CP_{(i,k)}^{\{\pi_i,\pi_k\}}}} p(\pi_{i},\dots,\pi_{k}) = f_{[i,k)}^\sss{(t,t')}(\pi_i,\pi_k).
  \end{align*}
  Now for concatenated paths, due to the product structure in \eqref{eq:pa-gamma}, and by relaxing the disjointness of sets, we have
  \begin{align}
   \P\big(\big|\{\pi_0\goodpathto{0}{i}\pi_i\}\circ \{\pi_i\goodpathto{i}{k}\pi_k\}\big|\ge 1\big)
    & \leq
   \sum_{\substack{(\pi_{1},\dots,\pi_{i-1})                                                \\\in \CP_{(0,i)}^{\{\pi_0,\pi_i,\pi_k\}}}}
   \sum_{\substack{(\pi_{i+1},\dots,\pi_{k-1})                                              \\\in \CP_{(i,k)}^{\{\pi_0,\dots,\pi_{i}\}}}} p(\pi_0,\dots,\pi_{i})p(\pi_{i},\dots,\pi_{k}) \nonumber\\
    & \leq f^\sss{(t,t')}_{[0,i)}(\pi_0,\pi_i)f^\sss{(t,t')}_{[i,k)}(\pi_i,\pi_k).\nonumber
  \end{align}
  Recall now \eqref{eq:lower-dynamic-bad}, so that \eqref{eq:lower-prob1-t-upper} follows by a union bound and choosing $\pi_0=q, \pi_i=t'$, and $\pi_k=x$. Similarly \eqref{eq:lower-prob2-t-upper} follows
  by union bounds over the rhs in \eqref{eq:lower-dynamic-short}. The bounds \eqref{eq:lower-prob1-t-upper-base} and \eqref{eq:lower-prob2-t-upper-base} follow analogously from their definition in \eqref{eq:lower-bad-def-t} and \eqref{eq:lower-short-def-t}.
 \end{proof}
\end{claim}
\noindent  We establish recursive bounds on $f^\sss{(t,t')}_{[i,k)}$ in the spirit of \cite[Lemma 1]{dereich2012typical}.
Let $(\ell_{k,t'}^\sss{(t)})_{k\geq0,t'\geq t}$ be an array satisfying Assumption \ref{ass:ell} such that $\eta_{j,t'}:=(t'/\ell_{j,t'}^\sss{(t)})\geq \re$ for all $j\geq0$ and $t'\geq t$. Define  for $\gamma:=1/(\tau-1)$ and some $c>1$
\begin{align}
 \alpha^\sss{(t')}_{[0, j)} :=
 \begin{dcases}
  \nu\ell_{0, t'}^{\gamma-1}                                                                              & j=1, \\
  c\big(\alpha^\sss{(t')}_{[0, j-1)}\log(\eta_{j-1, t'}) + \beta^\sss{(t')}_{[0, j-1)}t'^{2\gamma-1}\big) & j>1,
 \end{dcases} \label{eq:alpha} \\
 \beta^\sss{(t')}_{[0, j)} :=
 \begin{dcases}
  \nu\ell_{0, t'}^{-\gamma}                                                                                           & j=1, \\
  c\big(\alpha^\sss{(t')}_{[0, j-1)}\ell_{j-1, t'}^{1-2\gamma} + \beta^\sss{(t')}_{[0, j-1)}\log(\eta_{j-1, t'})\big) & j>1,
 \end{dcases} \label{eq:beta}
\end{align}
similar to the recursions in \cite[Lemma 1]{dereich2012typical}.
The sequence $\big(\alpha^\sss{(t')}_{[0, j)}\big)_{j\geq1}$ is related to the expected number of self-avoiding $t'$-good paths $(\pi_0,\dots,\pi_j)$ of length $j$ from $\pi_0\in\{u,v\}$ to $\pi_j$ such that $\pi_{j-1}>\pi_j$.
The sequence $\big(\beta^\sss{(t')}_{[0, j)}\big)_{j\geq1}$ is related to those paths where $\pi_{j-1}<\pi_j$.
Observe that since $c>1$, $\eta_{j,t'}\geq \re$, and $\alpha_{[0,1)}^\sss{(t)},\beta_{[0,1)}^\sss{(t)}\geq 0$, it follows that
$k\mapsto\alpha_{[0,k)}^\sss{(t)}$ and $k\mapsto\beta_{[0,k)}^\sss{(t)}$ are non-decreasing.
We define for the same constant $c>1$ the non-decreasing sequences
\begin{align}
 \phi^\sss{(t')}_{[i, i+j)} & :=
 \begin{dcases}
  \nu t'^{\gamma - 1}                                                                                        & j=1, \\
  c\big(\phi^\sss{(t')}_{[i, i+j-1)}\log(\eta_{i+j-1, t'}) + \psi^\sss{(t')}_{[i, i+j-1)}t'^{2\gamma-1}\big) & j>1,
 \end{dcases} \label{eq:phi} \\
 \psi^\sss{(t')}_{[i, i+j)} & :=
 \begin{dcases}
  0                                                                                                                        & j=1, \\
  c\big(\phi^\sss{(t')}_{[i, i+j-1)}\ell_{i+j-1, t'}^{1-2\gamma} + \psi^\sss{(t')}_{[i, i+j-1)}\log(\eta_{i+j-1, t'})\big) & j>1.
 \end{dcases}\label{eq:psi}
\end{align}
These sequences are related to the $t'$-good paths emanating from $t'$ that are good on the segment $[i,i+j)$. Observe that the recursions are identical to \eqref{eq:alpha} and \eqref{eq:beta}, except that their initial values are different. This is crucial to give summable error bounds in $t'$ later on. Below, we leave out the superscript $(t')$ for notational convenience, but we stress here that these four sequences are dependent on both $t'$ and $t$.
\begin{claim}[Recursive bounds for number of paths]\label{lemma:number-of-paths}
 Under the same assumptions as Proposition \ref{proposition:lower-bound-graph},
 let $(\ell_{k,t'}^\sss{(t)})_{k\geq0,t'\geq t}$ be an array satisfying Assumption \ref{ass:ell}. Let $\eta_{j, t'}=t'/\ell_{j, t'}^\sss{(t)}$ and $\gamma=1/(\tau-1)$.
 For sufficiently large $c=c(\tau), \nu=\nu(\tau)$ in \eqref{eq:alpha}, \eqref{eq:beta}, \eqref{eq:phi}, and \eqref{eq:psi}, it holds that
 \be
 f^\sss{(t,t')}_{[i,i+j)}(t',x)
 \leq x^{-\gamma}\phi_{[i, i+j)} + \ind{x>\ell_{i+j-1, t'}^\sss{(t)}}x^{\gamma-1}\psi_{[i, i+j)}.\label{eq:lower-paths-from-t}
 \ee
 Moreover,
 \begin{equation}
  f^\sss{(t,t')}_{[0,j)}(q,x)
  \leq \ind{x<t'}x^{-\gamma}\alpha_{[0, j)} + \ind{x>\ell_{j-1, t'}^\sss{(t)}}x^{\gamma-1}\beta_{[0, j)}. \label{eq:lower-paths-from-x}
 \end{equation}
\end{claim}
\noindent
We refer to the appendix for the proof, which follows by induction from arguments analogous to \cite[Lemma 1]{dereich2012typical}.
As a consequence of \eqref{eq:lower-paths-from-x}, we have for $q\in\{u,v\}$
\[
 f_{[0,i)}^\sss{(t,t')}(q,t') \leq \ind{t'<t'}t'^{-\gamma}\alpha_{[0, j)} + \ind{t'>\ell_{j-1, t'}^\sss{(t)}}t'^{\gamma-1}\beta_{[0, j)} = t'^{\gamma-1}\beta_{[0, j)}.
\]
Moreover, since $x<\ell_{k,t'}^\sss{(t)}$ implies that also $x<\ell_{k-1,t'}^\sss{(t)}$ since $k\mapsto \ell_{k,t'}^\sss{(t)}$ is nonincreasing, for $x<\ell_{k,t'}^\sss{(t)}$ it follows from \eqref{eq:lower-paths-from-t} that
\[
 f_{[i,k)}^\sss{(t)}(t', x) \leq x^{-\gamma}\phi_{[i,i+j)} + \ind{x>\ell_{k-1,t'}^\sss{(t)}}\psi_{[i,k)} = x^{-\gamma}\phi_{[i,i+j)}.
\]
Hence, we can bound the summands in \eqref{eq:lower-total-prob-bound-a} using Claim \ref{lemma:number-of-paths} to obtain for $k\geq 2$, $t'>t$,
\be
\P(\CE_{\mathrm{bad}}^{(q)}(k,t'))
\leq
\sum_{x=1}^{\ell_{k,t'}^\sss{(t)}-1}\sum_{i=1}^{k-1}f_{[0,i)}^\sss{(t,t')}(q, t')f_{[i,k)}^\sss{(t,t')}(t',x)
\leq
t'^{\gamma-1}\sum_{x=1}^{\ell_{k,t'}^\sss{(t)}-1}x^{-\gamma} \sum_{i=1}^{k-1}\beta_{[0,i)}\phi_{[i,k)}. \label{eq:lower-error-1-bound-1}
\ee
Similarly to \eqref{eq:lower-error-1-bound-1} we bound the summands in \eqref{eq:lower-total-prob-bound-b} from above using \eqref{eq:lower-prob2-t-upper} and \emph{replacing} the first sum over the permutation $\sigma(u,v)$ in \eqref{eq:lower-prob2-t-upper} by a factor two, i.e., for $n\geq2$ and $t'>t$,
\begin{align}
 \P\big( & \CE_{\text{short}}^{(u,v)}(n,t')\big)
 \leq
 2 t'^{\gamma-1}\sum_{x=\ell_{\floor{n/2},t'}}^{t'}\big(\ind{x<t'}\alpha_{[0,\ceil{n/2})}x^{-\gamma} + \beta_{[0,\ceil{n/2})}x^{\gamma-1}\big) \nonumber                                                                                         \\
         & \cdot\Big(\ind{x=t'}\beta_{[0,\floor{n/2})}+\ind{x<t'}\sum_{i=1}^{\floor{n/2}-1}\left(\beta_{[0,i)}\phi_{[i,\floor{n/2})}x^{-\gamma} +  \beta_{[0,i)}\psi_{[i,\floor{n/2})}x^{\gamma-1}\right)\Big). \label{eq:lower-error-2-bound-1}
\end{align}
Both \eqref{eq:lower-error-1-bound-1} and \eqref{eq:lower-error-2-bound-1} contain convolutions of the sequence $\beta_{[0,i)}$ with $\phi_{[i,k)}$ and $\psi_{[i,k)}$. This motivates to bound these convolutions in terms of the \emph{original} sequences $\alpha_{[0,k)}$ and $\beta_{[0,k)}$.
\begin{claim}\label{claim:convolution}
 Let $\phi_{[i,k)}, \psi_{[i,k)}, \alpha_{[0,k)}, \beta_{[0,k)}$ be as in \eqref{eq:phi}, \eqref{eq:psi}, \eqref{eq:alpha}, \eqref{eq:beta}, respectively.
 Then there exists $C>0$ such that for $k\geq 2$
 \begin{align}
  B_k^{\psi} & := \sum_{i=1}^{k-1}\beta_{[0,i)}\psi_{[i,k)} \leq C\left(k-2\right) \beta_{[0,k)}t'^{-\gamma} \label{eq:bk-psi}, \\
  B_k^{\phi} & := \sum_{i=1}^{k-1}\beta_{[0,i)}\phi_{[i,k)} \leq C(k-1) \alpha_{[0,k)}t'^{-\gamma}.\label{eq:bk-phi}
 \end{align}
 \begin{proof}
  We prove by induction.
  We initialize the induction for $k=2$, the smallest value of $k$ for which the sums in \eqref{eq:bk-phi} and \eqref{eq:bk-psi} are non-empty. Indeed, then \eqref{eq:bk-psi} holds  by the initial value of $\psi_{[i,i+1)}=0$ in \eqref{eq:psi}, i.e.,
  \[
   B_2^{\psi} = \beta_{[0,1)}\psi_{[1,2)} = \beta_{[0,1)}\cdot 0 \leq C\cdot 0\cdot \beta_{[0,2)}t'^{-\gamma}.
  \]
  For $k=2$ in \eqref{eq:bk-phi} we substitute the recursion \eqref{eq:alpha} on $\alpha_{[0,2)}$. Thus, we have to show that
  \[
   B_2^{\phi} = \beta_{[0,1)}\phi_{[1,2)} \leq cC\big(\alpha_{[0,1)}\log(\eta_{1,t'}) + \beta_{[0,1)}t'^{2\gamma-1}\big)t'^{-\gamma}.
  \]
  Using the initial values in \eqref{eq:alpha}, \eqref{eq:beta}, and \eqref{eq:phi}, this is indeed true for $C\geq \nu/c$, i.e.,
  \begin{align*}
   B_2^{\phi} =
   \nu\ell_{0,t'}^{-\gamma}\cdot \nu t'^{\gamma-1}
   \leq
   cC\big(\nu\ell_{0,t'}^{\gamma-1}\log(t'/\ell_{1,t'})t'^{-\gamma} + \nu\ell_{0,t'}^{-\gamma}t'^{\gamma-1}\big).
  \end{align*}
  Now, we advance the induction.
  To this end, one can derive the following recursions using \eqref{eq:phi} and \eqref{eq:psi}:
  \begin{align}
   B_{k+1}^{\psi} & = cB_k^{\phi}\ell_{k,t'}^{1-2\gamma} + cB_k^{\psi}\log(\eta_{k,t'}),                         & B_2^{\psi} & =0 \label{eq:induction-psi},                                       \\
   B_{k+1}^{\phi} & = \nu t'^{\gamma-1}\beta_{[0,k)} + c\log(\eta_{k,t'})B_k^{\phi} + cB_k^{\psi}t'^{2\gamma-1}, & B_2^{\phi} & =\nu^2 t'^{\gamma-1}\ell_{0,t'}^{-\gamma}.\label{eq:induction-phi}
  \end{align}
  The first term in \eqref{eq:induction-phi} is a result of the \emph{non-zero} initial value of $\phi_{[i,i+1)}$ in \eqref{eq:phi}, while $\psi_{[i,i+1)}=0$, so that there is no such term in \eqref{eq:induction-psi}.
  Since the two recursions depend only on each other's previous values, we can carry out the two induction steps simultaneously. By the two induction hypotheses \eqref{eq:bk-psi} and \eqref{eq:bk-phi}, and the definition of $\beta_{[0,k+1)}$ in \eqref{eq:beta}, we have that
  \begin{align}
   B_{k+1}^{\psi} = cB_k^{\phi}\ell_{k,t'}^{1-2\gamma} + cB_k^{\psi}\log(\eta_{k,t'})
    & \leq Cc\big((k-1)\alpha_{[0,k)}\ell_{k,t'}^{1-2\gamma} + (k - 2)\beta_{[0,k)}\log(\eta_{k,t'})\big)t'^{-\gamma} \nonumber \\
    & \leq C(k-1)\beta_{[0,k+1)}t'^{-\gamma}, \nonumber
  \end{align}
  proving \eqref{eq:induction-psi}. For \eqref{eq:induction-phi}, we assume that $cC\geq \nu$ so that using the induction hypotheses and \eqref{eq:alpha} the proof is finished, i.e.,
  \begin{align}
   B_{k+1}^\phi & = c\log(\eta_{k,t'})B_k^{\phi} + cB_k^{\psi}t'^{2\gamma-1} + \nu t'^{\gamma-1}\beta_{[0,k)}\nonumber                                               \\
                & \leq c\log(\eta_{k,t'})C(k-1) \alpha_{[0,k)}t'^{-\gamma} + ct'^{\gamma-1}C\left(k-2\right) \beta_{[0,k)} + \nu t'^{\gamma-1}\beta_{[0,k)}\nonumber \\
                & \leq C (k-1)\alpha_{[0,k+1)}t'^{-\gamma}. \nonumber
  \end{align}
 \end{proof}
\end{claim}
We combine Claim \ref{claim:convolution} with  \eqref{eq:lower-error-1-bound-1} and \eqref{eq:lower-error-2-bound-1} to arrive to the proof of Lemma \ref{cor:lower-conv}.

\begin{proof}[Proof of Lemma \ref{cor:lower-conv}]
 We start with \eqref{eq:lower-cor-1}. Recall for $q\in\{u,v\}$ the bound on $\P\big(\CE_{\text{bad}}^{(q)}(k,t')\big)$ in \eqref{eq:lower-error-1-bound-1} and observe that \eqref{eq:bk-phi} implies \eqref{eq:lower-cor-1}, since there is $C>0$ such that for $k\geq2$
 \begin{align}
  \P\big(\CE_{\text{bad}}^{(q)}(k,t')\big)\leq  t'^{\gamma-1}\sum_{i=1}^{k-1}\beta_{[0,i)}\phi_{[i,k)}\sum_{x=1}^{\ell_{k,t'}^{\sss{(t)}}-1}x^{-\gamma}
  \leq C(k-1)t'^{-1}\alpha_{[0,k)}\sum_{x=1}^{\ell_{k,t'}^{\sss{(t)}}-1}x^{-\gamma}. \nonumber
 \end{align}
 For \eqref{eq:lower-cor-2}, we recall the bound \eqref{eq:lower-error-2-bound-1} and bound using \eqref{eq:bk-phi} and \eqref{eq:bk-psi} the factor on the second line in  \eqref{eq:lower-error-2-bound-1} by
 \[\ind{x=t'}\beta_{[0,\floor{n/2})} + \ind{x<t'}Ct'^{-\gamma}\big((\floor{n/2}-1)\alpha_{[0,\floor{n/2})}x^{-\gamma} + (\floor{n/2}-2)\beta_{[0,\floor{n/2})}x^{\gamma-1}\big).\]
 Now \eqref{eq:lower-cor-2} follows by distinguishing the summands in \eqref{eq:lower-error-2-bound-1} between $x<t'$ and $x=t'$, and using that $j\mapsto\alpha_{[0,j)}$ and $j\mapsto\beta_{[0,j)}$ are non-decreasing so that we may round up their indices to $\ceil{n/2}$ to obtain the square.
 Lastly, the bounds \eqref{eq:lower-cor-3} and \eqref{eq:lower-cor-4} follow directly from \eqref{eq:lower-prob1-t-upper-base}, \eqref{eq:lower-prob2-t-upper-base}, and \eqref{eq:lower-paths-from-x}, where we again round up the indices to obtain the square.
\end{proof}

\subsection{Setting the birth-threshold sequence}
After the event decomposition in Lemma \ref{cor:lower-event-splitting} and the bounds on the individual summands in Lemma \ref{cor:lower-conv}, we are ready to choose the birth-threshold array $(\ell_{k, t'}^\sss{(t)})_{k\geq0,t'\geq t}$ to ensure that the sums in \eqref{eq:lower-cor-1}, \eqref{eq:lower-cor-2}, \eqref{eq:lower-cor-3}, and \eqref{eq:lower-cor-4} are sufficiently small.
The right choice of $(\ell_{k, t'}^\sss{(t)})_{k\geq0,t'\geq t}$ will make the error probabilities in \eqref{eq:lower-total-prob-bound-a} and \eqref{eq:lower-total-prob-bound-b} arbitrarily small.
Fix $\delta'=\delta'(\delta)>0$ that we choose later to be sufficiently small. We define
\begin{numcases}{\ell_{k, t'}^\sss{(t)}:=\label{eq:ell-total}}
 \lceil\delta' t\rceil & $k=0$, \label{eq:ell-init}\\
 \argmax_{x\in\mathbb{N}\backslash\{0,1\}}\left\{\alpha_{[0,k)} x^{1-\gamma}\overset{(\circledast)}{\leq} \left(k\log(t')\right)^{-3}\right\} & $k\geq 1$.
 \label{eq:ell}
\end{numcases}
Since $k\mapsto \alpha_{[0,k)}$ is non-decreasing and $1-\gamma>0$ by \eqref{eq:ell}, $k\mapsto\ell_{k,t'}^\sss{(t)}$ must be nonincreasing in both indices.
Using the upper bound on $t'/\ell_{k,t'}^\sss{(t)}$ in Lemma \ref{lemma:eta-k-appendix} in the appendix, one can verify
that $\ell_{k,t'}^\sss{(t)}\geq 2$ for all $k\leq\underline{K}_{t,t'}$ if $t$ is sufficiently large. Hence, the array $(\ell_{k,t'}^\sss{(t)})_{k\geq 0,t'\geq t}$ is well-defined.
The choice of $(\ell_{k,t'}^\sss{(t)})_{k\geq 0,t'\geq t}$ in \eqref{eq:ell-total} is similar to the choice in \cite[Proof of Theorem 2]{dereich2012typical} for $t'=t$.
The main difference is the extra $\log^{-3} (t')$ factor on the rhs in \eqref{eq:ell}. This factor, in combination with the $1/t'$-factor from Lemma \ref{cor:lower-conv} yields a summable error in $t'$ in
\eqref{eq:lower-total-prob-bound-a} and \eqref{eq:lower-total-prob-bound-b}.
We comment that the additional $\log^{-3} (t')$ factor could be changed to another slowly varying function, but the choice has to be $o((t')^\epsilon)$ for all $\epsilon>0$, otherwise the entries of $(\ell_{k,t'}^\sss{(t)})$ would not be at least two, whence the array would be ill-defined.

We are ready to prove Proposition \ref{proposition:lower-bound-graph}.
\begin{proof}[Proof of Proposition \ref{proposition:lower-bound-graph}]
 To prove \eqref{eq:lower-bound-graph}, due to Corollary \ref{cor:lower-event-splitting},
 we need to show that the rhs in \eqref{eq:lower-total-prob-bound-a} and \eqref{eq:lower-total-prob-bound-b} is at most $\delta$,  for $t$ sufficiently large. To keep notation light, we write $\ell_{k,t'}:=\ell_{k,t'}^\sss{(t)}$.
 First, we consider the terms in \eqref{eq:lower-total-prob-bound-a} where $t'=t$.
 Recalling the definition of $\CE_{\text{bad}}^{(q)}(0,t)$ from \eqref{eq:lower-bad-def-init} and the upper bound on its probability in \eqref{eq:lower-cor-3}, we have for $q\in\{u,v\}$
 \begin{equation}
  \sum_{k=0}^{\underline{K}_{t,t}}\P\big(\CE_{\text{bad}}^{(q)}(k, t)\big)
  \leq
  \delta' + O(1/t) +
  \sum_{k=1}^{\underline{K}_{t,t}}\alpha_{[0,k)}\sum_{x=1}^{\ell_{k,t}-1}x^{-\gamma}, \label{eq:lower-proof-t1}
 \end{equation}
 where the term $\delta'+O(1/t)$ comes from the probability that $q$, the uniform vertex in $[t]$, is born before $\ell_{0,t}= \ceil{\delta't}$.
 Now approximating the last sum in \eqref{eq:lower-proof-t1} by an integral and using $(\circledast)$ in \eqref{eq:ell}, we have for some $c_1>0$, $q\in\{u,v\}$
 \begin{align}
  \sum_{k=0}^{\underline{K}_{t,t}}\P\big(\CE_{\text{bad}}^{(q)}(k, t)\big)
   & \leq
  \delta' + o(1) + c_1\sum_{k=1}^{\underline{K}_{t,t}}\alpha_{[0,k)}\ell_{k,t}^{1-\gamma} \nonumber                         \\
   & = \delta' + o(1) + c_1\log^{-3}(t)\sum_{k=1}^{\underline{K}_{t,t}} k^{-3} = \delta' + o(1).  \label{eq:lower-proof-t2}
 \end{align}
 We move on to the terms on the rhs in \eqref{eq:lower-total-prob-bound-a} for $t'>t$ and show that their sum is of order $O(\delta')$.
 Recall for $q\in\{u,v\}$ the bound on $\P\big(\CE_{\text{bad}}^{(q)}(k,t')\big)$ in \eqref{eq:lower-cor-1},
 and observe that there is $C'>0$ such that, approximating the sum over $x$ in \eqref{eq:lower-cor-1} by an integral gives for $t'>t$, $k\geq2$
 \begin{align*}
  \P\big(\CE_{\text{bad}}^{(q)}(k,t')\big) \leq C't'^{-1}(k-1)\alpha_{[0,k)}\ell_{k,t'}^{1-\gamma}
  \leq \frac{C'}{k^2t'\log^3(t')}.
 \end{align*}
 The last inequality follows from $(\circledast)$ in \eqref{eq:ell}.
 The rhs is summable in $k$ and $t'$ so that, only considering the tail of the sum,
 \begin{align}
  \sum_{t'=t+1}^{\infty}\sum_{k=2}^{\underline{K}_{t,t'}}\P\big(\CE_{\text{bad}}^{(q)}(k,t')\big)
  =O\big(\log^{-2}(t)\big).\label{eq:lower-error-1-bound-final}
 \end{align}
 Combining  \eqref{eq:lower-proof-t2} and \eqref{eq:lower-error-1-bound-final}, this establishes that the rhs in \eqref{eq:lower-total-prob-bound-a} is at most $2\delta' + o(1)$ for $t$ sufficiently large, when summed over $q\in\{u,v\}$.

 We continue by proving that the summed error probability in \eqref{eq:lower-total-prob-bound-b} is small. First we consider the terms where $t'>t$. Recall \eqref{eq:lower-cor-2}.
 We use now that $(a+b)^2\leq 2(a^2+b^2)$ for $a,b>0$, so that there exists $C'>0$ such that
 \begin{align}
  \P\big(\CE_{\text{short}}^{(u,v)}(n,t')\big)
   & \leq
  2t'^{2\gamma-2}\beta_{[0,\ceil{n/2})}^2+ \frac{C'n}{t'}\beta_{[0,\ceil{n/2})}^2 \sum_{x=\ell_{\floor{n/2},t'}}^{t'-1}x^{2\gamma-2} \nonumber \\
   & \hspace{30pt} + \frac{C'n}{t'} \alpha_{[0,\ceil{n/2})}^2\sum_{x=\ell_{\floor{n/2},t'}}^{t'-1}x^{-2\gamma}\nonumber                        \\
   & =: T_{11}(n,t') + T_{12}(n,t') + T_{2}(n,t').\label{eq:lower-short-1}
 \end{align}
 Approximating the sums by integrals and using that $k\mapsto\ell_{k,t'}$ is nonincreasing, there exists a different $C'>2$ such that, relaxing the first two terms in \eqref{eq:lower-short-1},
 \begin{align*}
  T_{11}(n,t') + T_{12}(n,t')\leq 2C'n\beta_{[0,\ceil{n/2})}^2 t'^{2\gamma-2}, \qquad
  T_2(n,t') \leq C'n \alpha^2_{[0,\ceil{n/2})}\ell_{\ceil{n/2},t'}^{1-2\gamma}t'^{-1}.
 \end{align*}
 For $T_1:=T_{11}+T_{12}$, by \eqref{eq:alpha}
 it holds that $c\beta_{[0,\ceil{n/2})}\leq \alpha_{[0,\ceil{n/2}+1)}t'^{1-2\gamma}$, yielding by $(\circledast)$ in \eqref{eq:ell}
 \begin{align}
  T_1(n,t') & \leq 2C'n\alpha^2_{[0,\ceil{n/2}+1)}t'^{-2\gamma}/c^2 = 2C'n\Big(\alpha_{[0,\ceil{n/2}+1)} \ell_{\ceil{n/2}+1,t'}^{1-\gamma}\Big)^2 \left(t'/\ell_{\ceil{n/2}+1,t'}\right)^{2-2\gamma} (ct')^{-2} \nonumber \\
            & \leq \frac{2C'n}{(\ceil{n/2}+1)^6\log^6(t')(ct')^2}\left(t'/\ell_{\ceil{n/2}+1,t'}\right)^{2-2\gamma}.\label{eq:lower-short-3}
 \end{align}
 Rewriting $T_2$ similarly,
 \be
 T_2(n,t')\leq C'n\Big(\alpha_{[0,\ceil{n/2})}\ell_{\ceil{n/2},t'}^{1-\gamma}\Big)^2\big(t'/\ell_{\ceil{n/2},t'}\big)t'^{-2}
 \leq \frac{C'n}{\ceil{n/2}^6\log^6(t')t'^2}\left(t'/\ell_{\ceil{n/2},t'}\right).\label{eq:lower-short-4}
 \ee
 Recall that $\ell_{\ceil{n/2}+1,t'}\geq 2$ as mentioned after \eqref{eq:ell}. Thus, both \eqref{eq:lower-short-3} and \eqref{eq:lower-short-4} are summable in $t'$ and $n$. They tend to zero as $t$ tends to infinity, using for \eqref{eq:lower-short-3} that $2-2\gamma<1$.

 It is left to verify that the terms where $t'=t$ in \eqref{eq:lower-total-prob-bound-b} are of order $O(\delta')$ when summed over $n$.
 For this the same reasoning holds as above, starting after \eqref{eq:lower-error-1-bound-final}, where the initial bound is the one in \eqref{eq:lower-cor-4} instead of \eqref{eq:lower-cor-2}. Here, all terms are a factor $t'=t$ larger than before. This yields that
 \be
 \sum_{n=0}^{\underline{K}_{t,t}}\P\big(\CE_{\text{short}}^{(u,v)}(n,t)) = o(1) + \frac{C'}{\log^6(t)t}(t/\ell_{\underline{K}_{t,t}+1,t'})\sum_{n=1}^{\underline{K}_{t,t}}\ceil{n/2}^{-6} = O(\log^{-6}(t))=o(1). \nonumber
 \ee
 Recalling the conclusions after \eqref{eq:lower-error-1-bound-final} and \eqref{eq:lower-short-4}, we conclude that the error terms in \eqref{eq:lower-total-prob-bound-a} and \eqref{eq:lower-total-prob-bound-b} are of order $O(\delta')$, so that \eqref{eq:lower-bound-graph} follows by Corollary \ref{cor:lower-event-splitting}, when $\delta'$ is chosen sufficiently small so that the error probabilities are at most $\delta$, and $\eta_{0,t'}=t/\ell_{0,t}\geq \re$ as required by the definition of $\alpha_{[0,j)}$ and $\beta_{[0,j)}$ before \eqref{eq:alpha}.
\end{proof}
\subsection{Extension to weighted distances}
We extend the result on graph distances from Proposition \ref{proposition:lower-bound-graph} to weighted distances, refining \cite{jorkom2019weighted}.
For this we introduce the graph neighbourhood and its boundary.
\begin{definition}[Graph neighbourhoods]
 Let $x$ be a vertex in $\mathrm{PA}_t$. Its graph neighbourhood of radius $R\in\mathbb{N}$ at time $t$, denoted by $\CB_G^{(t)}(x,R)$, and its neighbourhood  boundary, $\partial\CB_G^{(t)}(x,R)$, are defined as
 \begin{equation*}
  \CB_G^{(t)}(x,R):=\{y\in [t]: d_G^{(t)}(x,y)\leq R\}, \qquad \partial\CB_G^{(t)}(x,R):=\{y\in [t]: d_G^{(t)}(x,y)= R\}.\label{eq:graph-neighbourhood}
 \end{equation*}
\end{definition}
\noindent Recall $Q_{t,t'}$ from \eqref{eq:qt-kt}.
\begin{proposition}[Lower bound weighted distance]\label{proposition:lower-bound-weighted}
 Consider the preferential attachment model with power-law exponent $\tau\in(2,3)$. Equip every edge upon creation with an i.i.d.\ copy of the non-negative random variable $L$. Let $u, v$ be two typical vertices in $\mathrm{PA}_t$.
 Then for any $\delta>0$, there exists $M_L>0$ such that
 \be
 \P\big(\exists t'\geq t: d_L^{(t')}(u,v) \leq 2Q_{t,t'}-2M_L\big) \leq \delta.\nonumber
 \ee
\end{proposition}
\begin{proof}
 Fix $\delta'$ sufficiently small. Define
 \be
 \CE_{\text{good}}(t) := \bigcap_{t'=t}^{\infty}\Bigg(\bigcap_{k=0}^{\underline{K}_{t,t'}}\bigcap_{q\in\{u,v\}}\neg \CE_{\text{bad}}^{(q)}(k,t')\Bigg)
 \cap
 \Bigg(\bigcap_{n=0}^{2\underline{K}_{t,t'}}\neg \CE_{\text{short}}^{(u,v)}(n,t')\Bigg),\nonumber
 \ee
 for $\underline{K}_{t,t'}:=K_{t,t'}-M_G$, where $M_G>0$ is such that the above event holds with probability at least $1-\delta'$ by the proof of Proposition \ref{proposition:lower-bound-graph}.
 Define the conditional probability measure  $\P_{\mathrm{g}}(\cdot):=\P(\,\cdot\mid\CE_{\text{good}}(t))$.
 On the event $\CE_{\text{good}}(t)$, $d_G^{(t')}(u,v)>2\underline{K}_{t,t'}$ for all $t'\geq t$. Hence, at all times $t'\geq t$ also the graph neighbourhoods of $u$ and $v$ of radius $\underline{K}_{t,t'}$ are disjoint, i.e.,
 \be
 \P_{\text{g}}\Bigg(\bigcap_{t'=t}^{\infty} \Big\{\CB_G^{(t')}(u, \underline{K}_{t,t'}) \cap \CB_G^{(t')}(v, \underline{K}_{t,t'}) = \varnothing\Big\} \Bigg) = 1.\nonumber
 \ee
 Since any path connecting $u$ and $v$ has to pass through the boundary of the graph neighbourhoods,  $\P_{\mathrm{g}}$-a.s.\ for all $t'$,
 \be
 d_L^{(t')}(u, v)
 \geq \sum_{q\in\{u,v\}} d_L^{(t')}\big(q, \partial \CB_G(q, \underline{K}_{t,t'})\big)
 \geq \sum_{q\in\{u,v\}}\sum_{k=0}^{\underline{K}_{t,t'}-1} d_L^{(t')}\big(\partial \CB_G(q, k), \partial \CB_G(q, k+1)\big),\label{eq:lower-bnd-to-bnd}
 \ee
 where for two sets of vertices $\CV,\CW\subset [t']$ we define $d_L^{(t')}(\CV,\CW):=\min_{v\in\CV,w\in\CW}d_L^{(t')}(v,w)$.
 This leaves to show that, for some $M_L=M_L(\delta')$, $q\in\{u,v\}, C>0$,
 \be
 \P_{\mathrm{g}}\bigg(\exists t'\geq t: \sum_{k=0}^{\underline{K}_{t,t'}-1} d_L^{(t')}\big(\partial \CB_G(q, k), \partial \CB_G(q, k+1)\big)\leq Q_{t,t'} - M_L\bigg) \leq C\delta'.\label{eq:lower-weighted-inproof-1}
 \ee
 We argue in three steps: we prove that it is sufficient to consider the error probabilities only along a specific subsequence $(t_i)_{i\geq0}$ of times.
 This is needed to obtain a summable error bound in $t'$.
 Similarly to \cite[Proposition 4.1]{jorkom2019weighted}, we prove along the subsequence $(t_i)_{i\geq 0}$ an upper bound on the sizes of the graph neighbourhood boundaries of $u$ and $v$ up to radius $\underline{K}_{t,t_i}$.
 This allows to bound the minimal weight on an edge between vertices at distance $k$ and $k+1$ from $q\in\{u,v\}$.

 By the definition of  $Q_{t,t'}$ and $K_{t,t'}$ in \eqref{eq:kt} and \eqref{eq:qt-kt}, due to the integer part, $\underline{K}_{t,t'}$ and the rhs between brackets in \eqref{eq:lower-weighted-inproof-1} decrease at the times
 \begin{equation}
  t_i := \min\{t': K_{t,t}- K_{t,t'} = 2i\},\qquad \text{ for }i\in\{0,\dots,K_{t,t}/2\},\label{eq:ti}
 \end{equation}
 while the lhs between brackets in \eqref{eq:lower-weighted-inproof-1} may decrease for any $t'\geq t$. Because the addition of new vertices can create new (shorter) paths, we have for $i\geq 1$
 \begin{align}
  \big\{\exists t'\in [t_{i-1}, t_i) & : d_L^{(t')}\big(q, \partial \CB_G(q, \underline{K}_{t,t'})\big)\leq Q_{t,t'} - M_L\big\}\nonumber \\
                                     & \subseteq
  \big\{d_L^{(t_i)}\big(q, \partial \CB_G(q, \underline{K}_{t,t_{i}}+2)\big)\leq Q_{t,t_{i-1}} - M_L\big\},\label{eq:lwr-weighted-events-1}
 \end{align}
 where $K_{t,t'}=K_{t,t_i}+2$ for $t'\in[t_{i-1},t_i)$ follows from \eqref{eq:ti}.
 By construction of $t_i$ and $Q_{t,t'}$ in \eqref{eq:qt-kt}, where the summands are nonincreasing, there exists $M_1>0$
 such that for all $t$
 \[
  |Q_{t,t_i} - Q_{t,t_{i-1}}|\leq M_1.
 \]
 This yields that we can bound \eqref{eq:lwr-weighted-events-1} further to obtain
 \begin{align}
  \big\{\exists t'\in [t_{i-1}, t_i) & : d_L^{(t')}\big(q, \partial \CB_G(q, \underline{K}_{t,t'})\big)\leq Q_{t,t'} - M_L\big\}\nonumber \\
                                     & \subseteq
  \big\{d_L^{(t_i)}\big(q, \partial \CB_G(q, \underline{K}_{t,t_{i}}+2)\big)\leq Q_{t,t_{i}} - M_L + M_1\big\}.\nonumber
 \end{align}
 Hence, by a union bound over $i$, we can bound \eqref{eq:lower-weighted-inproof-1}, i.e.,
 \begin{align}
  \P_{\mathrm{g}}\Big(\exists t'\geq t & : d_L^{(t')}\big(q, \partial \CB_G(q, \underline{K}_{t,t'})\big)\leq Q_{t,t'} - M_L\Big) \nonumber \\
                                       & \leq
  \sum_{i=1}^{K_{t,t}/2}\P_{\mathrm{g}}\Big(d_L^{(t_i)}\big(q, \partial \CB_G(q, \underline{K}_{t,t_i}+2)\big)\leq Q_{t,t_{i}} - M_L +M_1\Big).\label{eq:lwr-bound-subs}
 \end{align}
 \noindent In Lemma \ref{lemma:lower-neighbourhoods} in the appendix we show that a generalization of \cite[Lemma 4.5]{jorkom2019weighted} gives for $B$ sufficiently large (depending on $\delta'$) and $m_{i,k}^\sss{(t)}(B):= \exp\big(2B(1\vee\log(t_i/t))(\tau-2)^{-k/2}\big)$ that
 \begin{align}
  \P_{\mathrm{g}}\Bigg(\bigcup_{k=1}^{\underline{K}_{t,t_i }+2}\Big\{\vert\partial\CB_G^{(t_i)}(q, k)\vert \geq m_{i,k}^\sss{(t)}(B)\Big\} \Bigg)
  \leq 2\exp\big(-B(1\vee\log(t_i/t))\big). \label{eq:lower-claim-neighbourhoods}
 \end{align}
 We denote the complement of the event inside the $\P$-sign by $\CE_{\mathrm{neigh}}^\sss{(i)}(q)$ for a fixed $i$.  Define the conditional probability measure $\P_{\mathrm{g},\mathrm{n}}^\sss{(i)}(\cdot):=\P\big(\,\cdot\mid\CE_{\text{good}}\cap \CE_{\mathrm{neigh}}^\sss{(i)}(u)\cap\CE_{\mathrm{neigh}}^\sss{(i)}(v)\big)$.
 The number of edges connecting a vertex at distance $k$ from $q$ to a vertex at distance $k+1$ from $q$ can then be bounded for all $k\leq \underline{K}_{t,t_i}+2$, i.e., $\P_{\mathrm{g},\mathrm{n}}^\sss{(i)}$-a.s.
 \begin{align}
  |\partial\CB_G^{(t_i)} & (q,k)|\cdot|\partial\CB_G^{(t_i)}(q,k+1)| \nonumber   \\
                         & \leq m_{i,k}^\sss{(t)}(B)\cdot m_{i,k+1}^\sss{(t)}(B)
  \leq \exp\big(4B(1\vee\log(t_i/t))(\tau-2)^{-(k+1)/2}\big) =: n_{i,k}. \label{eq:upper-nik-def}
 \end{align}
 Since all edges in the graph are equipped with i.i.d.\ copies of $L$, and as the minimum of $K$ i.i.d.\ random variables is nonincreasing in $K$, we have by Lemma \ref{lemma:methods-of-minima} for $\xi>0$ that for $k\leq \underline{K}_{t,t_i}+1$
 \begin{align*}
  \P_{\text{g,n}}^\sss{(i)} & \Big(
  d_L^{(t_i)}\big(\partial \CB_G^{(t_i)}(q, k),
  \partial \CB_G^{(t_i)}(q, k+1)\big)
  \leq F_L^{(-1)}\big(n_{i,k}^{-1-\xi}\big)
  \Big)\nonumber                    \\
                            & \leq
  \P_{\text{g,n}}^\sss{(i)}\Big(\min_{j\in[n_{i,k}]} L_{j,k} \leq F_L^{(-1)}\big(n_{i,k}^{-1-\xi}\big)\Big)
  \leq \exp\big(-4B\xi(1\vee\log(t_i/t))(\tau-2)^{-(k+1)/2}\big).
 \end{align*}
 \noindent
 Recall \eqref{eq:lower-bnd-to-bnd}.
 We apply the inequality in the event in the first row above for $k\leq \underline{K}_{t,t_i}+1$
 to obtain a bound on the lhs between brackets in the second row in \eqref{eq:lwr-bound-subs}, i.e., by a union bound
 \begin{align}
  \P_{\mathrm{g,n}}^\sss{(i)}\bigg(d_L^{(t_i)}\big(q, & \partial \CB_G(q, \underline{K}_{t,t_i}+2)\big)\leq \sum_{k=0}^{\underline{K}_{t,t_i}+1}F_L^{(-1)}\big(n_{i,k}^{-1-\xi}\big)\bigg) \nonumber \\
                                                      & \leq \sum_{k=0}^{\underline{K}_{t,t_i}+1}\exp\big(-4B\xi(1\vee\log(t_i/t))(\tau-2)^{-(k+1)/2}\big) \nonumber                                 \\
                                                      & \leq 2\exp\big(-4B\xi(1\vee\log(t_i/t))(\tau-2)^{-1}\big).\label{eq:lower-weighted-bound-fixed-i}
 \end{align}
 We now bound the sum in the above event from below to relate it to the rhs between brackets in \eqref{eq:lwr-bound-subs}. We do so by modifying \cite[Proof of Proposition 4.1, after (4.31)]{jorkom2019weighted}. Afterwards, we bound the total error probability by taking a union bound over the times $(t_i)_{i\ge 0}$.

To bound $F_L^{(-1)}\big(n_{i,k}^{-1-\xi}\big)$ from below, we need an upper bound on $n_{i,k}$ in \eqref{eq:upper-nik-def} since $z\mapsto F_L^{-(1)}(1/z)$ is nonincreasing. We first establish a lower and upper bound on $t_i$.
 Recall the integer $K_{t,t'}$ defined in \eqref{eq:kt}, so that we may write for $t'\leq t_{K_{t,t}/2}$
 \[K_{t,t'}= 2\big(\log\log(t)-\log\big(\log(t'/t)\vee 1\big)\big)/|\log(\tau-2)| - a_{t'}\] for  some $a_{t'}\in(0,1)$ being the fractional part of the expression. Using this notation one can verify that
 \be
 t_i\in\big[t\exp\big((\tau-2)^{-i+1}\big), t\exp\big((\tau-2)^{-i-1}\big)\big] =: [\ubar{t}_i, \bar{t}_i] \qquad \text{ for $i\leq K_{t,t}/2$}. \label{eq:lower-ti}
 \ee
 Substituting the upper bound on $t_i$ into $n_{k,i}$ in \eqref{eq:upper-nik-def}, yields that there exists $C=C(\xi, B)>0$ such that
 \be
 n_{i,k}^{1+\xi} \in \big[\exp\big((\tau-2)^{-i-k/2}/C\big), \exp\big(C(\tau-2)^{-i-k/2}\big)\big],\nonumber
 \ee
 which implies that, recalling $\underline{K}_{t,t_i}=K_{t,t_i}-M_G$ for some constant $M_G>0$ by \eqref{eq:lower-kt},
 \be
 \sum_{k=0}^{\underline{K}_{t,t_i}+1}F_L^{(-1)}\big(n_{i,k}^{-(1+\xi)}\big)
 \ge
 \sum_{k=0}^{\underline{K}_{t,t_i}-M_G + 1}F_L^{(-1)}\big(\exp\big(-C(\tau-2)^{-i-k/2}\big)\big).\nonumber
 \ee
 Observe that for a monotone nonincreasing function $g$, $g(1)<\infty$
 \be
 \sum_{k=\lceil a\rceil + 1}^{\lfloor b\rfloor}g(k)\overset{(\star)}\le \int_a^b g(x)\rd x \overset{(\ast)}\le \sum_{k=\lfloor a\rfloor}^{\lfloor b\rfloor}g(k).\label{eq:integral-sum-tricks}
 \ee
 Since $z\mapsto F_L^{-(1)}(1/z)$ is nonincreasing and bounded, we obtain by $(\ast)$ that
 \be
 \sum_{k=0}^{K_{t,t_i}-M_G+1}F_L^{(-1)}\big(n_{i,k}^{-(1+\xi)}\big)
 \ge
 \int_{x=0}^{K_{t,t_i}}F_L^{(-1)}\big(\exp\big(-B(\tau-2)^{-i-x/2}\big) \big)\rd x - M.\nonumber
 \ee
 Applying the change of variables
 $
  y = x + 2i + 2\log(B)/|\log(\tau-2)|
 $
 yields for $C=2\log(B)/|\log(\tau-2)|$ and some constant $M_L\ge M$ that
 \begin{align*}
  \sum_{k=0}^{K_{t,t_i}-M_G+1}F_L^{(-1)}\big(n_{i,k}^{-(1+\xi)}\big)
   & \ge
  \int_{y=2i  + C}^{2i + K_{t,t_i} + C}F_L^{(-1)}\big(\exp\big(-(\tau-2)^{-y/2}\big)\big) \rd y - M \nonumber \\
   & \ge
  \int_{y=2i}^{2i + K_{t,t_i}}F_L^{(-1)}\big(\exp\big(-(\tau-2)^{-y/2}\big)\big) \rd y - M_L,
 \end{align*}
 again using that  $z\mapsto F_L^{(-1)}(1/z)$ is bounded and nonincreasing. By transforming the integral back to a summation using $(\star)$ in \eqref{eq:integral-sum-tricks}, we obtain by definition of $Q_{t,t_i}$ in \eqref{eq:qt-kt}, and $K_{t,t}-K_{t,t_i}=2i$  in \eqref{eq:kti}
 \begin{align*}
  \sum_{k=0}^{K_{t,t_i}-M_G+1}F_L^{(-1)}\big(n_{i,k}^{-(1+\xi)}\big)
   & \ge
  \sum_{k=2i+1}^{2i + K_{t,t_i}}F_L^{(-1)}\big(\exp\big(-(\tau-2)^{-k/2}\big)\big) - M_L         \\
   & =
  \sum_{k=K_{t,t}-K_{t,t_i}+1}^{K_{t,t}}F_L^{(-1)}\big(\exp\big(-(\tau-2)^{-k/2}\big)\big) - M_L \\
   & = Q_{t,t_i} - M_L.
 \end{align*}
 Using this lower bound inside the event in \eqref{eq:lower-weighted-bound-fixed-i}, yields that
 \begin{align}
  \P_{\mathrm{g,n}}^\sss{(i)}\big(d_L^{(t_i)}\big(q, \partial \CB_G(q, \underline{K}_{t,t_i}+2)\big)\leq Q_{t,t_i}-M_L\big)
  \leq 2\exp\big(-4B\xi(1\vee\log(t_i/t))(\tau-2)^{-1}\big).\label{eq:lower-weighted-bound-fixed-i-2}
 \end{align}
 Recall that we would like to show \eqref{eq:lower-weighted-inproof-1}.
 Its proof is accomplished by a union bound over the times $(t_i)_{i\leq\underline{K}_{t,t'/2}}$ if we show that there is a $B$ sufficiently large such that the error probabilities on the rhs in \eqref{eq:lower-claim-neighbourhoods} and \eqref{eq:lower-weighted-bound-fixed-i-2} are smaller than $\delta'$  when summed over $i\leq \underline{K}_{t,t'}/2$.
 For this it is sufficient to show that for any $\hat{\delta}>0$ and $C'>0$ there exists $B>0$ such that
 \be
 \sum_{i=0}^{\infty}\exp\big(-C'B(1\vee\log(t_i/t))\big) \leq \hat{\delta}.\nonumber
 \ee
 This follows from the lower bound on $t_i$ in \eqref{eq:lower-ti}, since for $B$ large
 \begin{align}
  \sum_{i=0}^{\infty}\exp\big(-C'B(1\vee\log(t_i/t))\big)
   & \leq
  \sum_{i=0}^{\infty}\exp\big(-C'B(1\vee\log(\ubar{t}_i/t))\big) \nonumber \\
   & =
  \sum_{i=0}^{\infty}\exp\big(-C'B(1\vee(\tau-2)^{-i+1})\big)\nonumber     \\
   & \leq 2\exp\big(-C'B(\tau-2)\big) < \hat{\delta}.\nonumber
 \end{align}
\end{proof}
\section{Proof of the upper bound}\label{sec:upper-bound}
The upper bound of Theorem \ref{thm:weighted-evolution} is stated in the following proposition.
\begin{proposition}[Upper bound weighted distance]\label{proposition:upper-bound}
 Consider the preferential attachment model with power-law exponent $\tau\in(2,3)$. Equip every edge upon creation with an i.i.d.\ copy of the non-negative random variable $L$. Let $u, v$ be two typical vertices in $\mathrm{PA}_t$.
 If $\bm{I}_2(L)<\infty$, then for any $\delta>0$, there exists $M_L>0$ such that
 \begin{equation}
  \P\Big(\exists t': d_L^{(t')}(u,v) \geq 2Q_{t,t'} + M_L \Big)\leq \delta.\label{eq:upper-bound}
 \end{equation}
 Regardless of the value of $\bm{I}_2(L)$, for any $\delta,\vareps>0$, there exists $M_L>0$ such that
 \begin{equation}
  \P\Big(\exists t': d_L^{(t')}(u,v) \geq 2(1+\vareps)Q_{t,t'} + M_L\Big)\leq \delta.\label{eq:upper-bound-weak}
 \end{equation}
\end{proposition}

\subsubsection*{Outline of the proof}
To prove Proposition \ref{proposition:upper-bound}, we have to show that for every $t'>t$ there is a $t'$-present $(u,v)$-path whose total weight is bounded from above by the rhs between brackets in \eqref{eq:upper-bound} and \eqref{eq:upper-bound-weak}, respectively.
We construct a $t'$-present five-segment path $\bm{\pi}^\sss{(t')}$ of three segment types. We write it as $\bm{\pi}^\sss{(t')}=\ora{\bm{\pi}}_{u,0}^\sss{(t)}\circ\ora{\bm{\pi}}_{u,1}^\sss{(t')}\circ\bm{\pi}_{\text{core}}^\sss{(t')}\circ\ola{\bm{\pi}}_{v,1}^\sss{(t')}\circ\ola{\bm{\pi}}_{v,0}^\sss{(t)}$.
Here, we denote for a path segment $\ora{\bm{\pi}}=(\pi_0,\dots,\pi_n)$ its reverse by $\ola{\bm{\pi}}:=(\pi_n,\dots,\pi_0)$.
The path segments are constructed similar to the methods demonstrated in \cite[Section 3]{jorkom2019weighted}. However, we need stronger error bounds compared to \cite{jorkom2019weighted}, so that the error terms are also small when \emph{summed} over $t'\geq t$. Let $\delta'>0$ be sufficiently small, and $(M_i)_{i\leq 3}$ be suitable positive constants.
\begin{enumerate}
 \item[Step 1.] For $q\in\{u,v\}$, the path segment $\ora{\bm{\pi}}_{q,0}^\sss{(t)}:=(q,\dots,q_0)$ connects $q$ to a vertex $q_0$ that has indegree at least $s_{0}^\sss{(0)}>0$ at time $(1-\delta')t$. The path segment uses only vertices that are older than $(1-\delta')t$. The path segments are fixed for all $t'\geq t$. The number of edges on $\ora{\bm{\pi}}_{q,0}^\sss{(t)}$  is bounded from above by $M_1=M_1(\delta', s_0^\sss{(0)})$ w/p close to one.
       \begin{enumerate}
        \item Since the number of edges on $\ora{\bm{\pi}}_{q,0}^\sss{(t)}$ is bounded, its total weight can also be bounded by a constant.
              This is captured by the constant $M_L$ in the statement of Proposition \ref{proposition:upper-bound}.
        \item For the end vertex $q_0$ of $\ora{\bm{\pi}}^\sss{(t)}_{q,1}$, for $q\in\{u,v\}$,  we identify the rate of growth of its indegree $(D^{\leftarrow}_{q_0}(t'))_{t'\geq t}$.  We bound $D^{\leftarrow}_{q_0}(t')$ from below during the \emph{entire interval} $[t,\infty)$ by a sequence that tends to infinity in $t'$ sufficiently fast.
       \end{enumerate}
 \item[Step 2.] For the path segments $\ora{\bm{\pi}}_{u,1}^\sss{(t')}, \bm{\pi}_{\text{core}}^\sss{(t')}$, and $\ola{\bm{\pi}}_{v,1}^\sss{(t')}$ we argue, similar to the proof of Proposition \ref{proposition:lower-bound-weighted}, that it is sufficient to construct these path segments along a specific subsequence $(t_i)_{i\geq0}$, as the $t_i$-present path segments have small enough total weight when compared to $Q_{t,t'}$ for all $t'\in(t_i,t_{i+1}]$. With a slight abuse of notation we abbreviate for the path (segments) $\bm{\pi}^\sss{(i)}:=\bm{\pi}^\sss{(t_i)}$.
 \item[Step 3.] For $q\in\{u,v\}$, the path segment $\ora{\bm{\pi}}_{q,1}^\sss{(i)}$ consists of at most $K_{t,t_i} + M_2$ edges and connects the vertex $q_0$ to the so-called $i$-th inner core, i.e., the set of vertices with a sufficiently large degree at time $(1-\delta')t_i$.
       These path segments use only edges that arrived \emph{after} time $(1-\delta')t_i$ and the total weight of any such path  segment is therefore independent of the total weight on the segments $\ora{\bm{\pi}}_{u,0}^\sss{(t)}$ and $\ora{\bm{\pi}}_{v,0}^\sss{(t)}$,  that use only edges that arrived \emph{before} $(1-\delta')t$.
       For the weighted distance, we construct the path segment $\ora{\bm{\pi}}_{q,1}^\sss{(i)}$ greedily (minimizing the edge weights) to bound the weighted distance between $q_0$ and the inner core from above by $Q_{t,t_i}+M_3$.
 \item[Step 4.]
       Denote the end vertices of $\ora{\bm{\pi}}_{u,1}^\sss{(i)}$ and $\ora{\bm{\pi}}^\sss{(i)}_{v,1}$ by $w_{u}^\sss{(i)}$ and $w_{v}^\sss{(i)}$, respectively.
       The middle path segment $\bm{\pi}_{\text{core}}^\sss{(i)}$ connects the two vertices $w_u^\sss{(i)},w_v^\sss{(i)}$ in the inner core. The number of disjoint paths of bounded length between from $w_u^\sss{(i)}$ to $w_v^\sss{(i)}$ is growing polynomially in $t'$. This yields that $d_L^{(t_i)}(w_u^\sss{(i)},w_v^\sss{(i)})$ is bounded by a constant for all $i$. This weight is captured by $M_L$ in Proposition \ref{proposition:upper-bound}.
 \item[Step 5.] Eventually we \emph{glue} the different path segments together and obtain the results \eqref{eq:upper-bound} and \eqref{eq:upper-bound-weak}. The path segments $\ora{\bm{\pi}}_{u,1}^\sss{(t')}, \bm{\pi}_{\text{core}}^\sss{(t')}$, and $\ola{\bm{\pi}}_{v,1}^\sss{(t')}$ change at the times $(t_i)_{i\geq0}$, while the segments $\ora{\bm{\pi}}_{u,0}^\sss{(t)}$ and $\ola{\bm{\pi}}_{v,0}^\sss{(t)}$ stay the same for all $t'\geq t$.
\end{enumerate}
\begin{figure}
 \includegraphics[width=0.9\textwidth]{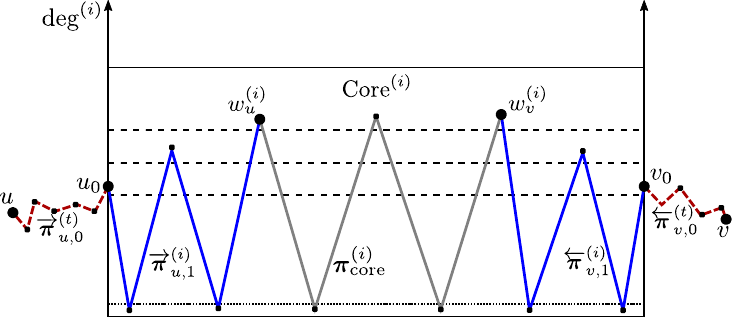}
 \caption{The five-segment path $\bm{\pi}^\sss{(i)}$ from $u$ to $v$ at a single time $t_i$. The $y$-axis represents the degree of the vertices at time $t_i$. The connected dots form the constructed path $\bm{\pi}^\sss{(i)}$ from $u$ to $v$ that is present at time $t_i$.
  The top continuous black line is the maximal degree in the graph at time $t_i$, while the dashed horizontal lines represent a degree-threshold sequence $(s_k^\sss{(i)})_{i\geq0,k\geq0}$ defined in the proof of Proposition \ref{proposition:weighted-to-inner} for the segments $\ora{\bm{\pi}}_{u,1}^{(i)}$ and $\ola{\bm{\pi}}_{v,1}^{(i)}$.
 }
 \label{fig:upper-outline}
\end{figure}
See Figure \ref{fig:upper-outline} for a sketch of the constructed path and Figure \ref{fig:upper-bound} in the introduction for a visualization of the construction of the subsequence $(t_i)_{i\geq 0}$ and the control of the degree of the vertex $q_0$. Recall that the proof of the lower bound was based on controlling the dynamically changing \emph{graph neighbourhood} up to distance $\underline{K}_{t,t'}= K_{t,t'}-M_G$.
For the proof of the upper bound, the dynamics of the graph are mostly captured by controlling \emph{the degree} of only two vertices, see Step 1b in the outline.
\subsection*{Step 1. Initial segments and degree evolution}
Recall $D_v^{\leftarrow}(t')$, the indegree of vertex $v$ at time $t'$.
\begin{lemma}[Finding a high-degree vertex]\label{lemma:upper-suff-degree}
 For any $s_{0}^\sss{(0)},\delta'>0$ there exists $M>0$ such that for a typical vertex $q$ in $\mathrm{PA}_t$
 \begin{equation}
  \P\big(\nexists q_0\in[(1-\delta')t]: d_L^{((1-\delta')t)}(q, q_0)\leq M \text{ and } D_{q_0}^{\leftarrow}\big((1-\delta')t\big)\geq s_{0}^\sss{(0)} \big)\leq 3\delta'.\label{eq:upper-suff-degree}
 \end{equation}
 \begin{proof}
  Let $\CE_{\text{old}}:=\{q<(1-\delta')t\}$. Since $q$ is chosen uniformly among the first $t$ vertices,
  \[
   \P(\CE_{\text{old}}) = 1-\delta' + O(1/t).
  \]
  Recall $\CB_G^{(t)}(x,R)$ from \eqref{eq:graph-neighbourhood}.
  From minor adaptations of the proofs of \cite[Theorem 3.6]{dommers2010diameters} for FPA, and \cite[Proposition 5.10]{monch2013distances} for VPA
  it follows for all $\delta'>0$ that there exists $M_1$ such that
  \[
   \P\big(\CB_G^{((1-\delta')t)}(q, M_1)\cap \{x: D_{x}^{\leftarrow}\big((1-\delta')t\big)\geq s_0) \neq \varnothing \mid \CE_{\text{old}}\big) \geq 1-\delta'.
  \]
  We refer the reader there for the details.
  Conditionally on the above event between brackets there is a $(1-\delta')t$-present path segment from $q$ to $q_0$, whose edges carry i.i.d.\ weights. Hence, there exists $M_2>0$ such that the weight on the segment can be bounded, yielding \eqref{eq:upper-suff-degree}.
 \end{proof}
\end{lemma}
The following lemma bounds the degree evolution from below. It uses martingale arguments that are inspired by \cite{mori2002}. However, here the statement only refers to the process of a single vertex that has initial degree at least $s$ where \cite{mori2002} considers a set of vertices for any initial degree in FPA. Our statement applies to both FPA \emph{and} VPA. Moreover, we consider the degree during the entire interval $[t,\infty)$.
See also \cite[Section 5.1]{berger2014asymptotic} for results on the degree of an early vertex in $\mathrm{FPA}(m,\delta)$, where the considered vertex is born at time $o(t)$.
\begin{lemma}[Indegree lower bound]\label{lemma:degree-lowerbound}
 Consider the preferential attachment model with power-law parameter $\tau>2$. Let $q_0$ be a vertex such that $D^{\leftarrow}_{q_0}(t)\geq s\geq2$.
 There exists a constant $c>0$, not depending on $s$, such that for all $\delta'>0$
 \be
 \P\left(\exists t'\geq t:  D_{q_0}^{\leftarrow}(t')\leq \delta' s(t'/t)^{1/(\tau-1)}\right)\leq c\delta'.\label{eq:degree-lowerbound}
 \ee
 \begin{proof}
  Let $\gamma:=1/(\tau-1)$. Both for FPA and VPA, it holds by Definitions \ref{def:fpa} and \ref{def:vpa}  that
  \[
   \P\left(D_{q_0}^{\leftarrow}(t'+1) \geq D_v^{\leftarrow}(t') + 1 \mid \mathrm{PA}_{t'}\right)
   \geq \gamma D_{q_0}^{\leftarrow}(t')/t'
  \]
  for any $t'\geq t$.

  Let $(X_{t,t'})_{t'\geq t}$ be a discrete-time pure birth process satisfying $X_{t,t}=s$ and
  \be
  \P\big(X_{t,t'+1} = X_{t,t'} + 1 \mid X_{t,t'} = x\big) = 1 - \P\big(X_{t,t'+1} = X_{t,t'}\mid X_{t,t'} = x\big) = \gamma x/t'.\label{eq:birth-def}
  \ee
  Then the degree evolution $(D^\leftarrow_{q_0}(t'))_{t'\geq t}$ and $(X_{t,t'})_{t'\geq t}$ can be coupled such that the degree evolution dominates the birth process in the entire interval $[t,\infty)$. We first show that for any $k>-s$ and $\gamma\in(0,1)$, provided that $t'\geq t\geq k\gamma$,
  \begin{equation}
   Z^{(k)}_{t,t'} := \frac{\Gamma(t')}{\Gamma(t'+k\gamma)}\frac{\Gamma(t+k\gamma)}{\Gamma(t)} \frac{\Gamma(X_{t,t'}+k)}{\Gamma(X_{t,t'})} \label{eq:martingale-zk}
  \end{equation}
  is a non-negative martingale. The result will then follow by an application of the maximal inequality for $k=-1$. Clearly $\E[|Z^{(k)}_{t,t'}|]<\infty$, as the arguments in the Gamma functions in \eqref{eq:martingale-zk} are bounded away from 0.
  Moreover, since
  \be
  \Gamma(x) = (x-1)\Gamma(x-1),\label{eq:gamma-rec}
  \ee
  and using \eqref{eq:birth-def},
  \begin{align}
   \E\left[\frac{\Gamma\big(X_{t,t'+1}+k\big)}{\Gamma(X_{t,t'+1})} \mid X_{t,t'}\right]
    & =
   \frac{t'-\gamma X_{t,t'}}{t'}\frac{\Gamma(X_{t,t'}+k)}{\Gamma(X_{t,t'})} + \frac{\gamma X_{t,t'}}{t'}\frac{\Gamma(X_{t,t'}+k+1)}{\Gamma(X_{t,t'}+1)} \nonumber
   \\
    & = \frac{\Gamma(X_{t,t'}+k)}{\Gamma(X_{t,t'})} \Big(1 + \frac{k\gamma }{t'}\Big),\nonumber
  \end{align}
  making it straightforward to verify that the martingale property for $Z_{t,t'}^{(k)}$ holds for $t'\geq t$.
  Due to Kolmogorov's maximal inequality, for any $T>t, \lambda>0$
  \begin{align}
   \P\bigg(\sup_{t\leq t'\leq T} Z_{t,t'}^{(k)} \geq \lambda\bigg) & \leq \frac{\E[Z_{t,T}^{(k)}]}{\lambda} = \frac{Z_{t,t}^{(k)}}{\lambda} = \frac{\Gamma(X_{t,t} + k)}{\lambda\Gamma(X_{t,t})} = \frac{\Gamma(s+k)}{\lambda\Gamma(s)}.\label{eq:kolm-z}
  \end{align}
  Substituting \eqref{eq:martingale-zk} and $k=-1$, we see, using \eqref{eq:gamma-rec},
  \begin{align*}
   \bigg\{\sup_{t'\geq t} Z_{t,t'}^{(-1)} \geq \lambda\bigg\}
    & =
   \bigg\{\exists t'\geq t: \frac{\Gamma(t')}{\Gamma(t'-\gamma)}\frac{\Gamma(t-\gamma)}{\Gamma(t)} \frac{\Gamma\big(X_{t,t'}-1\big)}{\Gamma\big(X_{t,t'}\big)} \geq \lambda\bigg\} \\
    & =\bigg\{\exists t'\geq t:X_{t,t'}-1\leq \frac{\Gamma(t')}{\Gamma(t'-\gamma)}\frac{\Gamma(t-\gamma)}{\Gamma(t)} \frac{1}{\lambda}\bigg\}                                      \\
    & \supseteq
   \bigg\{\exists t'\geq t:X_{t,t'}\leq \frac{\Gamma(t')}{\Gamma(t'-\gamma)}\frac{\Gamma(t-\gamma)}{\Gamma(t)} \frac{1}{\lambda}\bigg\}.
  \end{align*}
  Since $\Gamma(x)/\Gamma(x+a) = x^{-a}(1+O(1/x))$, by \eqref{eq:kolm-z} there exists $c'$ such that for $t$ sufficiently large
  \[
   \frac{\Gamma(t')}{\Gamma(t'-\gamma)}\frac{\Gamma(t-\gamma)}{\Gamma(t)}
   \leq
   c'(t'/t)^{\gamma}.
  \]
  Choosing $\lambda=c'/(s\delta')$ yields
  \begin{equation*}
   \P\Big(\exists t'\geq t:  X_{t,t'}\leq \frac{1}{\lambda}c'(t'/t)^{\gamma}\Big)
   =
   \P\left(\exists t'\geq t:   X_{t,t'}\leq \delta' s(t'/t)^{\gamma}\right)
   \leq \frac{\delta' s}{c'(s-1)}.
  \end{equation*}
  Now \eqref{eq:degree-lowerbound} follows, since $s\geq 2$ and $\big(D_{q_0}(t')\big)_{t'\geq t} \overset{d}{\geq} \big(X_{t,t'})\big)_{t'\geq t}$.
 \end{proof}
\end{lemma}
\begin{remarky}
 The proof of this lemma can be adapted to obtain an upper bound on the degree evolution, by applying Kolmogorov's maximal inequality to the martingale $Z^{(1)}_{t,t'}$.
\end{remarky}
\subsection*{Step 2. Sufficient to construct the path along a subsequence of times}
\begin{lemma}[Subsequence of times]\label{lemma:upper-sub}
 Consider the preferential attachment model under the same conditions as Proposition \ref{proposition:upper-bound} and let $(t_i)_{i\geq 0}$ be as defined in \eqref{eq:ti}.
 If there exists $M_L>0$ such that
 \be
 \P\big(\exists i\in[K_{t,t}/2]: d_L^{(t_i)}(u,v) \geq 2Q_{t,t_i} + M_L\big) \leq \delta, \label{eq:upper-sub}
 \ee
 then \eqref{eq:upper-bound} holds. Similarly, \eqref{eq:upper-bound-weak} holds if there exists $M_L>0$ such that for any $\vareps>0$ and $t$ sufficiently large
 \be
 \P\big(\exists i\in[K_{t,t}/2]: d_L^{(t_i)}(u,v) \geq 2(1+\vareps)Q_{t,t_i} + M_L\big) \leq \delta. \label{eq:upper-sub-weak}
 \ee
 \begin{proof}
  Recall \eqref{eq:qt-kt}. The rhs between brackets in \eqref{eq:upper-bound} only decreases at the times $(t_i)_{i\geq 0}$, while the lhs is nonincreasing in $t'$. By \eqref{eq:qt-kt}, for $i>K_{t,t}/2$, $Q_{t,t_i}=Q_{t,t_{K_{t,t}/2}}$ and the asserted statements follow.
 \end{proof}
\end{lemma}
\subsection*{Step 3. Greedy path to the inner core}
\noindent Define the $i$-th inner core, for $t_i$ from \eqref{eq:ti}, as
\begin{equation}
 \mathrm{Core}^\sss{(i)} := \{x\in [\hat{t}_i]: D_x\big(\hat{t}_i\big) \geq \hat{t}_i^{\frac{1}{2(\tau-1)}}\log^{-\frac{1}{2}}(\hat{t}_i)\}, \qquad \text{for }\hat{t}_i:=(1-\delta')t_i. \label{eq:core}
\end{equation}
\begin{proposition}[Weighted distance to the inner core]\label{proposition:weighted-to-inner}
 Consider the preferential attachment model under the same conditions as Proposition \ref{proposition:upper-bound}. There exists $C>0$ such that for every $\delta'>0$,  there exist $s_{0}^\sss{(0)},M>0$ such that for a vertex $q_0$ satisfying $D_{q_0}\big((1-\delta')t\big)\geq s_{0}^\sss{(0)}$, when $\bm{I}_2(L)<\infty$,
 \begin{equation}
  \P\Bigg(\bigcup_{i\leq K_{t,t}/2} \big\{d_L^{(t_i)}(q_0, \mathrm{Core}^\sss{(i)}) \geq Q_{t,t_i} + M\big\}
  \Bigg) \leq C\delta'. \label{eq:weighted-to-inner}
 \end{equation}
 Regardless of the value of $\bm{I}_2(L)$, there exists $M>0$ such that for every $\vareps>0$,
 there is an $s_{0}^\sss{(0)}>0$ such that for a vertex $q_0$ satisfying $D_{q_0}\big((1-\delta')t\big)\geq s_{0}^\sss{(0)}$,
 \begin{equation}
  \P\Bigg(\bigcup_{i\leq K_{t,t}/2} \big\{d_L^{(t_i)}(q, \mathrm{Core}^\sss{(i)}) \geq (1+\vareps)Q_{t,t_i} + M\big\}
  \Bigg) \leq C\delta'. \label{eq:weighted-to-inner-weak}
 \end{equation}
\end{proposition}
\noindent
The bounds on the weighted distance in Proposition \ref{proposition:weighted-to-inner} are realized by constructing the segments $\ora{\bm{\pi}}_{u,1}^\sss{(i)}$ and $\ora{\bm{\pi}}_{v,1}^\sss{(i)}$, whose total weight we bound from above. For this we follow the same ideas as in \cite[Proposition 3.4]{jorkom2019weighted}, up to computational differences. Therefore, the proof we give here is not completely self-contained and for some bounds we will refer the reader to \cite{jorkom2019weighted}.
\subsubsection*{Preparations for the proof of Proposition {\ref{proposition:weighted-to-inner}}}
For some constants $s_{0}^\sss{(0)}, \delta'>0$, the sequence $\vareps_k:=(k+1)^{-2}$, and $\hat{t}_i$ from \eqref{eq:core}, define the degree threshold sequence
\begin{numcases}
 {s_{k}^\sss{(i)} =}
 \delta' s_{0}^\sss{(0)} \left(\hat{t}_i/t\right)^{\frac{1}{\tau-1}} & $k=0$, \label{eq:sk-init}\\
 \min\left\{\big(s_{k-1}^\sss{(i)}\big)^{(1-\vareps_k)/(\tau-2)}, \hat{t}_i^{\frac{1}{2(\tau-1)}}\log^{-\frac{1}{2}}(\hat{t}_i)\right\} & $k>0$.\label{eq:sk}
\end{numcases}
For each time $t_i$, the initial value $s_{0}^\sss{(i)}$ is chosen such that it matches the bound on the degree in \eqref{eq:degree-lowerbound}.
The maximum value of $s_k^\sss{(i)}$, for each fixed $\hat{t}_i$, matches the condition for vertices to be in the $i$-th inner core, see \eqref{eq:core}.
Set
\be
\kappa^\sss{(i)} := \min\{k: s_{k+1}^\sss{(i)}=s_{k}^\sss{(i)}\}.\label{eq:kti}
\ee
Denote by $\CL_{k}^\sss{(i)}$ the $k$-th vertex layer: the set of vertices with degree at least  $s_{k}^\sss{(i)}$ at time $\hat{t}_i$, i.e.,
\begin{align}
 \CL_{k}^\sss{(i)} & := \{x\in[\hat{t}_i]: D_x(\hat{t}_i)\geq s_{k}^\sss{(i)}\}.\label{eq:layers}
\end{align}
The path segment $\ora{\bm{\pi}}_{q,1}^\sss{(i)}$ to the inner core has length $2\kappa^\sss{(i)}$ and uses alternately a \emph{young} vertex $y_{k}^\sss{(i)}\in [\hat{t}_i, t_i]$ and an \emph{old} vertex $\pi_{k}^\sss{(i)}$ from the layer $\CL_{k}^\sss{(i)}$.
Thus, for $\pi_0:=q_0$, $\ora{\bm{\pi}}_{q,1}^\sss{(i)}$ has the form $\ora{\bm{\pi}}^\sss{(i)}_{q,1}= (\pi_0, y_{1}^\sss{(i)}, \pi_{1}^\sss{(i)},\dots,y_{\kappa^\sss{(i)}}^\sss{(i)}, \pi_{\kappa^\sss{(i)}}^\sss{(i)})$.
To keep notation light we omit a subscript $q$ for the individual vertices on the segments $\ora{\bm{\pi}}^\sss{(i)}_{q,1}$ for $q\in\{u,v\}$.

In the next lemmas, we show that $\ora{\bm{\pi}}^\sss{(i)}_{q,1}$ exists for all $i$ w/p close to one, and bound its total weight. We outline the steps briefly.
Using the choice of $s_{k}^\sss{(i)}$, we bound $\kappa^\sss{(i)}$ in terms of $K_{t,t_i}$.
Then, since the number of vertices that have degree at least $s_{k+1}^\sss{(i)}$ at time $\hat{t}_i$ is sufficiently large,
it is likely that there are \emph{many} connections from a vertex $\pi_{k}^\sss{(i)}$ via a connector vertex   $y_{k+1}^\sss{(i)}$ to the $(k+1)$-th layer.
We denote the set of connectors by $\CA_{k+1}^\sss{(i)}\big(\pi_{k}^\sss{(i)}\big)$, i.e., for a vertex $\pi_{k}^\sss{(i)}\in\CL_{k}^\sss{(i)}$,
\begin{equation}
 \CA_{k+1}^\sss{(i)}\big(\pi_{k}^\sss{(i)}\big) := \{y\in(\hat{t}_i, t_i]: \exists x_{k+1}^\sss{(i)}\in \CL_{k+1}^\sss{(i)}: \pi_k^\sss{(i)}\leftrightarrow y \leftrightarrow x_{k+1}^\sss{(i)}\}.\nonumber
\end{equation}
Given $(\pi_{0},\pi_{1}^\sss{(i)},\dots,\pi_{k}^\sss{(i)})$, we greedily set, if $\CA_{k+1}^\sss{(i)}(\pi_{k}^\sss{(i)})$ is non-empty,
\begin{equation}
 (y_{k+1}^\sss{(i)},\pi_{k+1}^\sss{(i)}) := \argmin_{\substack{(y, x_{k+1}^\sss{(i)})\in \\ [\hat{t}_i, t_i]\times \CL_{k+1}^\sss{(i)}}} \big\{L_{(\pi_{k}^\sss{(i)}, y)} + L_{(y, x_{k+1}^\sss{(i)})}\big\}.
 \label{eq:choice-greedy}
\end{equation}
If there exists $k\leq \kappa^\sss{(i)}$ such that $\CA_{k+1}^\sss{(i)}(\pi_{k}^\sss{(i)})=\varnothing$, we say that the construction has failed.
When the construction succeeds, we can bound the weighted distance to the inner core, i.e.,
\begin{equation}
 d_L^{(t_i)}(\pi_0, \mathrm{Core}^\sss{(i)})  \leq \sum_{k=0}^{\kappa^\sss{(i)}-1}d_L^{(t_i)}(\pi_{k}^\sss{(i)}, \pi_{k+1}^\sss{(i)})\leq \sum_{k=0}^{\kappa^\sss{(i)}-1}\big(L_{(\pi_{k}^\sss{(i)}, y_{k+1}^\sss{(i)})} + L_{(y_{k+1}^\sss{(i)}, \pi_{k+1}^\sss{(i)})}\big).
 \label{eq:greedy-weight}
\end{equation}
We show that for all $i$ there exists a sequence $(n_k^\sss{(i)})_{k\leq \kappa^{(i)}}$ such that $|\CA_{k+1}^\sss{(i)}\big(\pi_{k}^\sss{(i)}\big)|\geq n_{k}^\sss{(i)}$ for all $k\leq \kappa^\sss{(i)}$ w/p close to one.
This allows to bound the minimal weight in the rhs of \eqref{eq:choice-greedy} from above, so that eventually this yields an upper bound for the rhs in \eqref{eq:greedy-weight}.

We start with a lemma that relates $K_{t,t_i}$ to $\kappa^\sss{(i)}$, half the length of $\ora{\bm{\pi}}^\sss{(i)}_{q,1}$. Also, we show that $k\mapsto s_{k}^\sss{(i)}$ is bounded from below by a doubly exponentially growing sequence.
\begin{lemma}
 Let $\big(s_{k}^\sss{(i)}\big)_{i \leq K_{t,t}/2,k\leq \kappa^\sss{(i)}}$ as in \eqref{eq:sk}, with $s_{0}^\sss{(0)}$ sufficiently large. Then
 \begin{equation}
  s_{k}^\sss{(i)} \geq \left(\delta' s_{0}^\sss{(0)}\left(\hat{t}_i/t\right)^{1/(\tau-1)}\right)^{c'(\tau-2)^{-k}}\label{eq:bound-sik}
 \end{equation}
 for some constant $c'>0$. There exists $M\in\mathbb{N}$ such that for $\kappa^\sss{(i)}$ defined in \eqref{eq:kti} and $i\leq K_{t,t}/2$
 \begin{equation}
  \kappa^\sss{(i)}
  \leq
  K_{t,t_i}/2 + M.\label{eq:bound-kti}
 \end{equation}
 \begin{proof}
  By our choice $\vareps_k=(k+1)^{-2}$ before \eqref{eq:sk}, it holds that $\prod_{j=1}^\infty(1-\vareps_k)>0$.
  The bound \eqref{eq:bound-sik} follows immediately from the definition of $(s^\sss{(i)}_k)$ in \eqref{eq:sk}.
  From \cite[Lemma 3.7]{jorkom2019weighted}  the bound \eqref{eq:bound-kti} immediately follows for $i=0$, leaving to verify the bound for $i\geq 1$.
  By the choice of $\kappa^\sss{(i)}$ in \eqref{eq:kti} and $(s_{k}^\sss{(i)})$ in \eqref{eq:sk}, and the bound \eqref{eq:bound-sik}, for any $k\geq \kappa^\sss{(i)}$ it holds that
  \begin{align*}
   \left(\delta' s_{0}^\sss{(0)}\left(\hat{t}_i/t\right)^{1/(\tau-1)}\right)^{c'(\tau-2)^{-k}}
    & \geq
   \hat{t}_i^{\frac{1}{2(\tau-1)}}\log^{-\frac{1}{2}}(\hat{t}_i)
  \end{align*}
  Taking logarithms twice, and rearranging gives that for $k\geq \kappa^\sss{(i)}$,
  \begin{align*}
   k \log(1/(\tau-2)) + \log\log(\hat{t}_i/t) & + \log\left(1 + \frac{(\tau-1)\log\big(\delta' s_{0}^\sss{(0)}\big)}{\log(\hat{t}_i/t)}\right) \\
                                              & \geq
   \log\log(\hat{t}_i) + \log\left(\frac{1}{2c'} - \frac{\tau-1}{2c'}\frac{\log\log(\hat{t}_i)}{\log(\hat{t}_i)}\right).
  \end{align*}
  From \eqref{eq:core} it follows for all $i\leq K_{t,t}/2$ that $\hat{t}_i\geq (1-\delta')t$. Thus, the last terms on both lines are bounded by a constant for large $t$. Hence, there is $M=M(\tau)$ such that if $s_{0}^\sss{(0)}\geq 1/\delta' $, $i\geq 1$
  \[
   \kappa^\sss{(i)}
   \leq
   \frac{\log\log(\hat{t}_i) - \log \log(\hat{t}_i/t)}{|\log(\tau-2)|}  + M'.
  \]
  By construction of $(t_i)_{i\leq K_{t,t}/2}$ in \eqref{eq:lower-ti}, there exists $b>0$ such that $\hat{t}_i\leq t^b$, for all $i\leq K_{t,t}/2$.
  This relates $\log\log(\hat{t}_i)$ to $\log\log(t)$. Hence, there exists $M$ such that \eqref{eq:bound-kti} holds for $i\geq 1$,
  recalling $\hat{t}_i=(1-\delta')t_i$, and the definition of $K_{t,t'}$ in \eqref{eq:kt}.
 \end{proof}
\end{lemma}
We now prove Proposition \ref{proposition:weighted-to-inner}. We construct the segment $\ora{\bm{\pi}}^\sss{(i)}_{q,1}$ for $i\leq \kappa^\sss{(i)}$ and $q\in\{u,v\}$.
\begin{proof}[Proof of Proposition \ref{proposition:weighted-to-inner}]
 Let $\CE_{\text{deg}}:=\{\forall i\geq 0: D_{q_0}^{\leftarrow}(\hat{t}_i)\geq s_{0}^\sss{(i)}\}$.
 By Lemma \ref{lemma:degree-lowerbound} and the choice of $(s_{0}^\sss{(i)})_{i\geq 0}$ in \eqref{eq:sk-init} we have that
 \[
  \P(\neg\CE_{\text{deg}}) \leq c\delta'.
 \]
 We write $\P_\text{deg}(\cdot):=\P(\,\cdot\mid\CE_{\text{deg}})$.
 We will first show that w/p close to one, the sets of connectors are sufficiently large. More precisely, for a set of vertices  $\{\pi_{k}^\sss{(i)}\}_{k\leq \kappa^\sss{(i)}, i\leq K_{t,t}/2}$, such that $\pi_{k}^\sss{(i)}\in\CL_{k}^\sss{(i)}$ and setting $\pi_{0}^\sss{(i)}:=\pi_0=q_0$ for all $i$, we show that
 \begin{equation}
  \P_\text{deg}\Bigg(\bigcup_{i\leq K_{t,t}/2}\bigcup_{k\leq \kappa^\sss{(i)}} \left\{|\CA_{k+1}^\sss{(i)}\big(\pi_{k}^\sss{(i)}\big)| \leq n_{k}^\sss{(i)}\right\}\Bigg) \leq \delta_1\big(s_{0}^\sss{(0)}\big),\label{eq:upper-connectors}
 \end{equation}
 where $\delta_1\big(s_{0}^\sss{(0)}\big)$ is a function that tends to 0 as $s_{0}^\sss{(0)}$ tends to infinity and, for $c_1>0$ chosen below,
 \be
 n_{k}^\sss{(i)}:=c_1\delta' \big(s_{k}^\sss{(i)}\big)^{\vareps_k}.\label{eq:upper-nik}
 \ee
  Then, conditioning on the complement of the event in \eqref{eq:upper-connectors}, we will bound the minimal weight of connections to $\CL_{k+1}^\sss{(i)}$ via the sets $\CA_{k+1}^\sss{(i)}\big(\pi_k^\sss{(i)}\big)$ using \eqref{eq:choice-greedy} and arrive to \eqref{eq:weighted-to-inner} and \eqref{eq:weighted-to-inner-weak} using the construction of the greedy path in \eqref{eq:choice-greedy}.
 We follow the same steps as in \cite[Lemma 3.10]{jorkom2019weighted}. For notational convenience we leave out the superscript $(i)$ for the various sequences and sets whenever it is clear from the context.
 For a set of vertices $\CV\subset [t']$, define $D_\CV^{\leftarrow}(t'):=\sum_{x\in\CV}D_x^{\leftarrow}(t')$.
 By Lemma \ref{lemma:upper-p-connector} in the appendix, the probability that an arbitrary vertex in $(\hat{t}_i,t_i]$ is in $\CA_{k+1}(\pi_k)$, is at least
 \be
 p_{k}(\pi_k, \CL_{k+1}):=\frac{1}{\hat{t}_i^2}\eta D_{\pi_{k}}^{\leftarrow}(\hat{t}_i)D^{\leftarrow}_{\CL_{k+1}}(\hat{t}_i),\label{eq:pk}
 \ee
 for some constant $\eta>0$, where this event  happens independently of the other vertices. Since the set $(\hat{t}_i,t_i]$ contains $\delta' t_i$ vertices, the random variable $|\CA_{k+1}(\pi_k)|$ stochastically dominates a binomial random variable, i.e.,
 \begin{equation}
  |\CA_{k+1}(\pi_k)| \overset{d}\geq \mathrm{Bin}\big(\delta' t_i, p_k(\pi_{k}, \CL_{k+1})\big) =: A_{k}.\label{eq:connectors-stoch-dom}
 \end{equation}
 Let $c_2$ be the constant from Lemma \ref{lemma:upper-impact}. After conditioning on $D^{\leftarrow}_{\CL_{k+1}}(\hat{t}_i)$, one obtains that
 \begin{equation}
  \E[A_k] \geq \E\big[A_k \mid D^{\leftarrow}_{\CL_{k+1}}(\hat{t}_i) \geq c_2 \hat{t}_is_{k+1}^{2-\tau}\big]
  \cdot \P\big(D^{\leftarrow}_{\CL_{k+1}}(\hat{t}_i)\geq c_2 \hat{t}_is_{k+1}^{2-\tau}\big),
  \label{eq:connector-exp-ak}
 \end{equation}
 where the latter factor equals $1-o(1)$ by Lemma \ref{lemma:upper-impact}. Since $\pi_{k}\in\CL_{k}$, we have that $D^{\leftarrow}_{\pi_{k}}(\hat{t}_i)\geq s_{k}$.
 We substitute this and the conditioned bound on $D^{\leftarrow}_{\CL_{k+1}}(\hat{t}_i)$ in \eqref{eq:connector-exp-ak} into $p_k(\pi_{k},\CL_{k+1})$ in \eqref{eq:pk}.
 By the recursive definition of $s_{k}^\sss{(i)}$ in \eqref{eq:sk} and $n_{k}^\sss{(i)}$ in \eqref{eq:upper-nik} we obtain that there exists $c_1>0$ such that
 \[
  \E[A_k] \geq \delta' t_i \frac{\eta c_2\hat{t}_i(s_{k+1}^\sss{(i)})^{2-\tau}s_k^\sss{(i)}}{\hat{t}_i^2}(1-o(1))
  \geq 2c_1\delta' (s_{k}^\sss{(i)})^{\vareps_k}=2n_{k}^\sss{(i)}.
 \]
 An application
 of Chernoff's bound and the constructed stochastic domination \eqref{eq:connectors-stoch-dom} yields for  $\pi_{k}^\sss{(i)}\in\CL_{k}^\sss{(i)}$ that
 \[
  \P_\text{deg}\big(|\CA_{k+1}^\sss{(i)}\big(\pi_{k}^\sss{(i)}\big)| \leq n_{k}^\sss{(i)}\big)
  \leq
  \exp\big(-n_{k}^\sss{(i)}/4\big).
 \]
 For details we refer the reader to \cite[Proof of Lemma 3.10]{jorkom2019weighted}.
 We return to \eqref{eq:upper-connectors}. By a union bound over the  layers $k\leq\kappa^\sss{(i)}$ and times $(t_i)_{i\leq K_{t,t}/2}$, it remains to show that
 \begin{equation}
  \sum_{i=1}^{K_{t,t}/2}\sum_{k=1}^{\kappa^\sss{(i)}}\exp\big(-n_{k}^\sss{(i)}/4\big)
  \longrightarrow 0, \qquad \text{as }s_{0}^\sss{(0)}\to\infty,\label{eq:weighted-error-1}
 \end{equation}
 with $n_{k}^\sss{(i)}$ from \eqref{eq:upper-nik}.
 We postpone showing this to the end of the proof.

 Define the conditional probability measure
 \[
  \P_\text{d,c}(\cdot):=\P\bigg(\cdot\,\Big|\,
  \CE_{\text{deg}}\cap
  \bigcap_{i=1}^{K_{t,t}/2}\bigcap_{k=1}^{\kappa^\sss{(i)}}
  \big\{|\CA_{k+1}^\sss{(i)}\big(\pi_{k}^\sss{(i)}\big)| > n_{k}^\sss{(i)}\big\}
  \bigg).
 \]
 Thus, the path segment $\ora{\bm{\pi}}_{q,1}$ from $q_0$ to the inner core exists $\P_{\text{d,c}}$-a.s.
 We greedily choose the vertices $(y_k^\sss{(i)}, \pi_{k+1}^\sss{(i)})$ as in \eqref{eq:choice-greedy}.
 We bound the weight of the segment, i.e., the rhs of \eqref{eq:greedy-weight} to prove \eqref{eq:weighted-to-inner} and \eqref{eq:weighted-to-inner-weak}.
 Let $L_{m,n}^\sss{(i)}$ be i.i.d.\ copies of $L$. Since the minimum of $N$ i.i.d.\ random variables is nonincreasing in $N$,
 the weighted distance between $\pi_{k}^\sss{(i)}$ and $\pi_{k}^\sss{(i)}$ can be bounded for $k\leq \kappa^\sss{(i)}-1$, i.e.,  for $k\leq \kappa^\sss{(i)}$ and $i\leq K_{t,t}/2$, $\P_{\text{d,c}}$-a.s.\
 \[
  d_L^{(t_i)}\big(\pi_{k}^\sss{(i)}, \pi_{k+1}^\sss{(i)}\big) \leq \min_{j\in[n_{k}^\sss{(i)}]} \big(L_{1,j}^\sss{(i)} + L_{2,j}^\sss{(i)}\big).
 \]
 Applying $(\star)$ in \eqref{eq:methods-of-minima} obtains for $\xi\in(0,1)$ that
 \be
 \P_{\text{d,c}}\Big(d_L^{(t_i)}\big(\pi_{k}^\sss{(i)}, \pi_{k+1}^\sss{(i)}\big)\geq F_{L_1 + L_2}^{(-1)}\Big(\big(n_{k}^\sss{(i)}\big)^{-1+\xi}\Big)\Big)
 \leq
 \exp\big(-\big(n_{k}^\sss{(i)}\big)^\xi\big),\nonumber
 \ee
 where $F_{L_1 + L_2}^{(-1)}$ denotes the generalized inverse of the distribution of the sum of two i.i.d.\ copies of $L$.
 Recall the bound \eqref{eq:greedy-weight}. By a union bound over the subsegments $(\pi_k^\sss{(i)}, y_k^{i}, \pi_{k+1}^\sss{(i)})$ for $k\leq\kappa^{(i)}$ and the times $(t_i)_{i\leq K_{t,t}/2}$,
 \begin{align}
  \P_{\text{d,c}}\Bigg(\bigcup_{i\leq K_{t,t}/2}\Big\{
  d_L^{(t_i)}(q_0, \mathrm{Core}^\sss{(i)})
  \geq
  \sum_{k=0}^{\kappa^\sss{(i)}-1} & F_{L_1 + L_2}^{(-1)}\Big(\big(n_{k}^\sss{(i)}\big)^{-1+\xi}\Big)
  \Big\}\Bigg)\nonumber                                                                              \\
                                  & \leq
  \sum_{i=0}^{K_{t,t}/2}\sum_{k=0}^{\kappa^\sss{(i)}-1}\exp\big(-\big(n_{k}^\sss{(i)}\big)^\xi\big).\label{eq:weighted-error-2}
 \end{align}
 To bound the error probabilities in \eqref{eq:weighted-error-1} and \eqref{eq:weighted-error-2}, and the sum inside the event in \eqref{eq:weighted-error-2}, we bound $n_k^\sss{(i)}$ from below. By its definition in \eqref{eq:upper-nik}, the bound on $t_i$ in \eqref{eq:lower-ti} and the bound on $s_{k}^\sss{(i)}$ in \eqref{eq:bound-sik}
 \begin{align}
  n_{k}^\sss{(i)}= c_1\delta'\big(s_{k}^\sss{(i)}\big)^{\vareps_k}
   & \geq
  c_1\delta'\left(\delta' s_{0}^\sss{(0)}\left(\hat{t}_i/t\right)^{1/(\tau-1)}\right)^{c'(\tau-2)^{-k}\varepsilon_k} \nonumber \\
  &=
c_1\delta'\left(\delta' s_{0}^\sss{(0)}(1-\delta')^{1/(\tau-1)}\exp\Big(\frac{\tau-2}{\tau-1}(\tau-2)^{-i}\Big)\right)^{c'(\tau-2)^{-k}\varepsilon_k} \nonumber
 \end{align}
 Assuming that  $s_0^\sss{(0)}\ge \delta'^2(1-\delta')^{2/(\tau-1)}$, we obtain since $\varepsilon_k=(k+1)^{-2}$ by definition above \eqref{eq:sk-init}
 \begin{align}
   n_{k}^\sss{(i)}&\ge
 c_1\delta'\left(\exp\Big(\frac{1}{2}\log(s_0^\sss{(0)})+\frac{\tau-2}{\tau-1}(\tau-2)^{-i}\Big)\right)^{c'(\tau-2)^{-k}\varepsilon_k}\nonumber\\
 &=c_1\delta'\exp\Big(\frac{c'}{2}\log(s_0^\sss{(0)})(k+1)^{-2}(\tau-2)^{-k}+\frac{\tau-2}{\tau-1}(k+1)^{-2}(\tau-2)^{-(i+k)}\Big).\label{eq:nki-upper}
\end{align}
 Since $(\tau-2)^{-k}$ grows exponentially for $\tau\in(2,3)$, while $\varepsilon_k=(1+k)^{-2}$ decreases polynomially, there exist $c_3>0$, $c_4>1$ such that
$
 n_k^\sss{(i)}
   \geq
 c_1\delta'\exp(\log(s_0)c_3c_4^{k} + c_3c_4^{k+i}).
$
Substituting this bound into \eqref{eq:weighted-error-1} and \eqref{eq:weighted-error-2}, respectively, we observe that the terms are summable in both $i$ and $k$ and tend to zero as $s_{0}^\sss{(0)}$ tends to infinity.

It is left to relate the sum on the rhs in the event in \eqref{eq:weighted-error-2} to the rhs in \eqref{eq:weighted-to-inner} and \eqref{eq:weighted-to-inner-weak}, respectively, assuming that $s_0^\sss{(0)}$ is large.
Recalling \eqref{eq:nki-upper}, we assume $s_0^\sss{(0)}$ is sufficiently large so that there exists $c_5>0$
\begin{align}
  (n_{k}^\sss{(i)})^{(1-\xi)}&\ge \exp\big(2c_5(k+1)^{-2}(\tau-2)^{-(i+k)}\big).\nonumber
\end{align}
Since $z\mapsto F_L^\sss{(-1)}(1/z)$ is nonincreasing and $\kappa^\sss{(i)}\le K_{t,t_i}/2 + M$ by \eqref{eq:bound-kti}, we obtain
\begin{align}
  \sum_{k=0}^{\kappa^\sss{(i)}-1}  F_{L_1 + L_2}^{(-1)}\Big(\big(n_{k}^\sss{(i)}\big)^{-1+\xi}\Big)
  \le
  \sum_{k=0}^{K_{t,t_i}/2 + M-1} F_{L_1 + L_2}^{(-1)}\Big(\exp\big(-2c_5(k+1)^{-2}(\tau-2)^{-(i+k)}\big)\Big).\label{eq:weighted-error-3}
\end{align}
Recall $\varepsilon>0$ from the statement of Proposition \ref{proposition:weighted-to-inner}. In Claim \ref{claim:upper-bound-transforms} we show that for all $\varepsilon>0$ there exists $M>0$ such that
 \begin{equation}
\sum_{k=0}^{K_{t,t_i}/2 + M-1} F_{L_1 + L_2}^{(-1)}\Big(\exp\big(-c_5(k+\zeta)^{-2}(\tau-2)^{-(i+k)}\big)\Big)
   \le
   (1+\varepsilon\ind{\bm{I}_2(L)=\infty})Q_{t,t_i} + M.
 \end{equation}
 Substituting this bound inside the event in \eqref{eq:weighted-error-2} and recalling that $s_{0}^\sss{(0)}$ is chosen sufficiently large so that the total error probability from \eqref{eq:weighted-error-1} and \eqref{eq:weighted-error-2} is at most $C\delta'$ yields Proposition \ref{proposition:weighted-to-inner}.
\end{proof}

\subsection*{Step 4. Bridging the inner core}
We prove a lemma that shows that the path segments $(\bm{\pi}_{\text{core}}^\sss{(i)})_{i\leq K_{t,t}/2}$  exist and their total weight is  bounded from above by a constant for all $i$. Recall $\mathrm{Core}^\sss{(i)}$ from \eqref{eq:core}.
\begin{lemma}\label{lemma:upper-core}
 Consider the preferential attachment model with power-law exponent $\tau\in(2,3)$.
 Let $\{w_{u}^\sss{(i)},w_{v}^\sss{(i)}\}_{i\leq K_{t,t}/2}$ be a set of vertices such that for all $i$,  $w_{u}^\sss{(i)},w_{v}^\sss{(i)}\in\mathrm{Core}^\sss{(i)}$.
 Then for every $\delta'>0$, there exists $M>0$ such that
 \begin{equation}
  \P\Bigg(\bigcup_{i\leq K_{t,t}/2}\Big\{ d_L^{(t_i)}(w_{u}^\sss{(i)}, w_{v}^\sss{(i)})\geq M\Big\}\Bigg)\leq \delta'. \label{eq:upper-core}
 \end{equation}
 \begin{proof}
  From \cite[Proposition 3.2]{dommers2010diameters} it follows for FPA that for fixed $i$, whp,
  \be
  \P\Big(\text{diam}_G^{(t_i)}(\text{Core}^\sss{(i)})\leq \frac{2(\tau-1)}{3-\tau}+6\Big) \geq 1-o(1/t_i),\label{eq:diam-g-inner}
  \ee
  where  $\text{diam}^{(t')}_G(\CV):=\max_{w_1,w_2\in\CV} d_G^{(t')}(w_1,w_2)$ for a set of vertices $\CV\subset[t']$. The statement \eqref{eq:diam-g-inner} holds also for VPA as explained in \cite[Proof of Proposition 3.5]{jorkom2019weighted}.
  A union bound yields that
  \[
   \P\Bigg(\bigcup_{i\leq K_{t,t}/2}\Big\{ \text{diam}_G(\text{Core}^\sss{(i)})> \frac{2(\tau-1)}{3-\tau}+6\Big\}\Bigg) = \sum_{i=1}^{K_{t,t}/2} o(1/t_i)=o(1),
  \]
  since $K_{t,t}=O(\log\log(t))$, and $t_i$ is increasing in $i$.
  We sketch how to extend this result to weighted distances, using the construction in the proof of  \cite[Proposition 3.2]{dommers2010diameters} which in turn relies on \cite[Chapter 10]{bollobas2001random}.
  In \cite[Proposition 3.2]{dommers2010diameters} it is shown that the inner core dominates an Erd\H{o}s-R\'enyi random graph (ERRG)
  $G(n_{i}, p_{i})$, where there is an edge between two vertices $x,y\in\mathrm{Core}^\sss{(i)}$ if there is a connector in $[(1-\delta')t_i, t_i]$ in $\mathrm{PA}_{t_i}$, where
  \[
   n_{i} = \sqrt{t_i},\qquad p_{i} = \frac{1}{2}t_i^{\frac{1}{(\tau-1)}-1}\log^{-2}(t_i).
  \]
  The weight on the edge $(w_1,w_2)$ in the ERRG is $L_{(i,y)} + L_{(y,j)}$, where $y$ is a uniformly chosen connector of $w_1$ and $w_2$ in $\mathrm{PA}_{t_i}$.
  Now, for the construction used in \cite[Chapter 10]{bollobas2001random}, one can embed two $r_i$-regular trees of depth $\Delta>0$ in the ERRG, rooted in  $w_{u}^\sss{(i)}$ and $w_{v}^\sss{(i)}$, respectively, whp. Here $r_i \geq t_i^a,$ for some $a>0$, and $\Delta$ is a constant such that all vertices at distance $\Delta$ from their root are members of both trees.
  Denote this event by $\CE_{\text{tree}}$.
  On this event, there are at least $r_i$ disjoint paths from $w_{u}^\sss{(i)}$ to $w_{v}^\sss{(i)}$ in $\mathrm{PA}_{t_i}$ of $4\Delta$ edges, and we can bound
  \[d_L^{(t_i)}(w_{u}^\sss{(i)}, w_{v}^\sss{(i)})\leq \min_{n\leq r_i}\sum_{j=1}^{4\Delta}L_{j}^{(n)},\]
  for i.i.d.\ copies of $L$. Moreover, for $F_{L_1+\dots L_{4\Delta}}^{(-1)}$ being the distribution of the sum of $4\Delta$ i.i.d.\ copies of $L$, for $C$ sufficiently large
  \begin{equation}
   \P\Bigg(\bigcup_{i\leq K_{t,t}/2}\Big\{\min_{j\leq t_i^a}\sum_{j=1}^{4\Delta}L_{j}^{(n)} \geq C\Big\}\Bigg)
   \leq
   \sum_{i=1}^{K_{t,t}/2}\big(1-F_{L_1+\dots L_{4\Delta}}(C)\big)^{t_i^a} <\delta',\nonumber
  \end{equation}
  since by choosing $C$ large, but independently of $t$, $F_{L_1+\dots L_{4\Delta}}(C)$ can be brought arbitrarily close to 1. The asserted bound \eqref{eq:upper-core} follows from a union bound over the above event and $\neg\CE_{\text{tree}}$.
 \end{proof}
\end{lemma}
\subsection*{Step 5. Gluing the segments}
We are ready to prove the main proposition of this section.
\begin{proof}[Proof of Proposition \ref{proposition:upper-bound}]
 Recall Lemma \ref{lemma:upper-sub}. We have to show that at the times $(t_i)_{i\geq0}$ there is a path from $u$ to $v$ such that its total weight is bounded from above by the rhs between brackets in \eqref{eq:upper-sub} and \eqref{eq:upper-sub-weak}, w/p at least $1-\delta$.
 Let $C_{\ref{proposition:weighted-to-inner}}$ be the constant from Proposition \ref{proposition:weighted-to-inner}.
 Set $\delta'=\delta/(7+2C_{\ref{proposition:weighted-to-inner}})$.
 Let $M_{\ref{lemma:upper-suff-degree}}$ and $M_{\ref{proposition:weighted-to-inner}}$ be the constants obtained from applying Lemma   \ref{lemma:upper-suff-degree} and Proposition \ref{proposition:weighted-to-inner} for $\delta'$, respectively.
 Lastly, let $M_{\ref{lemma:upper-core}}$ be the constant from applying Lemma \ref{lemma:upper-core} for $\delta'$.
 The existence of the path segments follows now directly from a union bound over the events described in Lemma \ref{eq:upper-suff-degree} and Proposition \ref{proposition:weighted-to-inner} for $q\in\{u,v\}$, and Lemma \ref{lemma:upper-core}. Hence, the summed error probability is
 $(2\cdot3+ 2C  + 1)\delta'=\delta$. The total weight of the constructed paths $(\bm{\pi}^\sss{(i)})_{i\geq0}$ is bounded from above by $2Q_{t,t_i}+2(M_{\ref{lemma:upper-suff-degree}}+M_{\ref{proposition:weighted-to-inner}}) + M_{\ref{lemma:upper-core}}$ for all $i\geq0$.
 Thus, setting $M_L:=2(M_{\ref{lemma:upper-suff-degree}}+M_{\ref{proposition:weighted-to-inner}}) + M_{\ref{lemma:upper-core}}$ in the statement of Proposition \ref{proposition:upper-bound} finishes the proof.
\end{proof}
\section{Acknowledgements}
The work of JJ and JK is partly supported by the Netherlands Organisation for Scientific Research (NWO) through grant NWO 613.009.122. We thank the referees for their careful reading and comments that led to significant improvements in the presentation and a remark about the summation interval for $Q_{t,t'}$ below Theorem \ref{thm:weighted-evolution}.
\appendix
\section{Preliminaries}
In order to bound the minimum of a sequence of i.i.d.\ random variables, we need the following lemma that we cite from \cite[Lemma 3.11]{jorkom2019weighted}.
\begin{lemma}[Minimum of i.i.d.\ random variables, {\cite[Lemma 3.11]{jorkom2019weighted}}]\label{lemma:methods-of-minima}
 Let $L_1,\dots,L_n$ be i.i.d.\ random variables having distribution $F_L$. Then for all $\xi>0$
 \begin{equation}
  \P\Big(\min_{j\in[n]} L_j \geq F_L^{(-1)}\big(n^{-1+\xi}\big)\Big) \overset{(\star)}{\leq}\re^{-n^\xi}, \qquad
  \P\Big(\min_{j\in[n]} L_j \leq F_L^{(-1)}\big(n^{-1-\xi}\big)\Big) \overset{(\ast)}{\leq}n^{-\xi}.\label{eq:methods-of-minima}
 \end{equation}
 \begin{proof}
  Since the random variables are i.i.d., for a function $z(n)$
  \[
   \P\Big(\min_{j\in[n]} L_j \geq F_L^{(-1)}\big(n^{-1+\xi}\big)\Big) = \big(1-F_L(z(n))\big)^{n}.
  \]
  We substitute $z(n)=F_L^{(-1)}\big(n^{-1\pm\xi}\big)$, so that applying $(1-x)^n\leq e^{-nx}$ yields $(\star)$ in \eqref{eq:methods-of-minima},
  and applying $(1-x)^n\geq 1-nx$ yields $(\ast)$.
 \end{proof}
\end{lemma}
\subsection{Lower bound}
\begin{proof}[Proof of Lemma \ref{lemma:number-of-paths}]
 We verify \eqref{eq:lower-paths-from-t} by induction.
 Recall the initial values in \eqref{eq:phi} and \eqref{eq:psi}. We initialize the induction for $j=1$. By \eqref{eq:pa-gamma}, since $x<t'$
 \[
  f^\sss{(t,t')}_{[i, i+1)}(t', x) = p(t',x)= \nu x^{-\gamma}t'^{\gamma-1} = \nu t'^{\gamma-1}\cdot x^{-\gamma} + 0\cdot x^{\gamma-1},
 \]
 establishing \eqref{eq:lower-paths-from-t} for $j=1$.
 We advance the induction so that we may assume \eqref{eq:lower-paths-from-t} for $j=k$.
 Then, using the definition of $f$ in \eqref{eq:lower-f}, which counts only the good paths and relies on the product form of $p$ in \eqref{eq:pa-gamma}, we can write
 \begin{align}
  f^\sss{(t,t')}_{[i, i+k+1)}(t', x) & \leq \sum_{z=\ell_{i+k, t'}}^{t'}f^\sss{(t,t')}_{[i, i+k)}(t', z)p(z, x). \nonumber
 \end{align}
 This bound does not hold with equality, because the first factor on the rhs counts the good self-avoiding paths from $t'$ to $z$,
 but the vertex $x$ is not necessarily excluded in these paths, while these paths are excluded on the lhs. Recall from \eqref{eq:pa-gamma} that $p(z,x)=\nu(x\wedge z)^{-\gamma}(x\vee z)^{\gamma-1}$.
 Since $f$ only counts the good paths, observe that if $z<x$, then $x>\ell_{i+k,t'}$. Thus, splitting the sum in two, whether $z\geq x$ or $z<x$, one obtains that
 \begin{align}
  f^\sss{(t,t')}_{[i, i+k+1)}(t', x)
   & \leq \nu x^{-\gamma}\sum_{z=\ell_{i+k, t'}\vee x}^{t'}f^\sss{(t,t')}_{[i, i+k)}(t', z) z^{\gamma-1} \label{eq:lower-bound-f-in-proof1} \\
   & \hspace{15pt}+ \ind{x>\ell_{i+k,t'}}\nu x^{\gamma-1}\sum_{z=\ell_{i+k,t'}}^{x-1}f^\sss{(t,t')}_{[i, i+k)}(t', z)z^{-\gamma}. \nonumber
 \end{align}
 By the induction hypothesis \eqref{eq:lower-paths-from-t}, we have that
 \begin{align*}
  f^\sss{(t,t')}_{[i, i+k+1)}(t', x)
   & \leq \nu x^{-\gamma}\phi_{[i, i+k)}\sum_{z=\ell_{i+k, t'}\vee x}^{t'}z^{-1} + \nu x^{-\gamma}\psi_{[i, i+k)}\sum_{z=\ell_{i+k-1, t'}\vee x}^{t'}z^{2\gamma - 2} \\
   & \hspace{-5pt} +
  \ind{x>\ell_{i+k,t'}}\nu x^{\gamma-1}\phi_{[i, i+k)}\sum_{z=\ell_{i+k,t'}}^{x-1}z^{-2\gamma} + \ind{x>\ell_{i+k,t'}}\nu x^{\gamma-1}\psi_{[i, i+k)}\sum_{z=\ell_{i+k-1,t'}}^{x-1}z^{-1},
 \end{align*}
 where the lower summation bounds in the second sum on both rows changed as a result of the indicator in \eqref{eq:lower-paths-from-t}.
 Approximating the sums by integrals obtains that there exists $c=c(\nu,\gamma)$ such that
 \begin{align*}
  f^\sss{(t,t')}_{[i, i+k+1)}(t', x)
   & \leq
  x^{-\gamma}c\left(\phi_{[i, i+k)}\log(t'/\ell_{i+k-1,t'}) + \psi_{[i,i+k)}t'^{2\gamma-1}\right)                                                   \\
   & \hspace{15pt}+ \ind{x>\ell_{i+k,t'}}x^{\gamma-1}c\left(\phi_{[i,i+k)}\ell_{i+k,t'}^{1-2\gamma} + \psi_{[i,i+k)}\log(t'/\ell_{i+k-1,t'})\right)
 \end{align*}
 and \eqref{eq:lower-paths-from-t} holds by the definitions \eqref{eq:phi} and \eqref{eq:psi}, as shifting the index of the terms in the logarithm is allowed because $k\mapsto \ell_{k,t'}$ is nonincreasing.
 The bound \eqref{eq:lower-paths-from-x} follows analogously. The first indicator follows since the sum on the rhs in the analogue of \eqref{eq:lower-bound-f-in-proof1} is equal to zero if $x=t'$, since $p(t',t')=0$ by definition in Proposition \ref{prop:pa-gamma}, i.e.,
 \[
  \sum_{z=\ell_{i+k,t'}\vee t'}^{t'}f_{[0,k)}^\sss{(t, t')}(q, z)p(z,t') = f_{[0,k)}^\sss{(t, t')}(q, t')p(t',t')= f_{[0,k)}^\sss{(t, t')}(q, t')\cdot 0=0.\qedhere
 \]
\end{proof}

\begin{lemma}[Upper bound on $t'/\ell_{k,t'}^\sss{(t)}$]\label{lemma:eta-k-appendix}
 There exists $B_{\ref{lemma:eta-k-appendix}}=B_{\ref{lemma:eta-k-appendix}}(\gamma, \nu)$
 such that for $B>B_{\ref{lemma:eta-k-appendix}}$,
 \begin{equation*}
  t'/\ell_{k,t'}^\sss{(t)} \leq \exp\Big(B\big(1\vee\log(t'/t)\big)\big(\tau-2\big)^{-k/2}\Big).
 \end{equation*}
\end{lemma}
\begin{proof}
 Let $\gamma=1/(\tau-1)$ so that $1/(\tau-2)=\gamma/(1-\gamma)$. We omit the superscript $(t)$ of $\ell_{k,t'}^{(t)}$. We prove by induction. For the induction base $k=0$,
 $
  t'/\ell_{0,t'} \leq t'/(\delta't).
 $
 The advancement of the induction follows from \cite[Lemma A.5, after (A.29)]{caravenna2016diameterArxiv}, which contains the appendices of \cite{caravenna2016diameter}.
 Recall $\alpha_{[0,k)}$, $\beta_{[0,k)}$, $\ell_{k,t'}$ from
 \eqref{eq:alpha}, \eqref{eq:beta}, and \eqref{eq:ell}, respectively. Write
 $\eta_{k,t'}:=t'/\ell_{k,t'}$,  and let $c=c(\gamma,\nu)$ be the constant from Lemma \ref{lemma:number-of-paths}.
 To use the same calculations as \cite[Lemma A.5, after (A.29)]{caravenna2016diameterArxiv}, we need to show that there exists $C=C(\gamma,\nu)>1$, such that
 \begin{equation}
  \left(\eta_{k+2,t'}^{-1} + 1/t'\right)^{\gamma-1}
  \leq
  C\left(\eta_{k,t'}^\gamma + \eta_{k+1,t'}^{1-\gamma}\log(\eta_{k+1,t'})\right).
  \label{eq:eta_recursive_upper}
 \end{equation}
 We start bounding the lhs.
 Observe that $(\circledast)$ in \eqref{eq:ell} holds by definition of the $\argmax$ in the opposite direction when
 we replace $\ell_{k,t'}$ by $\ell_{k,t'}+1$, i.e.,
 \[
  \alpha_{[0,k)}(\ell_{k,t'}^\sss{(t)}+1)^{1-\gamma}\geq \big(k\log(t')\big)^{-3}.
 \]
 Combining this with \eqref{eq:alpha} yields
 \begin{align}
  \left(\frac{\ell_{k+2,t'}+1}{t'}\right)^{\gamma-1}
   & \leq \log^3(t')(k+2)^3\alpha_{[0,k+2)}t'^{1-\gamma} \nonumber                                                               \\
   & = c\log^3(t')(k+2)^3\alpha_{[0,k+1)}t'^{1-\gamma}\log(\eta_{k+1,t'}) + c\log^3(t')(k+2)^3\beta_{[0,k+1)}t'^\gamma \nonumber \\
   & =: T_1 + T_2.\label{eq:eta_k_recursive_t1_t2}
 \end{align}
 Substituting $(\circledast)$ from \eqref{eq:ell} in $T_1$ obtains
 \begin{align}
  T_1 & :=c\log^3(t')(k+2)^3\alpha_{[0,k+1)}t'^{1-\gamma}\log(\eta_{k+1,t'}) \nonumber         \\
      & \leq c(k+2)^3\frac{1}{(k+1)^3}\ell_{k+1,t'}^{\gamma-1}t'^{1-\gamma}\log(\eta_{k+1,t'})
  \leq c \left(\frac{k+2}{k+1}\right)^3\eta_{k+1,t'}^{1-\gamma}\log(\eta_{k+1,t'}).\label{eq:eta_k_recursive_t1}
 \end{align}
 Hence, $T_1$ is bounded by the second term on the rhs in \eqref{eq:eta_recursive_upper} for $C$ sufficiently large. For $T_2$, we substitute \eqref{eq:beta} and \eqref{eq:ell} to get
 \begin{align*}
  T_2:=c\log^3(t')(k+2)^3\beta_{[0,k+1)}t'^\gamma
   & =c^2\log^3(t')(k+2)^3\left(t'^\gamma\alpha_{[0,k)}\ell_{k,t'}^{1-2\gamma} + t'^\gamma\beta_{[0,k)}\log(\eta_{k,t'})\right) \\
   & \leq c^2 \left(\frac{k+2}{k}\right)^3 \eta_{k,t'}^\gamma + c^2\log^3(t')(k+2)^3t'^\gamma\beta_{[0,k)} \log(\eta_{k,t'})    \\
   & =: cT_{21} + cT_{22}.
 \end{align*}
 The term $cT_{21}$ can be captured by the first term on the rhs in \eqref{eq:eta_recursive_upper}. For $T_{22}$, by \eqref{eq:alpha}, $\beta_{[0,k)}t'^\gamma \leq c^{-1}\alpha_{[0,k+1)}t'^{1-\gamma}$.
 Using \eqref{eq:ell} and that $\eta_{k,t'}$ is nonincreasing, we further bound
 \begin{align}
  T_{22}
  \leq \log^3(t')(k+2)^3\log(\eta_{k,t'})\alpha_{[0,k+1)}t'^{1-\gamma}
   & \leq (k+2)^3\log(\eta_{k,t'})\frac{1}{(k+1)^3}\ell_{k+1,t'}^{\gamma-1}t'^{1-\gamma} \nonumber                 \\
   & \leq  \left(\frac{k+2}{k+1}\right)^3\log(\eta_{k+1,t'})\eta_{k+1,t'}^{1-\gamma}.\label{eq:eta_k_recursive_t2}
 \end{align}
 Thus, $cT_{22}$ can be captured by the second term on the rhs in \eqref{eq:eta_recursive_upper} for $C$ sufficiently large.
 The desired bound \eqref{eq:eta_recursive_upper} follows by combining \eqref{eq:eta_k_recursive_t1_t2}, \eqref{eq:eta_k_recursive_t1}, and \eqref{eq:eta_k_recursive_t2} by increasing the constant $C$ in \eqref{eq:eta_recursive_upper}. Now that we have established \eqref{eq:eta_recursive_upper}, the proof is finished by step-by-step following the computations in \cite[Lemma A.5, after (A.29)]{caravenna2016diameterArxiv}.
\end{proof}

\begin{lemma}[Upper bound on neighbourhood size]\label{lemma:lower-neighbourhoods}
 Consider the preferential attachment model under the same assumptions as Proposition \ref{proposition:lower-bound-weighted} and recall $\CE_{\mathrm{bad}}^{(q)}$ from \eqref{eq:lower-bad-def}.
 Let $q$ be a typical vertex in $\mathrm{PA}_t$. Then for $t'\geq t$
 \begin{align}
  \P\Bigg(\bigcup_{k=1}^{\underline{K}_{t,t'}+2}\Big\{|\partial\CB_{G}^{(t')}(q,k)| \geq \exp\big(2B(1\vee\log(t'/t))(\tau- & 2)^{-k/2}\big)\Big\}\mid \bigcap_{k=0}^{\underline{K}_{t,t'}+2}\neg\CE_{\mathrm{bad}}^{(q)}\Bigg)\nonumber \\
                                                                                                                            & \leq
  2\exp\big(-B(1\vee\log(t'/t))\big). \label{eq:lower-neighbourhoods}
 \end{align}
 \begin{proof}
  Define $\CE_{\mathrm{good}}^{(q)}:=\bigcap_{k=0}^{\underline{K}_{t,t'}}\neg\CE_{\mathrm{bad}}^{(q)}$. We will first bound
  $\E\big[|\partial\CB_{G}^{(t')}(q,k)| \mid \CE_{\mathrm{good}}^{(q)}\big]$
  and let the result follow by a union bound on the events in \eqref{eq:lower-neighbourhoods} and Markov's inequality.
  Conditionally on $\CE_{\mathrm{good}}^{(q)}$, all vertices at distance $k<\underline{K}_{t,t'}$ can only be reached via $t'$-good $q$-paths.
  Recall $f_{[0,k)}^\sss{(t,t')}(q,x)$ from \eqref{eq:lower-f} and its interpretation as an upper bound for the expected number of good paths from $q$ to $x$ of length $k$.
  Thus, we have by the law of total probability, and the definition of good paths in Definition \ref{def:t-bad},
  \begin{equation}
   \E\big[|\partial\CB_{G}^{(t')}(q,k)| \mid \CE_{\mathrm{good}}^{(q)}\big] \leq \frac{1}{\P\big(\CE_{\mathrm{good}}^{(q)}\big)}\sum_{x=\ell_{k,t'}}^{t'}f_{[0,k)}^\sss{(t,t')}(q,x).\nonumber
  \end{equation}
  Recalling \eqref{eq:lower-proof-t2}, we see that it is sufficient to bound the sum on the rhs.
  Now, applying the bound in \eqref{eq:lower-paths-from-x} on $f$ yields for some $c_1,c_2>0$,
  \begin{align}
   \sum_{x=\ell_{k,t'}}^{t'}f_{[0,k)}^\sss{(t,t')}(q,x)
    & \leq
   \alpha_{[0,k)} \sum_{x=\ell_{k,t'}}^{t'}x^{-\gamma} + \beta_{[0,k)}\sum_{x=\ell_{k-1,t'}}^{t'}x^{\gamma-1}\nonumber \\
    & \leq
   c_1\left(\alpha_{[0,k)}t'^{1-\gamma} + \beta_{[0,k)}t'^{\gamma}\right)
   \leq c_2\alpha_{[0,k+1)}t'^{1-\gamma}. \label{eq:neighbourhoods-inproof}
  \end{align}
  The last line follows since by \eqref{eq:alpha}, $t'^{\gamma}\beta_{[0,k)}\leq \alpha_{[0,k+1)}t'^{\gamma-1}/c$, and $k\mapsto \alpha_{[0,k)}$ is non-decreasing.
  We bound the rhs in \eqref{eq:neighbourhoods-inproof} in terms of $(t'/\ell_{k,t'})$. By $(\circledast)$ in \eqref{eq:ell} and Lemma \ref{lemma:eta-k-appendix}
  \begin{align*}
   \alpha_{[0,k+1)}t'^{1-\gamma}
    & \leq
   ((k+1)\log(t'))^{-3}(t'/\ell_{k+1,t'})^{1-\gamma} \\
    & \leq
   ((k+1)\log(t'))^{-3}\exp\Big(B(1-\gamma)\big(1\vee \log(t'/t)\big)(\tau-2)^{-(k+1)/2}\Big).
  \end{align*}
  This obtains that for $B$ sufficiently large
  \begin{align*}
   \E\big[|\partial\CB_{G}^{(t')}(q,k)| \mid \CE_{\mathrm{good}}^{(q)}\big]
    & \leq
   \frac{2c}{\P\big(\CE_{\mathrm{good}}^{(q)}\big)}\exp\Big(B(1\vee\log(t'/t))(\tau-2)^{-k/2}\Big)\nonumber \\
    & \leq \exp\Big(B'(1\vee\log(t'/t))(\tau-2)^{-k/2}\Big),
  \end{align*}
  where the last bound follows for some $B'>B$ as $\P\big(\CE_{\mathrm{good}}^{(q)}\big)=1-\delta'+o(1)$ by \eqref{eq:lower-proof-t2}.
  The assertion \eqref{eq:lower-neighbourhoods} follows by a union bound over \eqref{eq:lower-neighbourhoods} and summing over $k$.
 \end{proof}
\end{lemma}

\subsection{Upper bound}
This lemma is establishing the probability for a vertex to be a two-connector, and is cited from \cite{dommers2010diameters}.
\begin{lemma}[{\cite[(3.6) in proof of Proposition 3.2]{dommers2010diameters}}]\label{lemma:upper-p-connector}
 For $x\in [\hat{t}_i]$, a set $\CV\in[\hat{t}_i]$, conditionally on $\mathrm{PA}_{\hat{t}_i}$, the probability that
 $y\in(\hat{t}_i, t_i]$ is a connector of $(x, \CV)$ is at least
 \begin{equation}
  \frac{\eta D_x(\hat{t}_i)D_{\CV}(\hat{t}_i)}{\hat{t}_i^2} =: p(x, \CV),\nonumber
 \end{equation}
 where $\eta>0$ is a constant and $D_{\CV}(\hat{t}_i):=\sum_{z\in\CV}D_{z}(\hat{t}_i)$. Moreover, w/p at least $p(x,\CV)$,
 the event $\{y \text{ is a connector of } (x,\CV)\}$ happens independently of other vertices in $(\hat{t}_i,t_i]$.
\end{lemma}
Recall $\hat{t}_i$ from \eqref{eq:core}, $s_k^\sss{(i)}$ from \eqref{eq:sk}, and $\CL_{k}^\sss{(i)}$ from \eqref{eq:layers}. The last lemma bounds the total degree of vertices with degree at least $s_k^\sss{(i)}$ at time $\hat{t}_i$.
\begin{lemma}[Impact of high degree vertices {\cite[Lemma A.1]{dommers2010diameters}}]\label{lemma:upper-impact}
 There is a constant $c>0$ such that for any $\vareps>0$
 \[
  \P\big(D_{\CL_{k}^\sss{(i)}}(\hat{t}_i) \leq c\hat{t}_i (s_{k}^\sss{(i)})^{2-\tau} \big) = o(1).
 \]
 \begin{proof}
  The proof for FPA can be found in \cite[Lemma A.1]{dommers2010diameters}, we refer the reader there to fill in the details.
  For VPA it follows from \cite[Theorem 1.1(a)]{dereich2009random}.
 \end{proof}
\end{lemma}

\begin{claim}\label{claim:upper-bound-transforms}
  Recall $K_{t,t_i}$ from \eqref{eq:kt}, $\bm{I}_2(L)$ from \eqref{eq:exp-crit}, and $Q_{t,t_i}$ from \eqref{eq:qt-kt}. Let $L_1$ and $L_2$ be two independent copies of the random variable $L$.
  For all $c_5, \varepsilon, M>0$, there exists $M_L>0$ such that
   \begin{equation}
     \sum_{k=0}^{K_{t,t_i}/2 + M-1} F_{L_1 + L_2}^{(-1)}\Big(\exp\big(-2c_5(k+1)^{-2}(\tau-2)^{-(i+k)}\big)\Big)
     \le
     (1+\varepsilon\ind{\bm{I}_2(L)=\infty})Q_{t,t_i} + M_L.\label{eq:claim-upper}
   \end{equation}
 \begin{proof}
   We proceed along the same lines as in \cite[Proof of Proposition 3.4]{jorkom2019weighted}. To shorten notation, we define
   \begin{align}
     \widetilde Q_{t,t_i}:=\sum_{k=0}^{K_{t,t_i}/2 + M-1} F_{L_1 + L_2}^{(-1)}\Big(\exp\big(-2c_5(k+1)^{-2}(\tau-2)^{-(i+k)}\big)\Big)\nonumber
   \end{align}
   First, we relate the inverse  $F_{L_1+L_2}^{(-1)}(\cdot)$ to $F_L^{(-1)}(\cdot)$ by observing that for $x>0$
   \be
   F_{L_1+L_2}(x)=\P(L_1 + L_2\le x)\ge \P(\max\{L_1,L_2\}\le x/2) = \big(F_L(x/2)\big)^2.\nonumber
   \ee
   Hence, for any $z>0$, it holds that $F_{L_1+L_2}(z)\le 2F_L^{(-1)}(\sqrt{z})$, so that for some $M_1>0$
   \begin{align}
    \widetilde Q_{t,t_i}
    &\le
    2 \sum_{k=0}^{K_{t,t_i}/2 + M-1}  F_{L}^{(-1)}\Big(\exp\big(-c_5(k+1)^{-2}(\tau-2)^{-(i+k)}\big)\Big)\nonumber\\
    &\le
    M_1 + 2 \sum_{k=0}^{K_{t,t_i}/2}  F_{L}^{(-1)}\Big(\exp\big(-c_5(k+1)^{-2}(\tau-2)^{-(i+k)}\big)\Big),\nonumber
   \end{align}
   since $z\mapsto F_L^{(-1)}(1/z)$ is nonincreasing and bounded.
  Define $b:=\inf\{x: F_L^{(-1)}(x)>0\}$ and $L':=L-b$, so that by $(\star)$ in \eqref{eq:integral-sum-tricks}
  \begin{align}
   \widetilde Q_{t,t_i}
    & \le
   M_1 + bK_{t,t_i} + 2\int_{x=0}^{K_{t,t_i}/2}  F_{L'}^{(-1)}\big(\exp\big(-c_5 (x+1)^{-2}(\tau-2)^{-i-x}\big)\big)\rd x\nonumber\\
   &\le M_2 + bK_{t,t_i} + 2\int_{x=x_0}^{K_{t,t_i}/2}  F_{L'}^{(-1)}\big(\exp\big(-c_5 (x+1)^{-2}(\tau-2)^{-i-x}\big)\big)\rd x\label{eq:claim-upper1}
  \end{align}
  for some large constants $x_0,M_2>0$. Apply the change of variables
  \[
   (k+1)^{-2}(\tau-2)^{-x} = (\tau-2)^{-y/2} \qquad
   \Leftrightarrow
   \qquad
   y = 2x - 4 \log(x+1)/|\log(\tau-2)|.
  \]
  Differentiating both sides, rearranging terms, and using $y=\Theta(x)$   yields a bound for $\rd x$, i.e.,
  \[
  \Big(1 - \frac{2/|\log(\tau-2)|}{x+1}\Big)\rd x = \frac{1}{2}\rd y
  \qquad \Rightarrow \qquad
  \rd x  \le \frac{1}{2}\bigg(1+\frac{C}{y+1}\bigg)\rd y,
  \]
  for some constant $C>0$ and $x\ge x_0$ sufficiently large.
  Continuing to bound \eqref{eq:claim-upper1} from above, we obtain if $x_0$ is sufficiently large
  \begin{align}
    \widetilde Q_{t,t_i}
     & \le
    M_2 + bK_{t,t_i} + \int_{y=2x_0-4\frac{\log(x_0+1)}{|\log(\tau-2)|}}^{K_{t,t_i} - 4\frac{\log(1+K_{t,t_i})}{|\log(\tau-2)|}} \bigg(1+\frac{C}{y+1}\bigg) F_{L'}^{(-1)}\big(\exp\big(-c_5 (\tau-2)^{-i-y/2}\big)\big)\rd y\nonumber\\
    &\le
    M_2 + bK_{t,t_i} + \int_{y=x_0}^{K_{t,t_i}} \bigg(1+\frac{C}{y+1}\bigg) F_{L'}^{(-1)}\big(\exp\big(-c_5 (\tau-2)^{-i-y/2}\big)\big)\rd y.\label{eq:claim-upper2}
  \end{align}
  Recall $\bm{I}_2(L')$ from \eqref{eq:exp-crit}.
  We first assume $\bm{I}_2(L')<\infty$. In this case, there exists $M_3>0$ such that
  \begin{align}
   \int_{y=x_0}^{\infty}\frac{C}{y+1}  F_{L'}^{(-1)}\big(\exp\big(-(\tau-2)^{-y/2}\big)\big)\rd y < M_3.\nonumber
  \end{align}
  Using that the integrand in \eqref{eq:claim-upper2} is bounded, we obtain for some $M_4>0$
  \begin{align}
    \widetilde Q_{t,t_i}
     & \le
    M_2 + M_3 + bK_{t,t_i} + \int_{y=x_0}^{K_{t,t_i} } F_{L'}^{(-1)}\big(\exp\big(-c_5 (\tau-2)^{-i-y/2}\big)\big)\rd y\nonumber\\
    &\le M_4 + bK_{t,t_i} + \int_{y=0}^{K_{t,t_i}} F_{L'}^{(-1)}\big(\exp\big(-c_5 (\tau-2)^{-i-y/2}\big)\big)\rd y.   \label{eq:claim-upper3}
  \end{align}
  Since $L'=L-b$ by definition, and using that the integration interval has length $K_{t,t_i}$, we obtain
  \begin{align*}
    \widetilde Q_{t,t_i}
    &\le M_4 + \int_{y=0}^{K_{t,t_i}} F_{L}^{(-1)}\big(\exp\big(-c_5 (\tau-2)^{-(y+2i)/2}\big)\big)\rd y \\
    &= M_4 + \int_{y=2i}^{K_{t,t_i}+2i} F_{L}^{(-1)}\big(\exp\big(-c_5 (\tau-2)^{-y/2}\big)\big)\rd y
  \end{align*}
  by shifting the integration boundaries. Recall $K_{t,t}-K_{t,t_i}=2i$ by \eqref{eq:kti}, yielding
  \begin{align*}
    \widetilde Q_{t,t_i}
    &\le
    M_4 + \int_{y=K_{t,t}-K_{t,t_i}}^{K_{t,t}} F_{L}^{(-1)}\big(\exp\big(-c_5 (\tau-2)^{-y/2}\big)\big)\rd y.
  \end{align*}
  We leave it to the reader to verify using another change of variables and $(\ast)$ in \eqref{eq:integral-sum-tricks} that, similarly to the proof for the lower bound after \eqref{eq:integral-sum-tricks}, there exists $M_5$ such that
  \begin{align*}
    \int_{y=K_{t,t}-K_{t,t_i}}^{K_{t,t}} &F_{L}^{(-1)}\big(\exp\big(-c_5 (\tau-2)^{-y/2}\big)\big)\rd y \\
    &-
    \sum_{k=K_{t,t}-K_{t,t_i}+1}^{K_{t,t}} F_{L}^{(-1)}\big(\exp\big(-(\tau-2)^{-k/2}\big)\big)\rd y
    \, \le\, M_5.
  \end{align*}
  This establishes \eqref{eq:claim-upper} when $\bm{I}_2(L)<\infty$. If $\bm{I}_2(L)=\infty$, we observe that there exists $M_6$ such that
  \[
    \int_{y=x_0}^{K_{t,t_i}}\frac{C}{y+1}  F_{L'}^{(-1)}\big(\exp\big(-c_5(\tau-2)^{-y/2}\big)\big)\rd y
    <
    M_6 + \varepsilon \int_{y=x_0}^{K_{t,t_i}}F_{L'}^{(-1)}\big(\exp\big(-c_5(\tau-2)^{-y/2}\big)\big)\rd y.
  \]
  We use this bound in \eqref{eq:claim-upper2}, bound $b \le (1+\varepsilon)b$, and follow the same steps as from \eqref{eq:claim-upper3} onwards, carrying a factor $(1+\varepsilon)$ for the integrals.
 \end{proof}
\end{claim}

\bibliographystyle{imsart}

\begin{thebibliography}{46}

\bibitem{adriaans2017weighted}
\begin{barticle}[author]
\bauthor{\bsnm{Adriaans},~\bfnm{Erwin}\binits{E.}} \AND
  \bauthor{\bsnm{Komj{\'{a}}thy},~\bfnm{J{\'{u}}lia}\binits{J.}}
(\byear{2018}).
\btitle{{Weighted Distances in Scale-Free Configuration Models}}.
\bjournal{Journal of Statistical Physics}
\bvolume{173}
\bpages{1082--1109}.
\bdoi{10.1007/s10955-018-1957-5}
\end{barticle}
\endbibitem

\bibitem{aiello2008spatial}
\begin{barticle}[author]
\bauthor{\bsnm{Aiello},~\bfnm{William}\binits{W.}},
  \bauthor{\bsnm{Bonato},~\bfnm{Anthony}\binits{A.}},
  \bauthor{\bsnm{Cooper},~\bfnm{Colin}\binits{C.}},
  \bauthor{\bsnm{Janssen},~\bfnm{Jeanette}\binits{J.}} \AND
  \bauthor{\bsnm{Pra{\l}at},~\bfnm{Pawe{\l}}\binits{P.}}
(\byear{2008}).
\btitle{A spatial web graph model with local influence regions}.
\bjournal{Internet Mathematics}
\bvolume{5}
\bpages{175--196}.
\end{barticle}
\endbibitem

\bibitem{alves2019preferentialedgestep}
\begin{barticle}[author]
\bauthor{\bsnm{Alves},~\bfnm{Caio}\binits{C.}},
  \bauthor{\bsnm{Ribeiro},~\bfnm{Rodrigo}\binits{R.}} \AND
  \bauthor{\bsnm{Sanchis},~\bfnm{Remy}\binits{R.}}
(\byear{2019}).
\btitle{Preferential Attachment Random Graphs with Edge-Step Functions}.
\bjournal{Journal of Theoretical Probability}
\bpages{1--39}.
\end{barticle}
\endbibitem

\bibitem{amin2018prefBootstrap}
\begin{barticle}[author]
\bauthor{\bsnm{Amin~Abdullah},~\bfnm{Mohammed}\binits{M.}} \AND
  \bauthor{\bsnm{Fountoulakis},~\bfnm{Nikolaos}\binits{N.}}
(\byear{2018}).
\btitle{A phase transition in the evolution of bootstrap percolation processes
  on preferential attachment graphs}.
\bjournal{Random Structures \& Algorithms}
\bvolume{52}
\bpages{379--418}.
\end{barticle}
\endbibitem

\bibitem{auffinger201750}
\begin{bbook}[author]
\bauthor{\bsnm{Auffinger},~\bfnm{Antonio}\binits{A.}},
  \bauthor{\bsnm{Damron},~\bfnm{Michael}\binits{M.}} \AND
  \bauthor{\bsnm{Hanson},~\bfnm{Jack}\binits{J.}}
(\byear{2017}).
\btitle{50 years of first-passage percolation}
\bvolume{68}.
\bpublisher{American Mathematical Soc.}
\end{bbook}
\endbibitem

\bibitem{baroni2017nonuniversality}
\begin{barticle}[author]
\bauthor{\bsnm{Baroni},~\bfnm{Enrico}\binits{E.}},
  \bauthor{\bsnm{Hofstad},~\bfnm{Remco van~der}\binits{R.~v.~d.}} \AND
  \bauthor{\bsnm{Komj{\'a}thy},~\bfnm{J{\'u}lia}\binits{J.}}
(\byear{2017}).
\btitle{Nonuniversality of weighted random graphs with infinite variance
  degree}.
\bjournal{Journal of Applied Probability}
\bvolume{54}
\bpages{146--164}.
\end{barticle}
\endbibitem

\bibitem{Baroni2019}
\begin{barticle}[author]
\bauthor{\bsnm{Baroni},~\bfnm{Enrico}\binits{E.}},
  \bauthor{\bsnm{Hofstad},~\bfnm{Remco van~der}\binits{R.~v.~d.}} \AND
  \bauthor{\bsnm{Komj{\'{a}}thy},~\bfnm{J{\'{u}}lia}\binits{J.}}
(\byear{2019}).
\btitle{Tight Fluctuations of Weight-Distances in Random Graphs with
  Infinite-Variance Degrees}.
\bjournal{Journal of Statistical Physics}
\bvolume{174}
\bpages{906--934}.
\bdoi{10.1007/s10955-018-2213-8}
\end{barticle}
\endbibitem

\bibitem{berger2005spread}
\begin{binproceedings}[author]
\bauthor{\bsnm{Berger},~\bfnm{Noam}\binits{N.}},
  \bauthor{\bsnm{Borgs},~\bfnm{Christian}\binits{C.}},
  \bauthor{\bsnm{Chayes},~\bfnm{Jennifer~T}\binits{J.~T.}} \AND
  \bauthor{\bsnm{Saberi},~\bfnm{Amin}\binits{A.}}
(\byear{2005}).
\btitle{On the spread of viruses on the internet}.
In \bbooktitle{Proceedings of the sixteenth annual ACM-SIAM symposium on
  Discrete algorithms}
\bpages{301--310}.
\bpublisher{Society for Industrial and Applied Mathematics}.
\end{binproceedings}
\endbibitem

\bibitem{berger2014asymptotic}
\begin{barticle}[author]
\bauthor{\bsnm{Berger},~\bfnm{Noam}\binits{N.}},
  \bauthor{\bsnm{Borgs},~\bfnm{Christian}\binits{C.}},
  \bauthor{\bsnm{Chayes},~\bfnm{Jennifer~T}\binits{J.~T.}} \AND
  \bauthor{\bsnm{Saberi},~\bfnm{Amin}\binits{A.}}
(\byear{2014}).
\btitle{Asymptotic behavior and distributional limits of preferential
  attachment graphs}.
\bjournal{The Annals of Probability}
\bvolume{42}
\bpages{1--40}.
\end{barticle}
\endbibitem

\bibitem{bhamidi2010first}
\begin{barticle}[author]
\bauthor{\bsnm{Bhamidi},~\bfnm{Shankar}\binits{S.}},
  \bauthor{\bsnm{Hofstad},~\bfnm{Remco van~der}\binits{R.~v.~d.}} \AND
  \bauthor{\bsnm{Hooghiemstra},~\bfnm{Gerard}\binits{G.}}
(\byear{2010}).
\btitle{First passage percolation on random graphs with finite mean degrees}.
\bjournal{The Annals of Applied Probability}
\bvolume{20}
\bpages{1907--1965}.
\end{barticle}
\endbibitem

\bibitem{bianconi2001bose}
\begin{barticle}[author]
\bauthor{\bsnm{Bianconi},~\bfnm{Ginestra}\binits{G.}} \AND
  \bauthor{\bsnm{Barab\'asi},~\bfnm{Albert-L\'aszl\'o}\binits{A.-L.}}
(\byear{2001}).
\btitle{Bose-{E}instein Condensation in Complex Networks}.
\bjournal{Phys. Rev. Lett.}
\bvolume{86}
\bpages{5632--5635}.
\bdoi{10.1103/PhysRevLett.86.5632}
\end{barticle}
\endbibitem

\bibitem{bollobas1980countingregulargraphs}
\begin{barticle}[author]
\bauthor{\bsnm{Bollob\'{a}s},~\bfnm{B\'{e}la}\binits{B.}}
(\byear{1980}).
\btitle{A probabilistic proof of an asymptotic formula for the number of
  labelled regular graphs}.
\bjournal{European J. Combin.}
\bvolume{1}
\bpages{311--316}.
\bdoi{10.1016/S0195-6698(80)80030-8}
\bmrnumber{595929}
\end{barticle}
\endbibitem

\bibitem{bollobas2001random}
\begin{bbook}[author]
\bauthor{\bsnm{Bollob\'{a}s},~\bfnm{B\'{e}la}\binits{B.}}
(\byear{2001}).
\btitle{Random graphs},
\bedition{second} ed.
\bseries{Cambridge Studies in Advanced Mathematics}
\bvolume{73}.
\bpublisher{Cambridge {U}niversity {P}ress, Cambridge}.
\bdoi{10.1017/CBO9780511814068}
\bmrnumber{1864966}
\end{bbook}
\endbibitem

\bibitem{bollobas2004diameterpa}
\begin{barticle}[author]
\bauthor{\bsnm{Bollob{\'a}s},~\bfnm{B{\'e}la}\binits{B.}} \AND
  \bauthor{\bsnm{Riordan},~\bfnm{Oliver}\binits{O.}}
(\byear{2004}).
\btitle{The Diameter of a Scale-Free Random Graph}.
\bjournal{Combinatorica}
\bvolume{24}
\bpages{5--34}.
\bdoi{10.1007/s00493-004-0002-2}
\end{barticle}
\endbibitem

\bibitem{can2015metastability}
\begin{barticle}[author]
\bauthor{\bsnm{Can},~\bfnm{Van~Hao}\binits{V.~H.}}
(\byear{2017}).
\btitle{Metastability for the contact process on the preferential attachment
  graph}.
\bjournal{Internet Mathematics}.
\bdoi{https://doi.org/10.24166/im.08.2017}
\end{barticle}
\endbibitem

\bibitem{caravenna2016diameterArxiv}
\begin{barticle}[author]
\bauthor{\bsnm{Caravenna},~\bfnm{Francesco}\binits{F.}},
  \bauthor{\bsnm{Garavaglia},~\bfnm{Alessandro}\binits{A.}} \AND
  \bauthor{\bsnm{Hofstad},~\bfnm{Remco van~der}\binits{R.~v.~d.}}
(\byear{2016}).
\btitle{Diameter in ultra-small scale-free random graphs: Extended version}.
\bjournal{Preprint arXiv:1605.02714}.
\end{barticle}
\endbibitem

\bibitem{caravenna2016diameter}
\begin{barticle}[author]
\bauthor{\bsnm{Caravenna},~\bfnm{Francesco}\binits{F.}},
  \bauthor{\bsnm{Garavaglia},~\bfnm{Alessandro}\binits{A.}} \AND
  \bauthor{\bsnm{Hofstad},~\bfnm{Remco van~der}\binits{R.~v.~d.}}
(\byear{2019}).
\btitle{{Diameter in ultra-small scale-free random graphs}}.
\bjournal{Random Structures {\&} Algorithms}
\bvolume{54}
\bpages{444--498}.
\bdoi{10.1002/rsa.20798}
\end{barticle}
\endbibitem

\bibitem{chung2002average}
\begin{barticle}[author]
\bauthor{\bsnm{Chung},~\bfnm{Fan}\binits{F.}} \AND
  \bauthor{\bsnm{Lu},~\bfnm{Linyuan}\binits{L.}}
(\byear{2002}).
\btitle{The average distances in random graphs with given expected degrees}.
\bjournal{Proceedings of the National Academy of Sciences}
\bvolume{99}
\bpages{15879--15882}.
\end{barticle}
\endbibitem

\bibitem{cipriani2019dynamical}
\begin{barticle}[author]
\bauthor{\bsnm{Cipriani},~\bfnm{Alessandra}\binits{A.}} \AND
  \bauthor{\bsnm{Fontanari},~\bfnm{Andrea}\binits{A.}}
(\byear{2019}).
\btitle{Dynamical fitness models: evidence of universality classes for
  preferential attachment graphs}.
\bjournal{Preprint arXiv:1911.12402}.
\end{barticle}
\endbibitem

\bibitem{cooper2004random}
\begin{barticle}[author]
\bauthor{\bsnm{Cooper},~\bfnm{Colin}\binits{C.}},
  \bauthor{\bsnm{Frieze},~\bfnm{Alan}\binits{A.}} \AND
  \bauthor{\bsnm{Vera},~\bfnm{Juan}\binits{J.}}
(\byear{2004}).
\btitle{Random deletion in a scale-free random graph process}.
\bjournal{Internet Mathematics}
\bvolume{1}
\bpages{463--483}.
\end{barticle}
\endbibitem

\bibitem{deijfen2009growing}
\begin{barticle}[author]
\bauthor{\bsnm{Deijfen},~\bfnm{Maria}\binits{M.}} \AND
  \bauthor{\bsnm{Lindholm},~\bfnm{Mathias}\binits{M.}}
(\byear{2009}).
\btitle{Growing networks with preferential deletion and addition of edges}.
\bjournal{Physica A: Statistical Mechanics and its Applications}
\bvolume{388}
\bpages{4297--4303}.
\end{barticle}
\endbibitem

\bibitem{dereichMaillerMortersCondens2017}
\begin{barticle}[author]
\bauthor{\bsnm{Dereich},~\bfnm{Steffen}\binits{S.}},
  \bauthor{\bsnm{Mailler},~\bfnm{C\'{e}cile}\binits{C.}} \AND
  \bauthor{\bsnm{M\"{o}rters},~\bfnm{Peter}\binits{P.}}
(\byear{2017}).
\btitle{Nonextensive condensation in reinforced branching processes}.
\bjournal{Ann. Appl. Probab.}
\bvolume{27}
\bpages{2539--2568}.
\bdoi{10.1214/16-AAP1268}
\bmrnumber{3693533}
\end{barticle}
\endbibitem

\bibitem{dereich2012typical}
\begin{barticle}[author]
\bauthor{\bsnm{Dereich},~\bfnm{Steffen}\binits{S.}},
  \bauthor{\bsnm{M{\"o}nch},~\bfnm{Christian}\binits{C.}} \AND
  \bauthor{\bsnm{M{\"o}rters},~\bfnm{Peter}\binits{P.}}
(\byear{2012}).
\btitle{Typical distances in ultrasmall random networks}.
\bjournal{Advances in Applied Probability}
\bvolume{44}
\bpages{583--601}.
\end{barticle}
\endbibitem

\bibitem{dereich2009random}
\begin{barticle}[author]
\bauthor{\bsnm{Dereich},~\bfnm{Steffen}\binits{S.}} \AND
  \bauthor{\bsnm{M{\"o}rters},~\bfnm{Peter}\binits{P.}}
(\byear{2009}).
\btitle{Random networks with sublinear preferential attachment: degree
  evolutions}.
\bjournal{Electronic Journal of Probability}
\bvolume{14}
\bpages{1222--1267}.
\end{barticle}
\endbibitem

\bibitem{dereich2013random}
\begin{barticle}[author]
\bauthor{\bsnm{Dereich},~\bfnm{Steffen}\binits{S.}} \AND
  \bauthor{\bsnm{M{\"o}rters},~\bfnm{Peter}\binits{P.}}
(\byear{2013}).
\btitle{Random networks with sublinear preferential attachment: the giant
  component}.
\bjournal{The Annals of Probability}
\bvolume{41}
\bpages{329--384}.
\end{barticle}
\endbibitem

\bibitem{dereichOrtgieseCondens2014}
\begin{barticle}[author]
\bauthor{\bsnm{Dereich},~\bfnm{Steffen}\binits{S.}} \AND
  \bauthor{\bsnm{Ortgiese},~\bfnm{Marcel}\binits{M.}}
(\byear{2014}).
\btitle{Robust analysis of preferential attachment models with fitness}.
\bjournal{Combin. Probab. Comput.}
\bvolume{23}
\bpages{386--411}.
\bdoi{10.1017/S0963548314000157}
\bmrnumber{3189418}
\end{barticle}
\endbibitem

\bibitem{dommers2010diameters}
\begin{barticle}[author]
\bauthor{\bsnm{Dommers},~\bfnm{Sander}\binits{S.}},
  \bauthor{\bsnm{Hofstad},~\bfnm{Remco van~der}\binits{R.~v.~d.}} \AND
  \bauthor{\bsnm{Hooghiemstra},~\bfnm{Gerard}\binits{G.}}
(\byear{2010}).
\btitle{Diameters in preferential attachment models}.
\bjournal{Journal of Statistical Physics}
\bvolume{139}
\bpages{72--107}.
\end{barticle}
\endbibitem

\bibitem{eckhoff2014vulnerability}
\begin{barticle}[author]
\bauthor{\bsnm{Eckhoff},~\bfnm{Maren}\binits{M.}} \AND
  \bauthor{\bsnm{M{\"o}rters},~\bfnm{Peter}\binits{P.}}
(\byear{2014}).
\btitle{Vulnerability of robust preferential attachment networks}.
\bjournal{Electronic Journal of Probability}
\bvolume{19}.
\end{barticle}
\endbibitem

\bibitem{faloutsos1999internet}
\begin{barticle}[author]
\bauthor{\bsnm{Faloutsos},~\bfnm{Michalis}\binits{M.}},
  \bauthor{\bsnm{Faloutsos},~\bfnm{Petros}\binits{P.}} \AND
  \bauthor{\bsnm{Faloutsos},~\bfnm{Christos}\binits{C.}}
(\byear{1999}).
\btitle{On Power-Law Relationships of the Internet Topology}.
\bjournal{SIGCOMM Comput. Commun. Rev.}
\bvolume{29}
\bpages{251–262}.
\bdoi{10.1145/316194.316229}
\end{barticle}
\endbibitem

\bibitem{freeman2018extensive}
\begin{barticle}[author]
\bauthor{\bsnm{Freeman},~\bfnm{Nic}\binits{N.}},
  \bauthor{\bsnm{Jordan},~\bfnm{Jonathan}\binits{J.}} \betal{et~al.}
(\byear{2020}).
\btitle{Extensive condensation in a model of preferential attachment with
  fitness}.
\bjournal{Electronic Journal of Probability}
\bvolume{25}.
\end{barticle}
\endbibitem

\bibitem{gracar2018age}
\begin{barticle}[author]
\bauthor{\bsnm{Gracar},~\bfnm{Peter}\binits{P.}},
  \bauthor{\bsnm{Grauer},~\bfnm{Arne}\binits{A.}},
  \bauthor{\bsnm{L\"{u}chtrath},~\bfnm{Lukas}\binits{L.}} \AND
  \bauthor{\bsnm{M\"{o}rters},~\bfnm{Peter}\binits{P.}}
(\byear{2019}).
\btitle{The age-dependent random connection model}.
\bjournal{Queueing Syst.}
\bvolume{93}
\bpages{309--331}.
\bdoi{10.1007/s11134-019-09625-y}
\bmrnumber{4032928}
\end{barticle}
\endbibitem

\bibitem{gracarluchtrath2020robustness}
\begin{barticle}[author]
\bauthor{\bsnm{Gracar},~\bfnm{Pete}\binits{P.}},
  \bauthor{\bsnm{L\"uchtrath},~\bfnm{Lukas}\binits{L.}} \AND
  \bauthor{\bsnm{M\"{o}rters},~\bfnm{Peter}\binits{P.}}
(\byear{2020}).
\btitle{Percolation phase transition in weight-dependent random connection
  models}.
\bjournal{Preprint arXiv:2003.04040}.
\end{barticle}
\endbibitem

\bibitem{hirsch2018distances_spatial_pa}
\begin{barticle}[author]
\bauthor{\bsnm{Hirsch},~\bfnm{Christian}\binits{C.}} \AND
  \bauthor{\bsnm{M{\"o}nch},~\bfnm{Christian}\binits{C.}}
(\byear{2020}).
\btitle{Distances and large deviations in the spatial preferential attachment
  model}.
\bjournal{Bernoulli}
\bvolume{26}
\bpages{927--947}.
\end{barticle}
\endbibitem

\bibitem{hofstad2016book1}
\begin{bbook}[author]
\bauthor{\bsnm{Hofstad},~\bfnm{Remco van~der}\binits{R.~v.~d.}}
(\byear{2016}).
\btitle{Random graphs and complex networks, {V}olume 1}.
\bpublisher{Cambridge {U}niversity {P}ress}.
\end{bbook}
\endbibitem

\bibitem{hofstad2017stochasticprocesses}
\begin{bmisc}[author]
\bauthor{\bsnm{Hofstad},~\bfnm{Remco van~der}\binits{R.~v.~d.}}
(\byear{2017}).
\btitle{Stochastic processes on random graphs}.
\bnote{\url{https://www.win.tue.nl/~rhofstad/SaintFlour_SPoRG.pdf}}.
\end{bmisc}
\endbibitem

\bibitem{hofstad2016book2}
\begin{bmisc}[author]
\bauthor{\bsnm{Hofstad},~\bfnm{Remco van~der}\binits{R.~v.~d.}}
(\byear{2020}+).
\btitle{Random graphs and complex networks, {V}olume 2}.
\bnote{\url{https://www.win.tue.nl/~rhofstad/NotesRGCNII.pdf}}.
\end{bmisc}
\endbibitem

\bibitem{jacob2015spatial}
\begin{barticle}[author]
\bauthor{\bsnm{Jacob},~\bfnm{Emmanuel}\binits{E.}} \AND
  \bauthor{\bsnm{M{\"o}rters},~\bfnm{Peter}\binits{P.}}
(\byear{2015}).
\btitle{Spatial preferential attachment networks: Power laws and clustering
  coefficients}.
\bjournal{The Annals of Applied Probability}
\bvolume{25}
\bpages{632--662}.
\end{barticle}
\endbibitem

\bibitem{jacob2017robustness}
\begin{barticle}[author]
\bauthor{\bsnm{Jacob},~\bfnm{Emmanuel}\binits{E.}} \AND
  \bauthor{\bsnm{M{\"o}rters},~\bfnm{Peter}\binits{P.}}
(\byear{2017}).
\btitle{Robustness of scale-free spatial networks}.
\bjournal{The Annals of Probability}
\bvolume{45}
\bpages{1680--1722}.
\end{barticle}
\endbibitem

\bibitem{janson2019preferential}
\begin{barticle}[author]
\bauthor{\bsnm{Janson},~\bfnm{Svante}\binits{S.}} \AND
  \bauthor{\bsnm{Warnke},~\bfnm{Lutz}\binits{L.}}
(\byear{2021}).
\btitle{Preferential attachment without vertex growth: Emergence of the giant
  component}.
\bjournal{The Annals of Applied Probability}
\bvolume{31}
\bpages{1523--1547}.
\end{barticle}
\endbibitem

\bibitem{jorkom2019weighted}
\begin{barticle}[author]
\bauthor{\bsnm{Jorritsma},~\bfnm{Joost}\binits{J.}} \AND
  \bauthor{\bsnm{Komjáthy},~\bfnm{Júlia}\binits{J.}}
(\byear{2020}).
\btitle{Weighted distances in scale-free preferential attachment models}.
\bjournal{Random Structures \& Algorithms}
\bvolume{57}
\bpages{823-859}.
\bdoi{https://doi.org/10.1002/rsa.20947}
\end{barticle}
\endbibitem

\bibitem{malyshkin2014powerofchoic}
\begin{barticle}[author]
\bauthor{\bsnm{Malyshkin},~\bfnm{Yury}\binits{Y.}} \AND
  \bauthor{\bsnm{Paquette},~\bfnm{Elliot}\binits{E.}}
(\byear{2014}).
\btitle{The power of choice combined with preferential attachment}.
\bjournal{Electron. Commun. Probab.}
\bvolume{19}
\bpages{1--13}.
\bdoi{10.1214/ECP.v19-3461}
\bmrnumber{3233206}
\end{barticle}
\endbibitem

\bibitem{molloyreed1995randomgraphswithgivendegree}
\begin{binproceedings}[author]
\bauthor{\bsnm{Molloy},~\bfnm{Michael}\binits{M.}} \AND
  \bauthor{\bsnm{Reed},~\bfnm{Bruce}\binits{B.}}
(\byear{1995}).
\btitle{A critical point for random graphs with a given degree sequence}.
In \bbooktitle{Proceedings of the {S}ixth {I}nternational {S}eminar on {R}andom
  {G}raphs and {P}robabilistic {M}ethods in {C}ombinatorics and {C}omputer
  {S}cience, ``{R}andom {G}raphs '93'' ({P}ozna\'{n}, 1993)}
\bvolume{6}
\bpages{161--179}.
\bdoi{10.1002/rsa.3240060204}
\bmrnumber{1370952}
\end{binproceedings}
\endbibitem

\bibitem{monch2013distances}
\begin{bphdthesis}[author]
\bauthor{\bsnm{M{\"o}nch},~\bfnm{Christian}\binits{C.}}
(\byear{2013}).
\btitle{Distances in preferential attachment networks},
\btype{PhD thesis},
\bpublisher{University of Bath}.
\end{bphdthesis}
\endbibitem

\bibitem{mori2002}
\begin{barticle}[author]
\bauthor{\bsnm{M\'{o}ri},~\bfnm{T.~F.}\binits{T.~F.}}
(\byear{2002}).
\btitle{On random trees}.
\bjournal{Studia Sci. Math. Hungar.}
\bvolume{39}
\bpages{143--155}.
\bdoi{10.1556/SScMath.39.2002.1-2.9}
\bmrnumber{1909153}
\end{barticle}
\endbibitem

\bibitem{norros2006conditionally}
\begin{barticle}[author]
\bauthor{\bsnm{Norros},~\bfnm{Ilkka}\binits{I.}} \AND
  \bauthor{\bsnm{Reittu},~\bfnm{Hannu}\binits{H.}}
(\byear{2006}).
\btitle{On a conditionally {P}oissonian graph process}.
\bjournal{Advances in Applied Probability}
\bvolume{38}
\bpages{59--75}.
\end{barticle}
\endbibitem

\bibitem{pittel2010random}
\begin{barticle}[author]
\bauthor{\bsnm{Pittel},~\bfnm{Boris}\binits{B.}}
(\byear{2010}).
\btitle{On a random graph evolving by degrees}.
\bjournal{Advances in Mathematics}
\bvolume{223}
\bpages{619--671}.
\end{barticle}
\endbibitem

\end{thebibliography}

\end{document}